\documentclass[11pt]{article}
\usepackage{amsfonts}
\usepackage{amsmath}
\usepackage{amssymb}
\usepackage{epsfig}
\usepackage{citesort}
\usepackage{url}
\usepackage{enumerate}
\title{Well-conditioned boundary integral equation formulations for the solution of high-frequency electromagnetic scattering problems}
\author{Yassine Boubendir, Catalin Turc\\ \small 
New Jersey Institute of Technology\\ \small
boubendi@njit.edu\ catalin.c.turc@njit.edu}

\newtheorem{theorem}{Theorem}[section]
\newtheorem{lemma}[theorem]{Lemma}

\newtheorem{corollary}[theorem]{Corollary}
\newtheorem{remark}[theorem]{Remark}
\newenvironment{proof}{\hspace{0.5cm} {\bf Proof.}}
{$\quad {}_\blacksquare$\vspace{0.3cm}}

\begin{document}
\date{}
\maketitle
\begin{abstract}
We present several versions of Regularized Combined Field Integral Equation (CFIER) formulations for the solution of three dimensional frequency domain electromagnetic scattering problems with Perfectly Electric Conducting (PEC) boundary conditions. Just as in the Combined Field Integral Equations (CFIE), we seek the scattered fields in the form of a combined magnetic and electric dipole layer potentials that involves a composition of the latter type of boundary layers with regularizing operators. The regularizing operators are of two types: (1) modified versions of electric field integral operators with complex wavenumbers, and (2) principal symbols of those operators in the sense of pseudodifferential operators. We show that the boundary integral operators that enter these CFIER formulations are Fredholm of the second kind, and invertible with bounded inverses in the classical trace spaces of electromagnetic scattering problems. We present a spectral analysis of CFIER operators with regularizing operators that have purely imaginary wavenumbers for spherical geometries---we refer to these operators as Calder\'on-Ikawa CFIER. Under certain assumptions on the coupling constants and the absolute values of the imaginary wavenumbers of the regularizing operators, we show that the ensuing  Calder\'on-Ikawa CFIER operators are coercive for spherical geometries. These properties allow us to derive wavenumber explicit bounds on the condition numbers of Calder\'on-Ikawa CFIER operators. When regularizing operators with complex wavenumbers with non-zero real parts are used---we refer to these operators as Calder\'on-Complex CFIER, we show numerical evidence that those complex wavenumbers can be selected in a manner that leads to CFIER formulations whose condition numbers can be bounded independently of frequency for spherical geometries. In addition, the Calder\'on-Complex CFIER operators possess excellent spectral properties in the high-frequency regime for both convex and non-convex scatterers. We provide numerical evidence that our solvers based on fast, high-order Nystr\"om discretization of these equations converge in very small numbers of GMRES iterations, and the iteration counts are virtually independent of frequency for several smooth scatterers with slowly varying curvatures.\\
\indent $\mathbf{Keywords}$: electromagnetic scattering, Combined Field Integral Equations, pseudodifferential operators, high-frequency.    
\end{abstract}

\section{Introduction\label{intro}}

\parskip 2pt plus2pt minus1pt

The simulation of frequency domain electromagnetic wave scattering gives rise to a host of computational
challenges that mostly result from oscillatory
solutions, and ill-conditioning
in the low and high-frequency regimes. Computational modeling of electromagnetic scattering has been attempted based on the classical Finite-Difference Time-Domain (FDTD) methods. However, algorithms based on the finite-difference or finite-element
discretizations require discretization of unoccupied volumetric
regions and give rise to numerical dispersion which is inevitably
associated with numerical propagation of waves across large numbers of
volumetric elements~\cite{babuska}. An important computational alternative to finite-difference and finite-element approaches is found in boundary integral
methods. Numerical methods based on integral formulations of
scattering problems enjoy a number of attractive properties as they
formulate the problems on lower-dimensional, bounded computational
domains and capture intrinsically the outgoing character of
scattered waves. Thus, on account of the dimensional reduction and
associated small discretizations (significantly smaller than the
discretizations required by volumetric finite-element or
finite-difference approximations), in conjunction with 
available {\em fast}
solvers~\cite{Bleszynski,Bojarski,BrunoKun,Catedra,CoifmanRokhlin,rokhlin,rokhlinetal,song-chu,PhilipsWhite,turc1,turc7},
numerical algorithms based on integral formulations, when applicable,
can outperform their finite-element/difference counterparts. On account of this, and whenever possible, the simulation of high-frequency scattering problems relies almost exclusively on boundary integral equations based solvers. There has been significant recent progress on extending the range of high-frequency solvers that can be solved by boundary integral equation solvers, mostly in the case of scalar problems with Dirichlet boundary conditions. This was made possible by hybrid methods that incorporate the known oscillatory behavior of solutions of boundary integral equations at high frequencies in order to reduce drastically the number of unknowns---see the excellent review paper~\cite{Chandler} for a full account of these methods. 

While well-conditioned integral formulations for scalar problems with Dirichlet boundary conditions have been known and used for quite some time, that is not the case for electromagnetic problems. The scope of this paper is to address the question: what integral equations should one use for the efficient simulation of high-frequency frequency-domain electromagnetic scattering problems. The most widely used integral equation formulations for solution of frequency domain scattering problems from perfectly electric conducting (PEC) closed three-dimensional objects are the Combined Field Integral Equations (CFIE) formulations~\cite{HarringtonMautz}. The CFIE are uniquely solvable throughout the frequency spectrum, yet the spectral properties of the boundary integral operators associated with the CFIE formulations are not particularly suited for Krylov-subspace iterative solvers such as GMRES~\cite{turc1,Contopa_et_al}. This is attributed to the fact that the electric field (EFIE) operator, which is a portion
of the CFIE, is a pseudodifferential operator of order $1$~\cite{MTaylor,Seeley}---that is, asymptotically, the action of the operator in Fourier space  amounts to multiplication by the Fourier-transform variable. Consistent with this fact, the eigenvalues
of these operators accumulate at infinity, which causes the condition numbers of CFIE formulations to grow with the discretization size,  a property that is shared by integral equations of the first kind. The lack of well conditioning of the operators in CFIE is exacerbated at high frequencies, a regime where CFIE require efficient preconditioners that should ideally control the amount of numerical work entailed by iterative solvers. In this regard, one possibility is to use algebraic preconditioners, typically based on multi-grid methods \cite{Carpentieri}, or Frobenius norm minimizations and sparsification techniques \cite{CarpentieriDuffGiraudSylvand}. However, the generic algebraic preconditioning strategies are not particularly geared towards wave scattering problems, and, in addition, they may encounter convergence breakdowns at higher frequencies that require large discretizations \cite{LeeZhang,Volakis}. 

On the other hand, several alternative integral equation formulations for PEC scattering problems that possess good conditioning properties have been introduced in the literature in the past fifteen years~\cite{BorelLevadouxAlouges,AlougesLevadoux,Andriulli,Darbas,Contopa_et_al,turc1,Adams,ChristiansenNedelec,Epstein,Taskinen}. Some of these formulations were devised to avoid the well-known``low-frequency breakdown''~\cite{Zhao,Epstein}. For instance, the current and charge integral equation formulation~\cite{Taskinen}, although not Fredholm of the second kind, does not suffer from the low-frequency breakdown and has reasonable properties throughout the frequency range~\cite{Bendali1}. Another class of Fredholm boundary integral equations of the second kind for the solution of PEC electromagnetic scattering problems can be derived using generalized Debye sources~\cite{Epstein}. Although these formulations targeted the low frequency case, their versions that use single layers with imaginary wavenumbers possess good condition numbers for higher frequencies for spherical scatterers~\cite{Vico}.

Another wide class of formulations that is directly related to the present work can be viewed as Regularized Integral Equations as they typically involve using pseudoinverses/regularizers of the electric field  integral operators to mollify the undesirable derivative-like effects of the latter operators. In the cases when the scattered electric fields are sought as linear combinations of magnetic and electric dipole distributions, the former acting on tangential densities while the latter acting on certain regularizing operators of the same tangential densities, the enforcement of the PEC boundary conditions leads to Regularized Combined Field Integral Equations (CFIER) or Generalized Combined Sources Integral Equations (GCSIE). In the case of smooth scatterers, the various regularizing operators proposed in the literature one the one hand (a) Stabilize the leading order effect of the pseudodifferential operators of order $1$ that enter CFIE, so that the integral operators in CFIER are compact perturbations of invertible diagonal matrix operators and (b) Have certain coercivity properties that ensure the invertibility of the CFIER operators. One way to construct regularizing operators that achieve the objective (a) can be pursued in the framework of approximations of admittance/Dirichlet-to-Neumann operators (that is the operators that map the values of the vector product between the unit normal and the electric field on the surface of the scatterer to the value of the vector product between the unit normal and the magnetic field on the surface of the scatterer---see Section~\ref{DtN})~\cite{Adams,BorelLevadouxAlouges,AlougesLevadoux,Darbas} which can be connected to on-surface radiation conditions (OSRC)~\cite{Kriegsman}. Another way to construct such operators is to start from Calder\'on's identities~\cite{Contopa_et_al,Andriulli,ChristiansenNedelec,turc1} that establish that the square of the electric field integral operator is a compact perturbation of the identity. All of these regularizing operators are either electric field integral operators, its vector single layer components, or their principal symbols in the sense of pseudodifferential operators. The regularizing operators that have property (a) can be modified to meet the requirement (b) either via quadratic partitions of unity~\cite{BorelLevadouxAlouges,AlougesLevadoux} or by means of {\em complexification} of the wavenumber in the definition of electric field integral operators or its components~\cite{Adams,Contopa_et_al,Darbas,turc1}. To the best of our knowledge, only the regularized formulations in~\cite{BorelLevadouxAlouges,AlougesLevadoux,turc1} are shown rigorously to be Fredholm integral equations of the second kind and invertible in the appropriate trace spaces of electromagnetic scattering of smooth scatterers. In the case of Lipschitz scatterers, the situation is more complicated. The well posedness of the classical CFIE formulations has not yet been proved for Lipschitz boundaries. The main difficulty stems from the lack of compactness of the magnetic field integral operators (sometimes referred to as electromagnetic double layer operators) in the case of Lipschitz boundaries. Alternative boundary integral equations~\cite{Hiptmair2,Steinbach} use regularizers that act precisely on the magnetic field integral operators and lead to formulations whose operators are compact perturbations of coercive operators. The latter property ensures the well posedness of the aforementioned regularized boundary integral equations in Lipschitz domains. 

The main contribution of this paper is the design, analysis, and fast, high-order implementation of a novel CFIER formulation which we refer to as Calder\'on-Complex CFIER. This formulation uses regularizing operators that are multiples of EFIE operators with {\em complex} wavenumbers with positive real and imaginary parts. We show that for smooth scatterers and under certain assumptions on the magnitude of the real and imaginary parts of the wavenumbers in the definition of the regularizing operators, the Calder\'on-Complex CFIER operators are Fredholm of index zero in appropriate functional spaces and that these operators are injective. While the Fredholm property of these novel operators can be established based on previous results in~\cite{turc7}, the proof of their injectivity requires a more involved analysis. We also show that in case of spherical scatterers, the complex wavenumbers in the definition of the regularizing operators can be selected so that to lead to Calder\'on-Complex CFIER formulations that are coercive and have condition numbers that can be bounded independently of the frequency for high-frequencies. This selection appropriately extended to scatterers with slowly varying curvatures leads to formulations with similar properties. 

In addition to the Calder\'on-Complex CFIER we extend the analysis developed in~\cite{turc1} to prove for the first time in the literature the well posedness of CFIER formulations whose regularizing operators are EFIE operators with {\em purely imaginary} wavenumbers. In order to be consistent with notations in the literature~\cite{Andriulli}, we refer to the latter operators as Calder\'on-Ikawa CFIER. These operators were originally introduced in~\cite{Contopa_et_al}, but no analysis was given in that reference. The Fredholm property of the Calder\'on-Ikawa CFIER operators is established based on techniques developed in~\cite{turc1}, and their injectivity is a consequence of certain well known positivity properties of single layer operators with purely imaginary wavenumbers. In contrast to the Calder\'on-Complex CFIER formulations, the Calder\'on-Ikawa CFIER formulations do not lead to iterations counts independent of the wavenumber. Indeed, we establish rigorously in this paper that the condition numbers of the Calder\'on-Ikawa CFIER  grow like $k^{2/3}$ as $k\to\infty$ in the case of spherical scatterers of radius one. A key ingredient in the proof of this result is the coercivity of the Calder\'on-Ikawa CFIER operators in the case of spherical geometries. We recall that if $H$ is a Hilbert space, an operator $\mathcal{A}:H\to H'$ (where $H'$ is the dual of $H$) is coercive if there exists a constant $\gamma>0$ such that $\gamma \|u\|_H^2\leq \Re \langle \mathcal{A}u,u\rangle$ for all $u\in H$, where $\langle\cdot,\cdot\rangle$ denotes the duality pairing between $H$ and $H'$. Our analysis uses and extends coercivity results introduced in~\cite{Graham,turc7}. In order to put things into perspective, we mention that in the case of unit spheres the condition numbers of the classical Brackhage-Werner CFIE formulations of Helmholtz equations with Dirichlet boundary conditions grow like $k^{1/3}$ as $k\to\infty$~\cite{Graham} and those of the Ikawa CFIER formulations of Helmholtz equations with Neumann boundary conditions also grow like $k^{1/3}$ as $k\to\infty$~\cite{turc7}.

Another important contribution in this paper is a rigorous analysis of CFIER formulations that use Fourier multiplier operators whose principal symbol (in the sense of pseudodifferential operators) coincides with those of the EFIE regularizers described above. Such formulations were originally introduced in~\cite{Darbas} without a rigorous analysis, and very recently implemented in~\cite{Boujaji,LevadouxN}. In our analysis, establishing the Fredholm property of the principal symbol CFIER relies on the pseudodifferential calculus, whereas their injectivity follows from the Helmholtz decomposition and simple algebraic manipulations.  

After discussing various CFIER formulations that require on {\em complexified} regularizing operators, we present extensive numerical results on the performance of solvers based on optimally designed Calder\'on-Complex CFIER, Calder\'on-Ikawa CFIER, and the classical CFIE formulations for smooth convex and non-convex scatterers ans a very wide range of high frequencies spanning from $10$ to $50$ wavelengths. Our numerical results are produced by our fast, high-order Nystr\"om discretization implementation of various boundary integral equation formulations considered in this text. The details of our Nystr\"om method were presented in our previous contribution~\cite{turc1}; we use overlapping partitions of unity and analytical resolution of kernel singularities to evaluate
accurately the integral operators. We use an accelerated version of our algorithms that generalizes to the electromagnetic case the FFT ``equivalent sources'' accelerated algorithms in~\cite{BrunoKun,turc3}. For scatterers with slowly varying curvatures, our solvers based on the novel Calder\'on-Complex CFIER formulations have the remarkable property that the number of Krylov subspace linear algebra solvers are virtually independent of frequency in the high-frequency range. This is in contrast with the behavior of our solvers based on Calder\'on-Ikawa CFIER and CFIE formulations. Although the cost of a matrix-vector product related to the Calder\'on-Complex CFIER formulations is on average $2.4$ times more expensive than that related to the classical CFIE formulations and $1.5$ times more expensive than that related to the Calder\'on-Ikawa CFIER formulations introduced in our previous effort~\cite{turc1}, the remarkably fast rate of convergence of solvers based on the former formulations garners important computational gains over the other two formulations. Specifically, in the high-frequency regime (e.g. problems of electromagnetic size of $50$ wavelengths) the computational gains of solvers based on the novel Calder\'on-Complex CFIER formulations over solvers based on classical CFIE formulations can be of factors $3.3$. The computational gains over solvers based on the Calder\'on-Ikawa CFIER~\cite{turc1} can be of factors $2.6$.

The paper is organized as follows: in Section~\ref{cfie} we introduce and analyze a wide class of Calder\'on CFIER formulations; in Section~\ref{pscfier} we introduce and analyze the principal symbol counterparts of the Calder\'on CFIER formulations, in Section~\ref{spec_prop_cfie} we analyze the spectral properties of the CFIER operators for spherical geometries; in Section~\ref{DtN} we present a methodology of selecting the complex wavenumbers in the definition of regularizing operators so that to lead to nearly optimal approximations of Dirichlet-to-Neumann operators in the case of spherical geometries; and in Section~\ref{num_exp} we present several numerical results that enable comparisons between the performance of our solvers based on fast, high-order Nystr\"om discretizations of various integral formulations of PEC scattering problems.

\section{Regularized Combined Field Integral Equations \label{cfie}}

We consider the problem of evaluating the scattered electromagnetic
field $(\mathbf E^{s},\mathbf H^{s})$ that results as an incident
field $(\mathbf E^{i},\mathbf H^{i})$ impinges upon the boundary
$\Gamma$ of a perfectly conducting scatterer $D$.  Defining the total
field by $(\mathbf E,\mathbf H) = (\mathbf E^{s}+\mathbf E^{i},\mathbf
H^{s}+\mathbf H^{i})$, the scattered field is determined uniquely by
the time-harmonic Maxwell equations for wavenumbers $k>0$
\begin{equation}
  \label{eq:Maxwell}
  {\rm curl}\ {\mathbf E}-ik {\mathbf H}=\mathbf{0},\qquad {\rm curl}\ {\mathbf H}+ik\mathbf{E}=\mathbf{0} \qquad \rm{in}\ \mathbb{R}^{3}\setminus D
\end{equation}
together with the perfect-conductor (PEC) boundary conditions
\begin{equation}
\label{eq:bc}
\mathbf{n}\times\mathbf{E}=\mathbf{0}\qquad \rm{on}\ \Gamma
\end{equation}
and the well known Silver-M\"uller radiation conditions at infinity on $(\mathbf
E^{s},\mathbf H^{s})$~\cite{co-kr}. Here and in what follows we assume that $\Gamma$ is smooth (actually $C^2$ would suffice) and $\mathbf{n}$ denote unit normals to the surface $\Gamma$ pointing into $\mathbb{R}^3\setminus D$. 

Integral equation formulations for electromagnetic scattering problems can be derived starting from magnetic and electric dipole distributions. Given a tangential density $\mathbf{m}$ on the scatterer $\Gamma$, the magnetic dipole distribution corresponding to the density $\mathbf{m}$ is defined as
\begin{equation}
\label{eq:MFIE_ind_scatt}
(\mathcal{M}\mathbf{m})(\mathbf{z})={\rm curl}\ \int_{\Gamma}G_k(\mathbf z-\mathbf {y})\mathbf{m}(\mathbf {y})d\sigma(\mathbf{y}),\ \mathbf{z}\in\mathbb{R}^3\setminus\Gamma
\end{equation}
and the electric dipole distribution corresponding to the tangential density $\mathbf{e}$ is defined as 
\begin{equation}
\label{eq:EFIE_ind_scatt}
(\mathcal{E}\mathbf{e})(\mathbf{z})={\rm curl}\ {\rm curl}\  \int_{\Gamma}G_k(\mathbf{z}-\mathbf{y})\mathbf{e}(\mathbf{y})d\sigma(\mathbf y),\ \mathbf{z}\in\mathbb{R}^3\setminus\Gamma,
\end{equation}
where $G_k$ is the outgoing fundamental solution attributed to the Helmholtz operator,
$G_k(\mathbf{z},\mathbf{y})=G_k(\mathbf{z}-\mathbf{y})=\frac{e^{ik|\mathbf{z}-\mathbf{y}|}}{4\pi|\mathbf{z}-\mathbf{y}|}$. Both $\mathcal{M}\mathbf{m}$ and $\mathcal{E}\mathbf{e}$ are radiative solutions of the electromagnetic wave equation ${\rm curl}\ {\rm curl}\ \mathbf{u}-k^2\mathbf{u}=0$ in $\mathbb{R}^3\setminus D$. The limits on $\Gamma$ of $\mathbf{n}\times \mathcal{M}\mathbf{m}$ and $\mathbf{n}\times \mathcal{E}\mathbf{e}$ can be expressed in the following form~\cite{HsiaoKleinman,Nedelec}:
\begin{eqnarray}\label{eq:traces}
\lim_{\epsilon\to 0+}\mathbf{n}(\mathbf{x})\times (\mathcal{M}\mathbf{m})(\mathbf{x}+\epsilon \mathbf{n}(\mathbf{x}))&=&\frac{\mathbf{m}(\mathbf{x})}{2}-(\mathcal{K}_k\mathbf{m})(\mathbf{x})\nonumber\\
\lim_{\epsilon\to 0+}\mathbf{n}(\mathbf{x})\times (\mathcal{E}\mathbf{e})(\mathbf{x}+\epsilon \mathbf{n}(\mathbf{x}))&=&(\mathcal{T}_k\mathbf{e})(\mathbf{x}),\quad \mathbf{x}\in\Gamma.
\end{eqnarray} 
In equations~\eqref{eq:traces} $\mathcal{K}_k$ and $\mathcal{T}_k$ denote the magnetic and respectively the electric field integral operators. These operators map tangential fields $\mathbf{a}$ on $\Gamma$ into tangential fields on $\Gamma$, and are defined as~\cite{HsiaoKleinman}
\begin{equation}
\label{eq:kernelMFIE}
(\mathcal K_k \mathbf{a})(\mathbf{x})=\mathbf n(\mathbf x)\times\int_{\Gamma}\nabla_{\mathbf{y}}G_{k}(\mathbf{x}-\mathbf{y})\times \mathbf{a}(\mathbf{y})d\sigma(\mathbf{y}),
\end{equation}
and
\begin{eqnarray}
\label{eq:kernelEFIE}
(\mathcal T_k \mathbf{a})(\mathbf{x})&=&ik\mathbf{n}(\mathbf{x})\times\int_{\Gamma}G_{k}(\mathbf{x}-\mathbf{y})\mathbf{a}(\mathbf{y})d\sigma(\mathbf{y})\nonumber\\
&+&\frac{i}{k}\mathbf{n}(\mathbf{x})\times\int_{\Gamma}\nabla_{\mathbf{x}}G_{k}(\mathbf{x}-\mathbf{y})\rm{div}_\Gamma \mathbf{a}(\mathbf{y})d\sigma(\mathbf{y})\nonumber\\
&=&(ik\ \mathbf{n}\times {\mathbf{S}}_{k}+\frac{i}{k}\ \mathcal{T}^{1}_k{\rm div}_\Gamma)\mathbf{a}(\mathbf{x}),\quad \mathbf{x}\in\Gamma.
\end{eqnarray}
The integrals in the definition of $\mathcal{K}_k$ and $\mathcal{T}_k^{1}{\rm div}_\Gamma$ should
be interpreted in the sense of Cauchy principal value integrals. 

Regularized Combined Field Integral Equations for the solution of electromagnetic scattering problems with perfectly electrically conducting boundary conditions seek representations of the electromagnetic fields of the form
\begin{eqnarray}
\label{eq:Representation}
\mathbf{E}^{s}(\mathbf {x}) &=& {\rm curl}\ \int_{\Gamma}G_k(\mathbf{x} - \mathbf {y})\mathbf{a}(\mathbf {y})d\sigma(\mathbf{y})\nonumber\\
&+&\frac{i}{k}\ {\rm curl}\ {\rm curl}\  \int_{\Gamma}G_k(\mathbf{x} - \mathbf{y})(\mathcal{R}\mathbf{a})(\mathbf{y})d\sigma(\mathbf y),\\
\label{eq:Representation2}\mathbf{H}^{s}(\mathbf{x}) &=& \frac{1}{ik}\ {\rm curl}\ \mathbf{E}(\mathbf{x}),\qquad \mathbf x \in \mathbb{R}^{3}\setminus D,
\end{eqnarray}
where $\mathcal{R}$ denotes a tangential operator to be specified in what follows. Owing to the boundary values in equations~\eqref{eq:traces}, 
equations~\eqref{eq:Representation} and~\eqref{eq:Representation2}
define outgoing solutions to the Maxwell equations with
perfect-conductor boundary conditions provided the tangential density
$\mathbf{a}$ is a solution to the integral equation
\begin{equation}
\label{eq:CFIE_R}
\frac{\mathbf{a}}{2}-\mathcal{K}_k\mathbf{a} +\mathcal{T}_k(\mathcal{R}\mathbf{a}) =-\mathbf{n}\times\mathbf{E}^{i}.
\end{equation}
We aim to design regularizing operators $\mathcal{R}$ such that the operators $\frac{I}{2}-\mathcal{K}_k +\mathcal{T}_k\mathcal{R}$ are sums of invertible diagonal matrix operators with constant entries (in the sense of the Helmholtz decomposition) and compact operators in the Sobolev spaces $H_{\rm
  div}^{m}(\Gamma)$ of $H^{m}(TM(\Gamma))$ tangential vector fields
that admit an $H^{m}$ divergence~\cite{HsiaoKleinman} for all $m\geq 2$. In this case the CFIER operators $\frac{I}{2}-\mathcal{K}_k +\mathcal{T}_k\mathcal{R}$ are second kind Fredhom operators. Given that the operator $\mathcal{K}_k$ is compact~\cite{HsiaoKleinman}, if the operator $\mathcal{R}$ is selected such that, modulo a compact operator, the composition of $\mathcal{T}_k$ and $\mathcal{R}$ plus $I/2$ is a bounded invertible
operator with a bounded inverse in the appropriate functional space setting, then the representations~\eqref{eq:Representation}-\eqref{eq:Representation2} lead to Regularized Combined Field Integral Equations (CFIER). Operators $\mathcal{R}$ with the aforementioned property are referred to as right regularizing operators for the operators $\mathcal{T}_k$. We note that the classical Combined Field Integral Equations assume the choice $\mathcal{R}=\xi\ \mathbf{n}\times \mathbf{I},\ \xi\in\mathbb{R}$~\cite{HarringtonMautz,ColtonKress}, and thus those operators $\mathcal{R}$ are not regularizing operators for $\mathcal{T}_k$. If $\mathcal{R}$ is chosen as a right regularizing operator for $\mathcal{T}_k$, the integral operator on the left hand side of (\ref{eq:CFIE_R}) is a second kind Fredholm operator, and, thus, the
unique solvability of equation (\ref{eq:CFIE_R}) is equivalent to the
injectivity of the left-hand-side operator. Operators $\mathcal{R}$ with the aforementioned property have been proposed and analyzed in the literature~\cite{turc1}. The aim of this paper is to present a general strategy to produce and analyze novel regularizing operators $\mathcal{R}$ and to analyze previously introduced regularizing operators for which an analysis does not exist in the literature.

The starting point of constructing suitable regularizing operators $\mathcal{R}$ is the Calder\'on's identity $\mathcal{T}_k^2=\frac{I}{4}-\mathcal{K}_k^2$~\cite{HsiaoKleinman}. The regularizing operators $\mathcal{R}$ should thus resemble the electric field operator $\mathcal{T}_k$. If the choice $\mathcal{R}=\mathcal{T}_k$ were made, the ensuing CFIER operators would not be injective since for certain wavenumbers $k$ the operators $\frac{I}{2}-\mathcal{K}_k$ are not injective~\cite{ColtonKress,HsiaoKleinman}. In order to ensure the injectivity of the resulting CFIER operators, one strategy is to modify the wavenumber in the definition of the regularizing operator $\mathcal{R}=\mathcal{T}_k$. Specifically, we propose a general regularizing operator $\mathcal{R}$ of the form
\begin{equation}\label{eq:defR}
\mathcal{R} = \eta\ \mathbf{n}\times \mathbf{S}_K +\zeta\ \mathcal{T}^1_K\ {\rm div}_\Gamma
\end{equation}
where $\eta$ and $\zeta$ are complex numbers and $K$ is a complex wavenumber such that $\Im K>0$. Given the starting point, we refer to the CFIER operators thus constructed as Calder\'on CFIER. We first establish that given the choice in equation~\eqref{eq:defR}, and for general closed, smooth, and simply connected  manifolds
$\Gamma$, the composition $\mathcal{T}_k\mathcal{R}$ can be
represented as a compact perturbation of an invertible diagonal matrix
operator. The main idea in the proof is to use the Helmholtz decomposition of tangential vector fields on the smooth surface $\Gamma$ in conjunction with the regularity
properties of integral operators with pseudo-homogeneous
kernels~\cite{Seeley} in the context of the Sobolev spaces $H_{\rm
  div}^{m}(\Gamma),\ m\geq 2$.

In what follows we use the notations and
relations~\cite{Nedelec}:
\begin{align}
&\overrightarrow{\rm{curl}}_\Gamma \phi = \nabla_\Gamma
\phi\times\mathbf{n}\\ &{\rm curl}_\Gamma \mathbf{a} = {\rm
div}_\Gamma (\mathbf{a}\times\mathbf{n})\\ &\Delta_\Gamma \phi = {\rm
div}_\Gamma \nabla_\Gamma \phi = - {\rm curl}_\Gamma
\overrightarrow{\rm{curl}}_\Gamma \phi
\end{align}
where $\mathbf{a}$ is a tangential vector field and where $\phi$ is a
scalar function defined on $\Gamma$. A few relevant properties of the
Helmholtz decomposition of tangential vector fields in Sobolev spaces are recounted in what
follows. We assume in what follows that in addition to being smooth, $\Gamma$ is simply connected. For a given smooth tangential vector field
$\mathbf{a} \in H_{\rm div}^{m}(\Gamma)$ we have the Helmholtz
decomposition~\cite{DeLaBourdon}
\begin{equation}
\label{eq:HDecomp}
\mathbf{a} = \nabla_\Gamma \phi + \overrightarrow{\rm{curl}}_\Gamma
\psi.
\end{equation}
We use the fact that $\Delta_\Gamma$ is an isomorphism from the space $H^{s+2}(\Gamma)/\mathbb{R}$ to the space $H^s_{*}(\Gamma)=\{u\in H^{s}(\Gamma): \int_\Gamma u=0\}$ for all $s\geq 0$~\cite{DeLaBourdon} and thus the right-inverse
$\Delta_\Gamma^{-1}:H^s_{*}(\Gamma)\to H^{s+2}(\Gamma)/\mathbb{R}$ can be properly defined. With the help of the latter operator, the functions $\phi$ and
$\psi$ in the Helmholtz decomposition~(\ref{eq:HDecomp}) are given by
$\phi=\Delta_\Gamma^{-1}\ {\rm div}_\Gamma\ \mathbf{a}$ and
$\psi=-\Delta_\Gamma^{-1}{\rm curl}_\Gamma\ \mathbf{a}$. Clearly for a tangential vector field $\mathbf{a} \in H_{\rm
div}^{m}(\Gamma)$ we have $\phi\in H^{m+2}(\Gamma)$ and $\psi
\in H^{m+1}(\Gamma)$~\cite{HsiaoKleinman}. Given that $H_{\rm
div}^{m}(\Gamma)=\nabla_\Gamma H^{m+2}(\Gamma)\oplus \overrightarrow{\rm{curl}}_\Gamma H^{m+1}(\Gamma)$~\cite{DeLaBourdon}, we can define the corresponding
orthogonal projection operators
\begin{align} 
& \Pi_{\nabla_\Gamma} = \nabla_\Gamma \Delta_\Gamma^{-1}\ {\rm
div}_\Gamma : H_{\rm div}^{m}(\Gamma)\to H_{\rm
div}^{m}(\Gamma),\\ & \Pi_{\overrightarrow{\rm curl}_\Gamma} =
-\overrightarrow{\rm curl}_\Gamma \Delta_\Gamma^{-1}\ {\rm curl}_\Gamma
: H_{\rm div}^{m}(\Gamma)\to H_{\rm 
div}^{m}(\Gamma).
\end{align}
We will also make use of a more detailed version of the Helmholtz decomposition~\eqref{eq:HDecomp} that we will review in what follows. For a smooth, closed, and simply connected two-dimensional manifold $\Gamma$ the
Laplace-Beltrami operator
$\Delta_\Gamma$ admits a complete and countable
sequence of eigenfunctions which form an orthonormal basis in
$L^{2}(\Gamma)$ \cite{Nedelec}, denoted by $\{Y_n\}_{0\leq n}$ such that
\begin{equation}
\label{eq:eig_laplace_beltrami}
-\Delta_\Gamma Y_n = \gamma_n Y_n,\ \gamma_n>0\ \mbox{for}\ 0<n.
\end{equation} 
These eigenfunctions of the Laplace-Beltrami operator turn out to be the building block for a
complete system of eigenfunctions of the vector Laplace-Beltrami
operator (or Hodge Laplace operator) $\overrightarrow{\Delta}_\Gamma=\nabla_\Gamma{\rm div}_\Gamma-\overrightarrow{\rm curl}_\Gamma{\rm curl}_\Gamma$.
Indeed, the system $\{\nabla_\Gamma Y_n,
\overrightarrow{\rm{curl}}_\Gamma Y_n \}_{1\leq n}$ forms a system of orthogonal
nontrivial eigenvectors for $\overrightarrow{\Delta}_\Gamma$ with the
same eigenvalues $\gamma_n$
\begin{eqnarray}
\label{eq:eig_vec_laplace_beltrami0}
-\overrightarrow{\Delta}_\Gamma \nabla_\Gamma Y_n &=& \gamma_n \nabla_\Gamma Y_n\\
\label{eq:eig_vec_laplace_beltrami1}
-\overrightarrow{\Delta}_\Gamma \overrightarrow{\rm{curl}}_\Gamma Y_n &=& \gamma_n \overrightarrow{\rm{curl}}_\Gamma Y_n.
\end{eqnarray} 
Given $\mathbf{v}\in L^{2}(TM(\Gamma))$, we have
\begin{equation}
\label{eq:HelmholtzDecomp}
\mathbf{v}=\sum_{n=1}^{\infty}v_n\ \frac{\nabla_\Gamma Y_n}{\sqrt{\gamma_n}} + \sum_{n=1}^{\infty}w_n\ \frac{\overrightarrow{\rm{curl}}_\Gamma Y_n}{\sqrt{\gamma_n}}
\end{equation}
so that $\{\frac{\nabla_\Gamma Y_n}{\sqrt{\gamma_n}},\frac{\overrightarrow{\rm{curl}}_\Gamma
Y_n}{\sqrt{\gamma_n}}\}_{1\leq n}$ is an orthonormal basis of the space of 
integrable tangential vector fields $L^{2}(TM(\Gamma))$, and an orthogonal basis in any of the 
Sobolev space $H^{s}(TM(\Gamma)),\ s\geq 0$ of tangential vector
fields~\cite[pp. 206, 207]{Nedelec}. 

Having reviewed the Helmholtz decomposition of tangential vector fields, we return to establishing the unique solvability of the CFIER equations~\eqref{eq:CFIE_R} with the choice of the regularizing operator given in equation~\eqref{eq:defR}. Using the notation $A\sim B$ for two operators $A$ and $B$
that differ by a compact operator from $H_{\rm
  div}^{m}(\Gamma),\ m\geq 2$ to itself, we recall a result established in~\cite{turc1}
\begin{lemma}\label{lem1}
The following property holds 
\begin{equation}
\mathcal{T}_k(\mathbf{n}\times \mathbf{S}_K)\sim \frac{i}{4k}\Pi_{\overrightarrow{\rm
    curl}_\Gamma}.
\end{equation}
\end{lemma}
Based on the result in Lemma~\ref{lem1}, we establish another useful result
\begin{lemma}\label{lem2}
The following property holds
\begin{equation}
\mathcal{T}_k(\mathcal{T}^1_K\ {\rm div}_\Gamma)\sim \frac{ik}{4}\Pi_{\nabla_\Gamma}.
\end{equation}
\end{lemma}
\begin{proof} We have that
\begin{eqnarray}
\mathcal{T}_k\ \mathcal{T}_K^1\ {\rm div}_\Gamma&=&ik(\mathbf{n}\times\mathbf{S}_k)\ \mathcal{T}_K^1\ {\rm div}_\Gamma=ik(\mathbf{n}\times\mathbf{S}_k)\ \mathcal{T}_k^1\ {\rm div}_\Gamma\nonumber\\
&+&ik (\mathbf{n}\times\mathbf{S}_k)\ (\mathcal{T}_K^1-\mathcal{T}_k^1)\ {\rm div}_\Gamma.\nonumber
\end{eqnarray}
 Given that $S_K-S_k:H^{s-1}(\Gamma)\to H^{s+2}(\Gamma)$ for all $s\geq 0$~\cite{AlougesLevadoux,Hiptmair2}, we get that $(\mathcal{T}_K^1-\mathcal{T}_k^1)\ {\rm div}_\Gamma=\mathbf{n}\times\nabla(S_K-S_k)\ {\rm div}_\Gamma:H^{s}(TM(\Gamma))\to H^{s+1}(TM(\Gamma))$ for all $s\geq 0$. If we further take into account the fact that $\mathbf{n}\times\mathbf{S}_k:H^p(TM(\Gamma))\to H^{p+1}(TM(\Gamma))$ for all $p\geq 0$~\cite{HsiaoKleinman,Nedelec}, we obtain $(\mathbf{n}\times\mathbf{S}_k)\  (\mathcal{T}_K^1-\mathcal{T}_k^1)\ {\rm div}_\Gamma:H^{s}(TM(\Gamma))\to H^{s+2}(TM(\Gamma))$, and hence ${\rm div}_\Gamma\ (\mathbf{n}\times\mathbf{S}_k)\  (\mathcal{T}_K^1-\mathcal{T}_k^1)\ {\rm div}_\Gamma:H^{s}(TM(\Gamma))\to H^{s+1}(TM(\Gamma))$ for all $s\geq 0$. Consequently, we get that $(\mathbf{n}\times\mathbf{S}_k)\  (\mathcal{T}_K^1-\mathcal{T}_k^1)\ {\rm div}_\Gamma:H_{\rm div}^{m}(\Gamma)\to H_{\rm div}^{m+1}(\Gamma)$. Given the compact embedding of $H_{\rm div}^{m+1}(\Gamma)$ into $H_{\rm div}^{m}(\Gamma)$~\cite{HsiaoKleinman,Nedelec} we obtain
$$\mathcal{T}_k\ \mathcal{T}_K^1\ {\rm div}_\Gamma\sim ik(\mathbf{n}\times\mathbf{S}_k)\  \mathcal{T}_k^1\ {\rm div}_\Gamma.$$
A simple consequence of Calder\'on's identity $\mathcal{T}_k^2=\frac{I}{4}-\mathcal{K}_k^2$ and the compactness of the magnetic field integral operator $\mathcal{K}_k$ in the space $H_{\rm div}^{m}(\Gamma)$~\cite{HsiaoKleinman,Nedelec} is that
$$(\mathbf{n}\times\mathbf{S}_k)\  \mathcal{T}_k^1\ {\rm div}_\Gamma+\mathcal{T}_k^1\ {\rm div}_\Gamma\  (\mathbf{n}\times\mathbf{S}_k)\sim\frac{I}{4}.$$
Given that $\mathcal{T}_k=ik\mathbf{n}\times\mathbf{S}_k+\frac{i}{k}\mathcal{T}_k^1\ {\rm div}_\Gamma$, we obtain
$$\mathcal{T}_k\ (\mathbf{n}\times \mathbf{S}_k)=ik(\mathbf{n}\times\mathbf{S}_k)\ (\mathbf{n}\times \mathbf{S}_k)+\frac{i}{k}(\mathcal{T}_k^1\ {\rm div}_\Gamma)\ (\mathbf{n}\times \mathbf{S}_k).$$
Since $\mathbf{n}\times \mathbf{S}_k:H^{s}(TM(\Gamma))\to H^{s+1}(TM(\Gamma))$ we get that $(\mathbf{n}\times\mathbf{S}_k)\ (\mathbf{n}\times \mathbf{S}_k):H^{s}(TM(\Gamma))\to H^{s+2}(TM(\Gamma))$ for all $s\geq 0$ and thus on account of compact embedding of Sobolev spaces we obtain
$$(\mathbf{n}\times\mathbf{S}_k)\ (\mathbf{n}\times \mathbf{S}_k):H_{\rm div}^{m}(\Gamma)\to H_{\rm div}^{m}(\Gamma)$$
is a compact operator. Consequently, the result in Lemma~\ref{lem1} leads to the following relation
$$\mathcal{T}_k^1\ {\rm div}_\Gamma\  (\mathbf{n}\times\mathbf{S}_k)\sim \frac{1}{4}\Pi_{\overrightarrow{\rm
    curl}_\Gamma}$$
from which we obtain
$$(\mathbf{n}\times\mathbf{S}_k)\  \mathcal{T}_k^1\ {\rm div}_\Gamma\sim \frac{1}{4}\Pi_{\nabla_\Gamma}$$
and hence the result of the Lemma now follows.
\end{proof}

\paragraph{Unique solvability} Combining the results from Lemma~\ref{lem1} and Lemma~\ref{lem2} we establish the unique solvability of the Calder\'on CFIER boundary integral equations~\eqref{eq:CFIE_R} with the choice of regularizing operator $\mathcal{R}$ given in equation~\eqref{eq:defR} under certain conditions on the coupling parameters $\eta$ and $\zeta$. We make use of a key result concerning certain positivity properties of scalar single layer potentials with complex wavenumbers. Specifically, we use the fact that for all  $\kappa_1>0$, {\em large enough} $\kappa_2>0$, and for all $\varphi\in H^{-1/2}(\Gamma),\ \varphi\neq 0$ the following positivity relations hold~\cite{Nedelec}
\begin{eqnarray}\label{eq:pos}
\Re\int_\Gamma\int_\Gamma G_{\kappa_1+i\kappa_2}(\mathbf{x}-\mathbf{y})\varphi(\mathbf{x})\overline{\varphi}(\mathbf{y})d\sigma(\mathbf{x})d\sigma(\mathbf{y})&>&0\nonumber\\
\Im\int_\Gamma\int_\Gamma G_{\kappa_1+i\kappa_2}(\mathbf{x}-\mathbf{y})\varphi(\mathbf{x})\overline{\varphi}(\mathbf{y})d\sigma(\mathbf{x})d\sigma(\mathbf{y})&>&0.
\end{eqnarray}
In what follows we use for simplicity the following notation
$$\langle S_{\kappa_1+i\kappa_2}\varphi,\varphi\rangle=\int_\Gamma\int_\Gamma G_{\kappa_1+i\kappa_2}(\mathbf{x}-\mathbf{y})\varphi(\mathbf{x})\overline{\varphi}(\mathbf{y})d\sigma(\mathbf{x})d\sigma(\mathbf{y})$$
and similar notations for its vector counterparts. Sufficient conditions for the validity of the positivity relations~\eqref{eq:pos} can be easily be established. Indeed, defining the Newton potential $U(\mathbf{z})=\int_\Gamma G_{\kappa_1+i\kappa_2}(\mathbf{z}-\mathbf{y})\varphi(\mathbf{y})d\sigma(\mathbf{y}),\ \mathbf{z}\in\mathbb{R}^3\setminus \Gamma$, applications of Green's identities and properties of the Dirichlet and Neumann traces of the potential $U$ on $\Gamma$ lead to the following relations
$$\Re\int_\Gamma\int_\Gamma G_{\kappa_1+i\kappa_2}(\mathbf{x}-\mathbf{y})\varphi(\mathbf{x})\overline{\varphi}(\mathbf{y})d\sigma(\mathbf{x})d\sigma(\mathbf{y})=(\kappa_2^2-\kappa_1^2)\|U\|^2_{L^2(\mathbb{R}^3)}+\|\nabla U\|^2_{L^2(\mathbb{R}^3)}$$
and
$$\Im\int_\Gamma\int_\Gamma G_{\kappa_1+i\kappa_2}(\mathbf{x}-\mathbf{y})\varphi(\mathbf{x})\overline{\varphi}(\mathbf{y})d\sigma(\mathbf{x})d\sigma(\mathbf{y})=2\kappa_1\kappa_2 \|U\|^2_{L^2(\mathbb{R}^3)}.$$
Thus, if $\kappa_2\geq \kappa_1$, the positivity properties~\eqref{eq:pos} hold. However, this requirement can be relaxed, see~\cite{Nedelec}.
\begin{theorem}\label{thm1}
If we take the wavenumber $K$ in the definition of the regularizing operator $\mathcal{R}$ defined in equation~\eqref{eq:defR} such that $K=\kappa_1+i\kappa_2,\ 0\leq \kappa_1,\ 0< \kappa_2$, then the boundary integral operator in the left-hand side of equations~(\ref{eq:CFIE_R}) satisfy the following property
\begin{equation}
\label{eq:param_3}
\frac{I}{2}-\mathcal{K}_k+ 
\mathcal{T}_k\  \mathcal{R}\sim
\left(\frac{1}{2}+\frac{ik \zeta}{4}\right)\Pi_{\nabla_\Gamma}+\left(\frac{1}{2}+\frac{i\eta}{4k}\right)\Pi_{\overrightarrow{\rm
    curl}_\Gamma}
\end{equation}
and thus are Fredholm operators of index zero in the spaces $H_{\rm div}^{m}(\Gamma),\ m\geq 2$. If in addition $\eta\neq 0$ and either $\Re\eta\geq 0,\ \Im\eta\leq 0,\ \Re\zeta\leq 0,\ Im\zeta\geq 0$ or $\Re\eta\leq 0,\ \Im\eta\geq 0,\ \Re\zeta\geq 0,\ \Im\zeta\leq 0$, the integral equations~\eqref{eq:CFIE_R} are uniquely solvable in the spaces $H_{\rm div}^{m}(\Gamma),\ m\geq 2$ for sufficiently large values of $\kappa_2$.
\end{theorem}

\begin{proof} The result in equation~\eqref{eq:param_3} follows from those established in Lemma~\ref{lem1} and Lemma~\ref{lem2}. Clearly, in the light of equation~\eqref{eq:param_3}, the boundary integral operator that enters the CFIER formulation is Fredholm in the spaces $H_{\rm div}^{m}(\Gamma),\ m\geq 2$. Consequently, the invertibility of this operator is equivalent to its injectivity. To establish its injectivity, let $\mathbf{a}$ be a solution of equation
(\ref{eq:CFIE_R}) with $\mathbf{E}^{i}=0$. It follows that the
electromagnetic field $(\mathbf{E}^{s},\mathbf{H}^{s})$ defined
by equations~(\ref{eq:Representation})-\eqref{eq:Representation2} is an outgoing
solution to the Maxwell equations in the unbounded domain
$\mathbb{R}^{3}\setminus D$ whose boundary values $\mathbf{E}_{+}^{s}$
on $\Gamma$ satisfy the homogeneous conditions
$\mathbf{n}\times\mathbf{E}_{+}^{s}=\mathbf{0}$. In view of the
uniqueness of radiating solutions for exterior Maxwell
problems~\cite{ColtonKress,Nedelec} we obtain
$\mathbf{E}^{s}=\mathbf{H}^{s}=\mathbf{0}$ identically in
$\mathbb{R}^{3}\setminus D$. It follows then from the standard jump relations of vector
layer potentials given in equations~\eqref{eq:traces} that the interior traces of the electric and magnetic 
fields defined in formulas~\eqref{eq:Representation} and~\eqref{eq:Representation2} satisfy the relations
$$-\mathbf{n}\times\mathbf{E}_{-}^{s} =\mathbf{a},\qquad -\mathbf{n}\times\ \mathbf{H}_{-}^{s} = \mathcal{R}\mathbf{a},\ {\rm on}\; \Gamma.$$
Taking the scalar product of the second of these relations with the
conjugate of the first one, using standard vector relations,
integrating over $\Gamma$ and apply the divergence theorem
gives
\begin{equation}\label{engy_rel}
  \int_\Gamma(\mathcal{R}\mathbf{a})\cdot \mathbf{n}\times\bar{\mathbf{a}}d\sigma=
  \int_\Gamma(\bar{\mathbf{E}}_{-}^{s}\times\mathbf{H}^s_{-})\cdot \mathbf{n}\ d\sigma=-ik\int_D\{|\mathbf{H}^{s}_{-}|^{2}-|\mathbf{E}^{s}_{-}|^{2}\}d\mathbf x.
\end{equation}
We have 
\begin{equation}\label{first_id}
\int_\Gamma (\mathbf{n}\times\mathbf{S}_K\mathbf{a})\cdot(\mathbf{n}\times\bar{\mathbf{a}})d\sigma=\int_\Gamma \mathbf{S}_K\mathbf{a}\cdot\bar{\mathbf{a}}d\sigma
\end{equation}
and
\begin{eqnarray}\label{second_id}
\int_\Gamma \mathcal{T}^1_K{\rm div}_\Gamma\mathbf{a}\cdot(\mathbf{n}\times\bar{\mathbf{a}})d\sigma&=&\int_\Gamma \mathbf{n}\times\nabla_\Gamma (S_K{\rm div}_\Gamma \mathbf{a})\cdot(\mathbf{n}\times\bar{\mathbf{a}})d\sigma\nonumber\\
&=&-\int_\Gamma \overrightarrow{\rm curl}_\Gamma (S_K{\rm div}_\Gamma \mathbf{a})\cdot(\mathbf{n}\times\bar{\mathbf{a}})d\sigma\nonumber\\
&=&-\int_\Gamma (S_K{\rm div}_\Gamma \mathbf{a})\ {\rm curl}_\Gamma(\mathbf{n}\times\bar{\mathbf{a}})d\sigma\nonumber\\
&=&-\int_\Gamma (S_K{\rm div}_\Gamma \mathbf{a})\ {\rm div}_\Gamma\bar{\mathbf{a}}d\sigma.
\end{eqnarray}
In conclusion we obtain from equations~\eqref{first_id} and~\eqref{second_id} that
\begin{equation}
\int_\Gamma(\mathcal{R}\mathbf{a})\cdot \mathbf{n}\times\bar{\mathbf{a}}d\sigma=\eta \int_\Gamma \mathbf{S}_K\mathbf{a}\cdot\bar{\mathbf{a}}d\sigma-\zeta \int_\Gamma (S_K{\rm div}_\Gamma \mathbf{a})\ {\rm div}_\Gamma\bar{\mathbf{a}}d\sigma.\nonumber
\end{equation}
Thus, denoting by $\eta=\eta_R+i\eta_I$ and $\zeta=\zeta_R+i\zeta_I$ we get that 
\begin{eqnarray}\label{eq:third_id}
\Re \int_\Gamma(\mathcal{R}\mathbf{a})\cdot \mathbf{n}\times\bar{\mathbf{a}}d\sigma&=&\eta_R \Re \langle\mathbf{S}_K\mathbf{a},\mathbf{a}\rangle-\eta_I  \Im \langle\mathbf{S}_K\mathbf{a},\mathbf{a}\rangle\nonumber\\
&-&\zeta_R \Re\langle S_K {\rm div}_\Gamma \mathbf{a},{\rm div}_\Gamma \mathbf{a}\rangle + \zeta_I \Im\langle S_K {\rm div}_\Gamma \mathbf{a},{\rm div}_\Gamma \mathbf{a}\rangle.
\end{eqnarray}
Given that for $K=\kappa_1+i\kappa_2$ with $\kappa\geq 0$ and $\epsilon\geq 0$ we have the following positivity relations which follow from equations~\eqref{eq:pos}
\begin{equation}\label{eq:positiv_R}
\Re \langle\mathbf{S}_K\mathbf{a},\mathbf{a}\rangle>0,\ \Im \langle\mathbf{S}_K\mathbf{a},\mathbf{a}\rangle\geq 0,\ \Re\langle S_K {\rm div}_\Gamma \mathbf{a},{\rm div}_\Gamma \mathbf{a}\rangle\geq 0,\ \Im\langle S_K {\rm div}_\Gamma \mathbf{a},{\rm div}_\Gamma \mathbf{a}\rangle \geq 0
\end{equation}
for all $\mathbf{a}\in H_{\rm div}^{-1/2}(\Gamma)$, $\mathbf{a}\neq 0$. Taking into account the identity~\eqref{eq:third_id} and the positivity relations~\eqref{eq:positiv_R} we obtain that given the choice
$$\eta\neq 0$$
we have that if
$$\eta_R\geq 0,\ \eta_I\leq 0,\ \zeta_R\leq 0,\ \zeta_I\geq 0,$$
or
$$\eta_R\leq 0,\ \eta_I\geq 0,\ \zeta_R\geq 0,\ \zeta_I\leq 0,$$
then 
$$\Re \int_\Gamma(\mathcal{R}\mathbf{a})\cdot \mathbf{n}\times\bar{\mathbf{a}}d\sigma>0\qquad {\rm or}\qquad  \Re \int_\Gamma(\mathcal{R}\mathbf{a})\cdot \mathbf{n}\times\bar{\mathbf{a}}d\sigma<0.$$
The last relations and equation~\eqref{engy_rel} imply that $\mathbf{a}=\mathbf{0}$
and thus the integral operator in the left-hand side of equation~\eqref{eq:CFIE_R} is
injective under the assumptions in this Theorem. 
\end{proof}
\begin{remark} We note that the choice $K=i\kappa_2,\ \kappa_2>0$, $\eta=\xi\kappa_2$, and $\zeta=-\xi\kappa_2^{-1}$, where $\xi>0$ leads to regularizing operators of the form $\mathcal{R}=-\xi\mathcal{T}_{i\kappa_2}$ that have been introduced in~\cite{Contopa_et_al}, but no rigorous proof of invertibility was provided in that reference. Also, the choice $K=i\kappa_2,\ \kappa_2>0$, $\eta=-\xi\kappa_2$, and $\zeta=0$ that amounts to choosing the regularizing operator in the form $\mathcal{R}=\xi\ \mathbf{n}\times \mathbf{S}_{i\kappa_2}$ has been proposed in the literature~\cite{turc1}. Since these types of regularizing operators are related to the Ikawa operator $-\Delta+\kappa_2^2 I$, we refer to the ensuing CFIER formulations as Calder\'on-Ikawa CFIER~\cite{Andriulli}.
\end{remark}
A main contribution of this work is to analyze a new type of regularizing operators given by $\mathcal{R}=-\gamma\ \mathcal{T}_{\kappa_1+i\kappa_2}$ with $\gamma>0$, $\kappa_1>0$ and $\kappa_2>0$ large enough---we refer to the ensuing CFIER formulations as Calder\'on-Complex CFIER. In this case, $\eta=-\gamma(\kappa_2+i\kappa_1)$ and $\zeta=\gamma\frac{\kappa_2+i\kappa_1}{\kappa_1^2+\kappa_2^2}$, and thus the unique solvability of the Calder\'on-Complex CFIER equations with the choice $\mathcal{R}=-\gamma\ \mathcal{T}_{\kappa_1+i\kappa_2}$ is not guaranteed by the result in Theorem~\ref{thm1}. Nevertheless, we establish this property in the following result:
\begin{theorem}\label{tm11}
If we consider the regularizing operator $\mathcal{R}=-\gamma\ \mathcal{T}_{\kappa_1+i\kappa_2}$ where $\gamma>0$, $\kappa_1>0$ and $\kappa_2>0$ and (i) there exists $c_{12}>0$ independent of $\kappa_1$ and $\kappa_2$ such that $\kappa_2=(1+c_{12})\kappa_1$ and (ii) $\kappa_2$ is large enough, then the boundary integral operators
\begin{equation}\label{eq:opB}
\mathcal{B}_{k,\gamma,\kappa_1,\kappa_2}=\frac{I}{2}-\mathcal{K}_k-\gamma\mathcal{T}_k\  \mathcal{T}_{\kappa_1+i\kappa_2}
\end{equation}
have the property
\begin{equation}
\label{eq:eig_Bgeneral}
\mathcal{B}_{k,\gamma,\kappa_1,\kappa_2}\sim
\left(\frac{1}{2}+\frac{\gamma\ k}{4(\kappa_1+i\kappa_2)}\right)\Pi_{\nabla_\Gamma}+\left(\frac{1}{2}+\frac{\gamma(\kappa_1+i\kappa_2)}{4k}\right)\Pi_{\overrightarrow{\rm
    curl}_\Gamma}.
\end{equation}
and in addition they are continuous with bounded inverses in the spaces $H_{\rm div}^{m}(\Gamma),\ m\geq 2$.
\end{theorem}
\begin{proof} We note that relation~\eqref{eq:eig_Bgeneral} follows from the results established in Lemma~\ref{lem1} and Lemma~\ref{lem2}. Just as in the proof of Theorem~\ref{thm1}, the result of the Theorem is established once we show that
\begin{eqnarray}\label{eq:third_id_n}
-\Re \int_\Gamma(\mathcal{T}_{\kappa_1+i\kappa_2}\mathbf{a})\cdot \mathbf{n}\times\bar{\mathbf{a}}d\sigma&=&\kappa_2 \Re \langle\mathbf{S}_{\kappa_1+i\kappa_2}\mathbf{a},\mathbf{a}\rangle+\kappa_1  \Im \langle\mathbf{S}_{\kappa_1+i\kappa_2}\mathbf{a},\mathbf{a}\rangle\nonumber\\
&+&\frac{\kappa_2}{\kappa_1^2+\kappa_2^2} \Re\langle S_{\kappa_1+i\kappa_2} {\rm div}_\Gamma \mathbf{a},{\rm div}_\Gamma \mathbf{a}\rangle - \frac{\kappa_1}{\kappa_1^2+\kappa_2^2} \Im\langle S_{\kappa_1+i\kappa_2} {\rm div}_\Gamma \mathbf{a},{\rm div}_\Gamma \mathbf{a}\rangle\nonumber\\
&>&0
\end{eqnarray}
for all $\kappa_1$ and $\kappa_2$ satisfying the assumptions (i) and (ii), and all $\mathbf{a}\in H^{m}_{\rm div}(\Gamma),\ \mathbf{a}\neq 0$ where $m\geq 2$. We will make use in the proof of this result of techniques developed in the proof of Lemma 5.6.1 in~\cite{Nedelec}. In a nutshell, we show that the sesquilinear form in the left-hand side of equation~\eqref{eq:third_id_n} can be represented as a principal (diagonal) part plus a part that can be made arbitrarily small for large enough values of $\kappa_2$. The principal part, in turn, is shown to be coercive with a coercivity constant that is independent of $\kappa_1$ and $\kappa_2$.

Let us assume a covering of $\mathbb{R}^3$ by a collection of uniformly distributed balls of radii $m(\log \kappa_2)/\kappa_2$. Let us denote these balls by $B_j$ and their centers by $b_j$. Let us denote by $B_i$ the balls that intersect $\Gamma$, in which case we denote $\Gamma_i=\Gamma\cap B_i$. Let us denote by $N$ the cardinality of the set $\{\Gamma_i\}$ that constitute a covering of $\Gamma$ for large enough $\kappa_2$; obviously $N\approx \kappa_2^2$. We use a partition of unity associated with the sets $\Gamma_i$ that consists of functions $\lambda_i$ that have value $1$ on the ball $B_i$ and their support included in the union of neighboring balls intersecting $B_i$. In each of these sets we consider a central point $y_i$ and an associated chart that is the projection $\psi_i$ on the tangent plane to $\Gamma$ at the point $y_i$. The maps $\psi_i$ are isomorphisms from $\Gamma_i$ to $\mathbb{R}^2$. 

Let $\varphi\in L^2(\Gamma)$; we decompose $\varphi$ using the partition of unity $\{\lambda_i\}$ in the form
$$\varphi=\sum_i \varphi_i,\quad \varphi_i=\varphi\lambda_i.$$
We have then 
$$\int_\Gamma\int_\Gamma G_{\kappa_1+i\kappa_2}(\mathbf{x}-\mathbf{y})\varphi(\mathbf{x})\overline{\varphi}(\mathbf{y})d\sigma(\mathbf{x})d\sigma(\mathbf{y})=\sum_{j,\ell}\int_{\Gamma_j}\int_{\Gamma_\ell}\frac{e^{(-\kappa_2+i\kappa_1)|\mathbf{x}-\mathbf{y}|}}{4\pi|\mathbf{x}-\mathbf{y}|}\varphi_j(\mathbf{x})\overline{\varphi}_\ell(\mathbf{y})d\sigma(\mathbf{x})d\sigma(\mathbf{y}).$$
For any pair of points $\mathbf{x}$ and $\mathbf{y}$ in two different balls $B_j$ and $B_\ell$ that do not intersect, we get that since $|\mathbf{x}-\mathbf{y}|\geq \frac{m(\log{\kappa_2})}{\kappa_2}$, then $\kappa_2^me^{-\kappa_2|\mathbf{x}-\mathbf{y}|}\leq 1$, and thus if the sets $\Gamma_j$ and $\Gamma_\ell$ do not intersect the following estimate holds
\begin{equation}\label{eq:est_tm2}
\left|\int_{\Gamma_j}\int_{\Gamma_\ell} G_{\kappa_1+i\kappa_2}(\mathbf{x}-\mathbf{y})\varphi_j(\mathbf{x})\overline{\varphi}_\ell(\mathbf{y})d\sigma(\mathbf{x})d\sigma(\mathbf{y})\right|\leq c\kappa_2^{m-1}\|\varphi\|_{L^2(\Gamma_j)}\|\varphi\|_{L^2(\Gamma_\ell)}
\end{equation}
where $c$ is a constant independent of $\kappa_2$. For a given number $p$ and any chart $\Gamma_i$ we associate the $p$ neighboring charts so that the distance between the central points $\mathbf{y}_i$ and $\mathbf{y}_j$ is less than $p(\log{\kappa_2^m})/\kappa_2$. We denote by $\Gamma_i^p$ the union of $\Gamma_i$ and its $p$ neighbors, and by $s_p(i)$ the set of indices $j$ such that $\Gamma_j$ is contained in $\Gamma_i^p$. We define the diagonal and quasi-diagonal contributions $D$ as
\begin{eqnarray}\label{eq:est_0_tm2}
D&=&\sum_i\int_\Gamma\int_\Gamma G_{\kappa_1+i\kappa_2}(\mathbf{x}-\mathbf{y})\left(\sum_{j\in s_p(i)}\varphi_j(\mathbf{x})\right)\left(\sum_{\ell\in s_p(i)}\overline{\varphi_\ell}(\mathbf{y})\right)d\sigma(\mathbf{x})d\sigma(\mathbf{y})\nonumber\\
&=&\frac{1}{N}\left(\sum_i{\rm card}\ s_p(i)\right)\int_\Gamma\int_\Gamma G_{\kappa_1+i\kappa_2}(\mathbf{x}-\mathbf{y})\varphi(\mathbf{x})\overline{\varphi}(\mathbf{y})d\sigma(\mathbf{x})d\sigma(\mathbf{y})+R\nonumber
\end{eqnarray}
where the remainder $R$ was shown in~\cite{Nedelec} to equal to 
\begin{eqnarray}
R&=&\int_\Gamma\int_\Gamma \left(\sum_{j,\ell}\left({\rm card}\ [s_p(j)\cap s_p(\ell)]-\frac{1}{N}\sum_i{\rm card}\ s_p(i)\right)\varphi_j(\mathbf{x})\overline{\varphi_\ell}(\mathbf{y})\right)\nonumber\\
&\times&G_{\kappa_1+i\kappa_2}(\mathbf{x}-\mathbf{y})d\sigma(\mathbf{x})d\sigma(\mathbf{y}).
\end{eqnarray}
For large enough $\kappa_2$, we have that ${\rm card}\ s_p(i)$ are all bounded by $cp^2$. Furthermore, for indices $j$ and $\ell$ such that the corresponding central points $\mathbf{y}_j$ and $\mathbf{y}_\ell$ satisfy $|\mathbf{y}_j-\mathbf{y}_\ell|\leq (\log{\kappa_2^m})/\kappa_2$ we have the following estimate
$$\left|{\rm card}\ [s_p(j)\cap s_p(\ell)]-\frac{1}{N}\sum_i{\rm card}\ s_p(i)\right|\leq cp.$$
Combining the estimate above with the estimate~\eqref{eq:est_tm2}, it follows~\cite{Nedelec} that the constants $m$ and $p$ can be chosen so that
\begin{eqnarray}\label{eq:est_1_tm2}
|\Re(R)|&\leq&c\frac{p}{N}\sum_{j,\ell}\|\varphi\|_{L^{2}(\Gamma_j)}\|\varphi\|_{L^{2}(\Gamma_\ell)}\leq c\frac{p}{\kappa_2^2}\sum_{j,\ell}\|\varphi\|_{L^{2}(\Gamma_j)}\|\varphi\|_{L^{2}(\Gamma_\ell)}\nonumber\\
|\Im(R)|&\leq&c\frac{p}{N}\sum_{j,\ell}\|\varphi\|_{L^{2}(\Gamma_j)}\|\varphi\|_{L^{2}(\Gamma_\ell)}\leq c\frac{p}{\kappa_2^2}\sum_{j,\ell}\|\varphi\|_{L^{2}(\Gamma_j)}\|\varphi\|_{L^{2}(\Gamma_\ell)}.
\end{eqnarray}
Estimates~\eqref{eq:est_1_tm2} just established show that for large enough values of $\kappa_2$ the main contributions to the expressions $\Re\langle S_{\kappa_1+i\kappa_2}\varphi,\varphi\rangle$ and $\Im\langle S_{\kappa_1+i\kappa_2}\varphi,\varphi\rangle$ come from the real and imaginary parts of the diagonal and quasi-diagonal terms $D$. 

We turn our attention to the diagonal and quasi-diagonal contributions $D$. For each of the domains $\Gamma_i^p$ we use the charts associated with $\Gamma_i$, that is the projection on the tangent plane at the central point $\mathbf{y}_i$. Given two points $\mathbf{x}$ and $\mathbf{y}$ in $\Gamma_i^p$, we denote by $\mathbf{x}_p$ and $\mathbf{y}_p$ their projections onto the tangent plane at $\mathbf{y}_i$ so that $\mathbf{x}=\mathbf{x}_p+s(\mathbf{x}_p)\mathbf{n}(\mathbf{y}_i)$ and $\mathbf{y}=\mathbf{y}_p+s(\mathbf{y}_p)\mathbf{n}(\mathbf{y}_i)$. We have that $|\mathbf{x}-\mathbf{y}|^2=|\mathbf{x}_p-\mathbf{y}_p|^2+|s(\mathbf{x}_p)-s(\mathbf{y})_p|^2$, $\nabla s(\mathbf{y}_i)=0$, and $D^2s$ is bounded, from which it follows that~\cite{Nedelec}
$$\frac{e^{(-\kappa_2+i\kappa_1)|\mathbf{x}-\mathbf{y}|}}{|\mathbf{x}-\mathbf{y}|}=\frac{e^{(-\kappa_2+i\kappa_1)|\mathbf{x}_p-\mathbf{y}_p|}}{|\mathbf{x}_p-\mathbf{y}_p|}(1+\psi(\mathbf{x}_p,\mathbf{y}_p))$$
where the function $\psi$ is bounded by the quantity $p(\log{\kappa_2})/\kappa_2$. Thus, for large enough $\kappa_2$, the dominant contributions to the expression $D$ related to $\Gamma_i^p$ stem from expressions of the form
\begin{equation}
D_i=\int_{\mathbb{R}^2}\int_{\mathbb{R}^2}\frac{e^{(-\kappa_2+i\kappa_1)|\mathbf{x}_p-\mathbf{y}_p|}}{|\mathbf{x}_p-\mathbf{y}_p|}\varphi(\mathbf{x}_p)\overline{\varphi}(\mathbf{y}_p)d\sigma(\mathbf{x}_p)d\sigma(\mathbf{y}_p)\nonumber
\end{equation}
where $\varphi$ is compactly supported in $\mathbb{R}^2$, that is the difference $|\langle S_{\kappa_1+i\kappa_2}\varphi,\varphi\rangle-D_i|$ can be bounded by products of quantities that decay rapidly to zero as $\kappa_2\to\infty$ and $\|\varphi\|_{L^2(\Gamma)}$. The same argument can be repeated via an additional integration by parts procedure to deliver similar bounds on expressions $\langle S_{\kappa_1+i\kappa_2}{\rm div}_\Gamma\mathbf{v},{\rm div}_\Gamma \mathbf{v}\rangle$ for $\mathbf{v}\in H_{\rm div}^1$: the absolute value of the difference between those expressions and their dominant contributions can be bounded by products of quantities that decay rapidly to zero as $\kappa_2\to\infty$ and $\|\mathbf{v}\|_{L^2(TM(\Gamma))}$. Therefore, the targeted inequality~\eqref{eq:third_id_n} follows once we establish that the sesquilinear form in the left hand side of equation~\eqref{eq:third_id_n} is coercive with a coercivity constant independent of $\kappa_1$ and $\kappa_2$ in the case when $\Gamma=\{x_3=0\}$ and $\mathbf{a}$ is a vector field compactly supported in $\mathbb{R}^2$. The latter inequality, in turn, is established using Fourier transforms; in what follows the hat notation refers to Fourier transformed functions. In the case when $\Gamma=\{x_3=0\}$, all of the boundary integral operators that enter equation~\eqref{eq:third_id_n} are convolutions and can be expressed in the Fourier space in terms of the Fourier transform of $G_{\kappa_1+i\kappa_2}(x_1,x_2;0)$ with respect to the first two variables. Given the outgoing property of  $G_{\kappa_1+i\kappa_2}(x_1,x_2;0)$, it can be shown that~\cite{Erdelyi} 
\begin{equation}
\hat{G}_{\kappa_1+i\kappa_2}(\mathbf{\xi})=\frac{1}{2\sqrt{|\mathbf{\xi}|^2-(\kappa_1+i\kappa_2)^2}}\nonumber
\end{equation}
where the square root is chosen so that its real and imaginary parts are both positive. Using Plancherel's identity we obtain
\begin{eqnarray}\label{eq:third_id_nn}
\kappa_2\Re \langle\mathbf{S}_{\kappa_1+i\kappa_2}\mathbf{a},\mathbf{a}\rangle&+&\kappa_1  \Im \langle\mathbf{S}_{\kappa_1+i\kappa_2}\mathbf{a},\mathbf{a}\rangle+\frac{\kappa_2}{\kappa_1^2+\kappa_2^2} \Re\langle S_{\kappa_1+i\kappa_2} {\rm div}_\Gamma \mathbf{a},{\rm div}_\Gamma \mathbf{a}\rangle\nonumber\\
&-&\frac{\kappa_1}{\kappa_1^2+\kappa_2^2} \Im\langle S_{\kappa_1+i\kappa_2} {\rm div}_\Gamma \mathbf{a},{\rm div}_\Gamma \mathbf{a}\rangle\nonumber\\
&=&\frac{\kappa_2}{2}\int_{\mathbb{R}^2}\Re\{(|\xi|^2-(\kappa_1+i\kappa_2)^2)^{-1/2}\}|\hat{\mathbf{a}}(\xi)|^2d\xi\nonumber\\
&+&\frac{\kappa_1}{2}\int_{\mathbb{R}^2}\Im\{(|\xi|^2-(\kappa_1+i\kappa_2)^2)^{-1/2}\}|\hat{\mathbf{a}}(\xi)|^2d\xi\nonumber\\
&+&\frac{\kappa_2}{2(\kappa_1^2+\kappa_2^2)}\int_{\mathbb{R}^2}\Re\{(|\xi|^2-(\kappa_1+i\kappa_2)^2)^{-1/2}\}|\xi\cdot\hat{\mathbf{a}}(\xi)|^2d\xi\nonumber\\
&-&\frac{\kappa_1}{2(\kappa_1^2+\kappa_2^2)}\int_{\mathbb{R}^2}\Im\{(|\xi|^2-(\kappa_1+i\kappa_2)^2)^{-1/2}\}|\xi\cdot\hat{\mathbf{a}}(\xi)|^2d\xi.
\end{eqnarray}
The precise definition of the square root in equation~\eqref{eq:third_id_nn} is given below. We define
\begin{eqnarray*}
\eta&:=&\kappa_1^2-\kappa_2^2\nonumber\\
\rho&:=&\left[(|\xi|^2-\eta)^2+4\kappa_1^2\kappa_2^2\right]^{1/2},
\end{eqnarray*}
where the expression in the brackets in the definition of $\rho$ can be seen to be positive. With the notations above in place, we compute the real and imaginary parts of $(|\xi|^2-(\kappa_1+i\kappa_2)^2)^{-1/2}$:
\begin{eqnarray*}
\Re\{(|\xi|^2-(\kappa_1+i\kappa_2)^2)^{-1/2}\}&=&\frac{\sqrt{|\xi|^2-\eta+\rho}}{\sqrt{2\rho}}\nonumber\\
\Im \{(|\xi|^2-(\kappa_1+i\kappa_2)^2)^{-1/2}\}&=&\frac{\sqrt{2}\kappa_1\kappa_2}{\sqrt{\rho}\sqrt{|\xi|^2-\eta+\rho}}.
\end{eqnarray*} 
Given the assumption (i) about $\kappa_1$ and $\kappa_2$ we have that $\kappa_2>\kappa_1$. We show first that under this assumption the following inequality holds
\begin{equation}\label{eq:ineqRe}
\kappa_2 \Re\{(|\xi|^2-(\kappa_1+i\kappa_2)^2)^{-1/2}\}\leq \kappa_1 \Im \{(|\xi|^2-(\kappa_1+i\kappa_2)^2)^{-1/2}\},\ {\rm for\ all}\ \xi\in\mathbb{R}^2.
\end{equation}
We see that the inequality~\eqref{eq:ineqRe} is equivalent to 
\[
|\xi|^2\leq \kappa_1^2 +\kappa_2^2 + \rho=\kappa_1^2 +\kappa_2^2+\left[(|\xi|^2-\eta)^2+4\kappa_1^2\kappa_2^2\right]^{1/2}\ {\rm for\ all}\ \xi\in\mathbb{R}^2,
\]
which obviously holds true given that $\eta\leq 0$. We introduce the following notations
\begin{eqnarray*}
A:&=&\int_{\mathbb{R}^2}\left(\kappa_2\Re\{(|\xi|^2-(\kappa_1+i\kappa_2)^2)^{-1/2}\}+\kappa_1 \Im\{(|\xi|^2-(\kappa_1+i\kappa_2)^2)^{-1/2}\}\right)|\hat{\mathbf{a}}(\xi)|^2d\xi\nonumber\\
&=&\int_{\mathbb{R}^2}\Im\left((\kappa_1+i\kappa_2)[|\xi|^2-(\kappa_1+i\kappa_2)^2]^{-1/2}\right)|\hat{\mathbf{a}}(\xi)|^2d\xi\nonumber\\
B:&=&\int_{\mathbb{R}^2}\frac{\kappa_1\Im\{(|\xi|^2-(\kappa_1+i\kappa_2)^2)^{-1/2}\}-\kappa_2\Re\{(|\xi|^2-(\kappa_1+i\kappa_2)^2)^{-1/2}\}}{\kappa_1^2+\kappa_2^2}|\xi\cdot\hat{\mathbf{a}}(\xi)|^2d\xi\nonumber\\
&=&\int_{\mathbb{R}^2}\Im\left(\frac{[|\xi|^2-(\kappa_1+i\kappa_2)^2]^{-1/2}}{\kappa_1+i\kappa_2}\right)|\xi\cdot\hat{\mathbf{a}}(\xi)|^2d\xi.
\end{eqnarray*}
We note that equation~\eqref{eq:ineqRe} implies the positivity of the imaginary part of the expression in the definition of expression $B$. This, together with the Cauchy-Schwartz inequality $|\xi\cdot\hat{\mathbf{a}}(\xi)|^2\leq |\xi|^2|\hat{\mathbf{a}}(\xi)|^2$, imply that
\[
B\leq E:=\int_{\mathbb{R}^2}\Im\left(\frac{|\xi|^2[|\xi|^2-(\kappa_1+i\kappa_2)^2]^{-1/2}}{\kappa_1+i\kappa_2}\right)|\hat{\mathbf{a}}(\xi)|^2d\xi.
\]
We have then that
\[
A-B\geq A-E=\int_{\mathbb{R}^2}\Im\left(\frac{-\kappa_1+i\kappa_2}{\kappa_1^2+\kappa_2^2}[|\xi|^2-(\kappa_1+i\kappa_2)^2]^{1/2}\right)|\hat{\mathbf{a}}(\xi)|^2d\xi.
\]
Given that
\begin{eqnarray*}
\Re \sqrt{|\xi|^2-(\kappa_1+i\kappa_2)^2}&=&\frac{\sqrt{|\xi|^2-\eta+\rho}}{\sqrt{2}}\nonumber\\
\Im \sqrt{|\xi|^2-(\kappa_1+i\kappa_2)^2}&=&-\frac{\sqrt{2}\kappa_1\kappa_2}{\sqrt{|\xi|^2-\eta+\rho}}
\end{eqnarray*} 
we get that 
\[
\Im\left(\frac{-\kappa_1+i\kappa_2}{\kappa_1^2+\kappa_2^2}[|\xi|^2-(\kappa_1+i\kappa_2)^2]^{1/2}\right)=
\frac{1}{\kappa_1^2+\kappa_2^2}\left(\frac{\kappa_2\sqrt{|\xi|^2-\eta+\rho}}{\sqrt{2}}+\frac{\sqrt{2}\kappa_1^2\kappa_2}{\sqrt{|\xi|^2-\eta+\rho}}\right).
\]
Given that $\eta\leq 0$, it follows from its definition that $\rho\geq |\xi|^2-\eta$, and thus $\sqrt{|\xi|^2-\eta+\rho}\geq \sqrt{2}\sqrt{\kappa_2^2-\kappa_1^2}$ for all $\xi\in\mathbb{R}^2$. We have then
\[
\Im\left(\frac{-\kappa_1+i\kappa_2}{\kappa_1^2+\kappa_2^2}[|\xi|^2-(\kappa_1+i\kappa_2)^2]^{1/2}\right)\geq \frac{1}{2}\left(1-\frac{\kappa_1^2}{\kappa_2^2}\right)^{1/2}\geq\frac{1}{2}\left(1-\frac{1}{(1+c_{12})^2}\right)^{1/2},
\]
by taking into account assumption (i) on $\kappa_1$ and $\kappa_2$. In conclusion, we obtain
\begin{eqnarray*}\label{eq:third_id_nnf}
\kappa_2\Re \langle\mathbf{S}_{\kappa_1+i\kappa_2}\mathbf{a},\mathbf{a}\rangle&+&\kappa_1  \Im \langle\mathbf{S}_{\kappa_1+i\kappa_2}\mathbf{a},\mathbf{a}\rangle+\frac{\kappa_2}{\kappa_1^2+\kappa_2^2} \Re\langle S_{\kappa_1+i\kappa_2} {\rm div}_\Gamma \mathbf{a},{\rm div}_\Gamma \mathbf{a}\rangle\nonumber\\
&-&\frac{\kappa_1}{\kappa_1^2+\kappa_2^2} \Im\langle S_{\kappa_1+i\kappa_2} {\rm div}_\Gamma \mathbf{a},{\rm div}_\Gamma \mathbf{a}\rangle=A-B\geq \nonumber\frac{1}{2}\left(1-\frac{1}{(1+c_{12})^2}\right)^{1/2}\|\mathbf{a}\|_{L^2(TM(\Gamma))}^2,
\end{eqnarray*}
from which the proof of the theorem follows.
\end{proof}
\begin{remark}The assumption (i) in the statement of the Theorem~\ref{tm11} is not sharp, see the numerical results in and Section~\ref{DtN} and Section~\ref{num_exp}.\end{remark}

\section{Principal symbol regularizing operators}\label{pscfier}
We present in what follows another class of regularizing operators that consists of operators that have the same principal symbol in the sense of pseudodifferential operators~\cite{MTaylor,Seeley} as the operators $\mathcal{R}$ defined in equations~\eqref{eq:defR}. The main motivation behind using principal symbol regularizing operators is that the resulting CFIER formulations have similar spectral properties as those that use regularizing operators defined by boundary layers (see Section~\ref{spec_prop_cfie}), but the computational cost associated with the evaluation of principal symbol operators is a fraction of the cost associated with the evaluation of the corresponding boundary layer operators. Indeed, since the latter operators are Fourier multipliers, they can be evaluated efficiently by means of FFTs~\cite{turc7,Levadoux}, or can be approximated by rational functions (e.g. Pad\'e approximants) which can be also evaluated efficiently~\cite{Boujaji,Antoine,AntoineX}.

The calculation of the principal symbols of the operators $\mathcal{R}$ defined in equations~\eqref{eq:defR}, is based on the principal symbols of scalar and vector single layer operators $S_K$ and $\pi_\Gamma \mathbf{S}_K$ for complex wavenumbers $K$ such that $\Im K>0$, where the projection operator onto the tangent space of $\Gamma$ is defined by $\pi_\Gamma=(\mathbf{n}\times\cdot)\times\mathbf{n}$. The principal symbols of these operators are equal to $\frac{1}{2\sqrt{|\mathbf{\xi}|^2-K^2}}$ and $\frac{1}{2\sqrt{|\mathbf{\xi}|^2-K^2}} \mathbf{I}_2$ respectively~\cite{Chen,Smirnov,AlougesLevadoux,turc7}, where we view the operator $\pi_\Gamma \mathbf{S}_k$ in terms of the Helmholtz decomposition and $\mathbf{I}_2$ stands for the identity $2\times 2$ matrix. In the previous equations the variable $\mathbf{\xi}\in TM^{*}(\Gamma)$ represents the Fourier symbol of the tangential gradient operator $\nabla_\Gamma$, and $TM^{*}(\Gamma)$ represents the cotangent bundle of $\Gamma$~\cite{Taylor}. Given that the principal symbol of the operators $\frac{1}{2}(-\Delta_\Gamma-K^2)^{-\frac{1}{2}}$ is equal to $\frac{1}{2\sqrt{|\mathbf{\xi}|^2-K^2}}$ and the principal symbol of the operator $\frac{1}{2}(-\Delta_\Gamma-K^2)^{-\frac{1}{2}}\mathbf{I}_2$ is equal to $\frac{1}{2\sqrt{|\mathbf{\xi}|^2-K^2}} \mathbf{I}_2$~\cite{Taylor}, the previous statement can be expressed more precisely in the form~\cite{Smirnov,AlougesLevadoux,turc7}
\begin{equation}
\label{eq:psSL}
S_K = \frac{1}{2}(-\Delta_\Gamma-K^2)^{-\frac{1}{2}}\ \rm{mod}\ \Psi^{-3}(\Gamma),\ \pi_\Gamma \mathbf{S}_K = \frac{1}{2}(-\Delta_\Gamma-K^2)^{-\frac{1}{2}}\mathbf{I}_2\ \rm{mod}\ \Psi^{-3}(TM(\Gamma))
\end{equation}
where $\Psi^{-3}(\Gamma)$ denotes the operator algebra of pseudodifferential operators of order $-3$ on $\Gamma$ and  $\Psi^{-3}(TM(\Gamma))$ denotes the operator algebra of pseudodifferential operators of order $-3$ on the space of tangential vector fields on $TM(\Gamma)$. The meaning of the notation $A=B\ \rm{mod}\ \Psi^{s}(\Gamma)$ where $A$ and $B$ are pseudodifferential operators defined on scalar functions on $\Gamma$ is that $A-B$ is a pseudodifferential operator of order $s$, that is $A-B:H^p(\Gamma)\to H^{p+s}(\Gamma)$; the definition is similar in the case when the operators $A$ and $B$ are tangential pseudodifferential operators. It follows immediately then from equations~\eqref{eq:psSL} that
\begin{equation}\label{eq:ntpsSL}
\mathbf{n}\times\mathbf{S}_K=\frac{1}{2}\left(\begin{array}{c}\nabla_\Gamma\\\overrightarrow{\rm curl}_\Gamma\end{array}\right)^T\left(\begin{array}{cc}0&(-\Delta_\Gamma-K^2)^{-\frac{1}{2}}\\-(-\Delta_\Gamma-K^2)^{-\frac{1}{2}}&0\end{array}\right)\left(\begin{array}{c}\Delta_\Gamma^{-1}{\rm div}_\Gamma\\ -\Delta_\Gamma^{-1}{\rm curl}_\Gamma\end{array}\right)\ \rm{mod}\ \Psi^{-3}(TM(\Gamma)).
\end{equation}
We make use of the following vector calculus identities
$$(-\overrightarrow{\Delta_\Gamma}-K^2)^{-\frac{1}{2}}\nabla_\Gamma\phi=\nabla_\Gamma(-\Delta_\Gamma-K^2)^{-\frac{1}{2}}\phi\qquad (-\overrightarrow{\Delta_\Gamma}-K^2)^{-\frac{1}{2}}\overrightarrow{\rm curl}_\Gamma\psi=\overrightarrow{\rm curl}_\Gamma(-\Delta_\Gamma-K^2)^{-\frac{1}{2}}\psi$$
 which follow from representing the (smooth) scalar functions $\phi$ and $\psi$ in the complete basis $\{Y_n\}_{0\leq n}$ of $L^2(\Gamma)$ and taking into account equations~\eqref{eq:eig_laplace_beltrami} and~\eqref{eq:HelmholtzDecomp}. Let $\mathbf{a}=\nabla_\Gamma\phi+\overrightarrow{\rm curl}_\Gamma \psi$; taking into account the vector calculus identities presented above, we get
\[
\left(\begin{array}{c}\nabla_\Gamma\\\overrightarrow{\rm curl}_\Gamma\end{array}\right)^T\left(\begin{array}{cc}0&(-\Delta_\Gamma-K^2)^{-\frac{1}{2}}\\-(-\Delta_\Gamma-K^2)^{-\frac{1}{2}}&0\end{array}\right)(\nabla_\Gamma \phi+\overrightarrow{\rm curl}_\Gamma \psi)\]
\[
=(-\overrightarrow{\Delta}_\Gamma-K^2)^{-\frac{1}{2}}(\nabla_\Gamma \psi-\overrightarrow{\rm curl}_\Gamma \phi)=(-\overrightarrow{\Delta}_\Gamma-K^2)^{-\frac{1}{2}}(\mathbf{n}\times \mathbf{a}).
\]
Thus, we can rewrite the previous equation~\eqref{eq:ntpsSL} in the form
\begin{equation}\label{eq:PSntS}
\mathbf{n}\times\mathbf{S}_K=\frac{1}{2}(-\overrightarrow{\Delta}_\Gamma-K^2)^{-\frac{1}{2}}(\mathbf{n}\times\pi_\Gamma)\ \rm{mod}\ \Psi^{-3}(TM(\Gamma)).
\end{equation}
Similarly, given that $\mathcal{T}_K^1{\rm div}_\Gamma=-\overrightarrow{\rm curl}_\Gamma S_k{\rm div}_\Gamma$ we obtain
\begin{equation}\label{eq:PSTa}
\mathcal{T}_K^1{\rm div}_\Gamma=-\frac{1}{2}\left(\begin{array}{c}\nabla_\Gamma\\\overrightarrow{\rm curl}_\Gamma\end{array}\right)^T\left(\begin{array}{cc}0&0\\(-\Delta_\Gamma-K^2)^{-\frac{1}{2}}\Delta_\Gamma&0\end{array}\right)\left(\begin{array}{c}\Delta_\Gamma^{-1}{\rm div}_\Gamma\\ -\Delta_\Gamma^{-1}{\rm curl}_\Gamma\end{array}\right)\ \rm{mod}\ \Psi^{-1}(TM(\Gamma)).
\end{equation}
Using the previous vector calculus identities we get an equivalent representation of equations~\eqref{eq:PSTa} in the form 
\begin{equation}\label{eq:PST}
\mathcal{T}_K^1{\rm div}_\Gamma=-\frac{1}{2}(-\overrightarrow{\Delta}_\Gamma-K^2)^{-\frac{1}{2}}\ \overrightarrow{\rm curl}_\Gamma{\rm curl}_\Gamma(\mathbf{n}\times\pi_\Gamma)\ \rm{mod}\ \Psi^{-1}(TM(\Gamma)).
\end{equation}
If we define 
\begin{eqnarray}\label{eq:defPSR}
PS(\mathcal{R}) &=& \eta\ PS(\mathbf{n}\times \mathbf{S}_K) +\zeta\ PS(\mathcal{T}^1_K\ {\rm div}_\Gamma)\nonumber\\
PS(\mathbf{n}\times \mathbf{S}_K)&=&\frac{1}{2}(-\overrightarrow{\Delta}_\Gamma-K^2)^{-\frac{1}{2}}(\mathbf{n}\times\pi_\Gamma)\nonumber\\
PS(\mathcal{T}^1_K\ {\rm div}_\Gamma)&=&-\frac{1}{2}(-\overrightarrow{\Delta}_\Gamma-K^2)^{-\frac{1}{2}}\ \overrightarrow{\rm curl}_\Gamma{\rm curl}_\Gamma(\mathbf{n}\times\pi_\Gamma)
\end{eqnarray}
then the following result is the counterpart of the result in Theorem~\ref{thm1} if we use in equations~\eqref{eq:Representation} principal symbol regularizing operators $PS(\mathcal{R})$ defined in equations~\eqref{eq:defPSR} instead of the regularizing operators $\mathcal{R}$ defined in equations~\eqref{eq:defR}
\begin{theorem}\label{thm2}
If we take the wavenumber $K$ in the definition of the regularizing operator $PS(\mathcal{R})$ defined in equation~\eqref{eq:defPSR} such that $K=\kappa_1+i\kappa_2,\ 0\leq \kappa_1,\ 0\leq \kappa_2$, then the following property holds
\begin{equation}
\label{eq:param_PS30}
\frac{I}{2}-\mathcal{K}_k+ 
\mathcal{T}_k\  PS(\mathcal{R})\sim
\left(\frac{1}{2}+\frac{ik \zeta}{4}\right)\Pi_{\nabla_\Gamma}+\left(\frac{1}{2}+\frac{i\eta}{4k}\right)\Pi_{\overrightarrow{\rm
    curl}_\Gamma}.
\end{equation}
Thus, the operators $\frac{I}{2}-\mathcal{K}_k+ 
\mathcal{T}_k\  PS(\mathcal{R})$ are Fredholm operators of index zero in the spaces $H_{\rm div}^{m}(\Gamma),\ m\geq 2$. If in addition $\eta\neq 0$ and either $\Re\eta\geq 0,\ \Im\eta\leq 0,\ \Re\zeta\leq 0,\ Im\zeta\geq 0$ or $\Re\eta\leq 0,\ \Im\eta\geq 0,\ \Re\zeta\geq 0,\ \Im\zeta\leq 0$, the operators $\frac{I}{2}-\mathcal{K}_k+ 
\mathcal{T}_k\  PS(\mathcal{R})$ are invertible with bounded inverses in the spaces $H_{\rm div}^{m}(\Gamma),\ m\geq 2$.
\end{theorem}
\begin{proof} We obtain from relations~\eqref{eq:PSntS} and~\eqref{eq:PST} that
$$\eta \mathbf{n}\times \mathbf{S}_K +\zeta\ \mathcal{T}^1_K\ {\rm div}_\Gamma-\eta\ PS(\mathbf{n}\times \mathbf{S}_K) -\zeta\ PS(\mathcal{T}^1_K\ {\rm div}_\Gamma):H^{s}(TM(\Gamma))\to H^{s+1}(TM(\Gamma)).$$
Furthermore, since ${\rm div}_\Gamma\ \mathcal{T}^1_K\ {\rm div}_\Gamma={\rm div}_\Gamma\ PS(\mathcal{T}^1_K\ {\rm div}_\Gamma)=0$, we obtain 
$${\rm div}_\Gamma(\mathcal{R}-PS(\mathcal{R}))=\eta\ {\rm div}_\Gamma (\mathbf{n}\times \mathbf{S}_K-PS(\mathbf{n}\times \mathbf{S}_K)):H^{s}(TM(\Gamma))\to H^{s+2}(TM(\Gamma)).$$
Consequently, we get that
$$\mathcal{R}-PS(\mathcal{R}):H_{\rm div}^{m}(\Gamma)\to H_{\rm div}^{m+1}(\Gamma)$$
and hence the operator $\mathcal{R}-PS(\mathcal{R}):H_{\rm div}^{m}(\Gamma)\to H_{\rm div}^{m}(\Gamma)$ is compact given the compact embedding of the space $H_{\rm div}^{m+1}(\Gamma)$ in the space $H_{\rm div}^{m}(\Gamma)$. Taking into account the mapping property $\mathcal{T}_k:H_{\rm div}^{m}(\Gamma)\to H_{\rm div}^{m}(\Gamma)$~\cite{HsiaoKleinman,Nedelec}, we established
$$\frac{I}{2}-\mathcal{K}_k+ 
\mathcal{T}_k\  PS(\mathcal{R})\sim \frac{I}{2}-\mathcal{K}_k+ 
\mathcal{T}_k\  \mathcal{R}.$$
Relation~\eqref{eq:param_PS30} follows now from the formula established above and relation~\eqref{eq:param_3}. Also, the operators $\frac{I}{2}-\mathcal{K}_k+ 
\mathcal{T}_k\  PS(\mathcal{R}):H_{\rm div}^{m}(\Gamma)\to H_{\rm div}^{m}(\Gamma)$ are Fredholm of index zero and thus their invertibility is equivalent to their injectivity. Just like in the proof of Theorem~\ref{thm1}, the injectivity of the operators $\frac{I}{2}-\mathcal{K}_k+ 
\mathcal{T}_k\  PS(\mathcal{R})$ will follow once we establish the positivity/negativity of the sesquilinear form $\Re \int_\Gamma(PS(\mathcal{R})\mathbf{a})\cdot \mathbf{n}\times\bar{\mathbf{a}}d\sigma$. The main ingredient in the proof of the coercivity property of $\pm \Re(PS(\mathcal{R}))$ is the Helmholtz decomposition~\eqref{eq:HelmholtzDecomp}. We write the Helmholtz decomposition of the smooth tangential vector field $\mathbf{v}=\mathbf{n}\times\mathbf{a}$ in the form
\begin{equation}
\label{eq:HelmholtzDecompF}
\mathbf{n}\times\mathbf{a}=\sum_{n=1}^{\infty}a_n\ \frac{\nabla_\Gamma Y_n}{\sqrt{\gamma_n}} + \sum_{n=1}^{\infty}b_n\ \frac{\overrightarrow{\rm{curl}}_\Gamma Y_n}{\sqrt{\gamma_n}}.
\end{equation}
Given the Helmholtz decomposition~\eqref{eq:HelmholtzDecompF}, the definition of the operator $PS(\mathbf{n}\times\mathbf{S}_K)$ given in equation~\eqref{eq:defPSR}, and the spectral properties recounted in formulas~\eqref{eq:eig_vec_laplace_beltrami0}-\eqref{eq:eig_vec_laplace_beltrami1}, we obtain
$$\int_\Gamma (PS(\mathbf{n}\times\mathbf{S}_K)\mathbf{a})\cdot(\mathbf{n}\times\bar{\mathbf{a}})d\sigma=\frac{1}{2}\sum_{n=1}^{\infty}(\gamma_n-K^2)^{-\frac{1}{2}}(|a_n|^2+|b_n|^2).$$
On the other hand, given the definition of the operator $PS(\mathcal{T}^1_K\ {\rm div}_\Gamma)$ in~\eqref{eq:defPSR}, we obtain
$$\int_\Gamma PS(\mathcal{T}^1_K{\rm div}_\Gamma\mathbf{a})\cdot(\mathbf{n}\times\bar{\mathbf{a}})d\sigma=-\frac{1}{2}\sum_{n=1}^{\infty}(\gamma_n-K^2)^{-\frac{1}{2}}\gamma_n|b_n|^2.$$
Consequently, if we denote $\eta=\eta_R+i\eta_I$ and $\zeta=\zeta_R+i\zeta_I$ we get that 
\begin{eqnarray}
\Re \int_\Gamma(PS(\mathcal{R})\mathbf{a})\cdot \mathbf{n}\times\bar{\mathbf{a}}d\sigma&=&\frac{1}{2}\sum_{n=1}^\infty(\eta_R \Re (\gamma_n-K^2)^{-\frac{1}{2}} -\eta_I \Im (\gamma_n-K^2)^{-\frac{1}{2}})(|a_n|^2+|b_n|^2)\nonumber\\
&-&\frac{1}{2}\sum_{n=1}^\infty(\zeta_R \Re (\gamma_n-K^2)^{-\frac{1}{2}} -\zeta_I \Im (\gamma_n-K^2)^{-\frac{1}{2}})\gamma_n|b_n|^2.
\end{eqnarray}
Since $\gamma_n>0$ and $\Re K\geq 0$ and $\Im K>0$, then $\Re (\gamma_n-K^2)^{-\frac{1}{2}}>0$ and $\Im (\gamma_n-K^2)^{-\frac{1}{2}}\geq 0$, the coercivity of $\pm \Re(PS(\mathcal{R}))$ now follows, which completes the proof of the theorem.
\end{proof}

Another choice of regularizing operators that has been proposed in the literature~\cite{Darbas} consists of the following selection: $K=\kappa_1+i\kappa_2$, $\eta=-i\gamma\ K$, and $\zeta=-\gamma\ \frac{i}{K}$ with $\gamma>0$ in equations~\eqref{eq:defPSR}, that is 
\begin{equation}\label{eq:defPST}
\mathcal{R}=-\gamma\ PS(\mathcal{T}_{\kappa_1+i\kappa_2})=\gamma\ (\kappa_2-i\kappa_1)PS(\mathbf{n}\times \mathbf{S}_{\kappa_1+i\kappa_2}) -\gamma \frac{\kappa_2+i\kappa_1}{\kappa_1^2+\kappa_2^2}\ PS(\mathcal{T}^1_{\kappa_1+i\kappa_2}\ {\rm div}_\Gamma)
\end{equation}
 where $\kappa_1>0$ and $\kappa_2>0$. We note that these regularizing operators are outside the range of applicability of the result in Theorem~\ref{thm2} since $\Re \eta>0$, $\Im \eta<0$, $\Re \zeta<0$, but $\Im \zeta<0$. Nevertheless, the choice $\mathcal{R}=-\gamma\ PS(\mathcal{T}_K)$ also leads to uniquely solvable CFIER formulations, as given in the following
\begin{theorem}\label{thm3}
If we take the wavenumber $K$ in the definition of the regularizing operator $-PS(\mathcal{T}_K)$ defined in equation~\eqref{eq:defPST} such that $K=\kappa_1+i\kappa_2,\ 0<\kappa_1,\ 0<\kappa_2$, then the following property holds
\begin{equation}
\label{eq:param_PS3}
PS\mathcal{B}_{k,\gamma,\kappa_1,\kappa_2}=\frac{I}{2}-\mathcal{K}_k- 
\gamma\ \mathcal{T}_k\  PS(\mathcal{T}_K)\sim
\left(\frac{1}{2}+\frac{\gamma\ k}{4(\kappa_1+i\kappa_2)}\right)\Pi_{\nabla_\Gamma}+\left(\frac{1}{2}+\frac{\gamma(\kappa_1+i\kappa_2)}{4k}\right)\Pi_{\overrightarrow{\rm
    curl}_\Gamma}.\\
\end{equation}
Furthermore, the operators $PS\mathcal{B}_{k,\gamma,\kappa_1,\kappa_2}$ are invertible with bounded inverses in the spaces\\ $H_{\rm div}^{m}(\Gamma),\ m\geq 2$ provided that $\gamma>0$.
\end{theorem}
\begin{proof} The fact that the operators $\frac{I}{2}-\mathcal{K}_k-
\gamma\ \mathcal{T}_k\  PS(\mathcal{T}_K)$ has the property~\eqref{eq:param_PS3} was shown in Theorem~\ref{thm2}. The injectivity of these operators follows once we show that $\Re \int_\Gamma(PS(\mathcal{T}_K)\mathbf{a})\cdot \mathbf{n}\times\bar{\mathbf{a}}d\sigma>0$ for $\mathbf{a}\neq 0$. Assuming the Helmholtz decomposition~\eqref{eq:HelmholtzDecompF}, we get just as in the proof of Theorem~\ref{thm2} the following identity
\begin{eqnarray}
\Re \int_\Gamma(PS(\mathcal{T}_K)\mathbf{a})\cdot \mathbf{n}\times\bar{\mathbf{a}}d\sigma&=&\frac{1}{2}\sum_{n=1}^\infty(\epsilon\Re (\gamma_n-K^2)^{-\frac{1}{2}} + \kappa\Im (\gamma_n-K^2)^{-\frac{1}{2}})|a_n|^2\nonumber\\
&+&\frac{1}{2}\sum_{n=1}^\infty\alpha_n|b_n|^2\nonumber\\
\alpha_n&=&\epsilon\Re (\gamma_n-K^2)^{-\frac{1}{2}} +\kappa\Im (\gamma_n-K^2)^{-\frac{1}{2}})\nonumber\\
&+&\left(\frac{\epsilon}{\kappa^2+\epsilon^2}\Re (\gamma_n-K^2)^{-\frac{1}{2}}-\frac{\kappa}{\kappa^2+\epsilon^2}\Im (\gamma_n-K^2)^{-\frac{1}{2}}\right)\gamma_n.
\end{eqnarray}
A simple calculation gives
$$\alpha_n=\Re\left(\frac{i}{K}(\gamma_n-K^2)^{\frac{1}{2}}\right)=\frac{1}{\kappa^2+\epsilon^2}(\epsilon \Re (\gamma_n-K^2)^{\frac{1}{2}}-\kappa \Im(\gamma_n-K^2)^{\frac{1}{2}}).$$
Given that for $\gamma_n>0$ we have $\Re (\gamma_n-K^2)^{\frac{1}{2}}>0$ and $\Im(\gamma_n-K^2)^{\frac{1}{2}}<0$, it follows that $\alpha_n>0$ and hence $\Re \int_\Gamma(PS(\mathcal{T}_K)\mathbf{a})\cdot \mathbf{n}\times\bar{\mathbf{a}}d\sigma>0$ and the proof is complete.
\end{proof}
\begin{remark} The result established in Theorem~\ref{thm3} can be viewed as a principal symbol counterpart of that in Theorem~\ref{tm11}, but without any restrictions on the magnitude of the real and imaginary parts of the wavenumber $K$. This class of regularized formulations was considered in~\cite{Darbas}, yet without an analysis of their unique solvability.
\end{remark}

\section{Spectral properties of the CFIER operators for spherical
  scatterers\label{spec_prop_cfie}}

The selection of the various parameters that enter the CFIER formulations considered in the two previous sections (e.g. the complex wavenumber $\kappa_1+i\kappa_2$ and the coupling parameter $\gamma$ in the definition of the regularizing operators $\mathcal{R}$ and $PS\mathcal{R}$) is typically guided by optimizing the spectral properties of the CFIER operators in the case of spherical scatterers. We investigate in this section the spectral properties of the various CFIER operators considered in Section~\ref{cfie} and Section~\ref{pscfier} in the case when $\Gamma$ is spherical. It turns out that in this case the eigenvalues of the CFIER operators are related to eigenvalues of regularized combined field integral equations for 2D scattering problems with Dirichlet and Neumann boundary conditions. We recall first several results about spectra of combined field and regularized combined field integral operators for circular geometries in two dimensions. The scattering problem from the unit circle in two dimensions with Dirichlet boundary conditions consists of finding scattered fields $u^s$ that are solutions of 
\begin{eqnarray}
  \Delta u^s+k^2 u^s = 0\ \mbox{in}\ \{\mathbf{x}:|\mathbf{x}|> 1\}\nonumber\\
  u^s=-u^{\rm inc}\ \mbox{on}\ \mathbb{S}^1\nonumber\\
  \lim_{|r|\to\infty}r^{1/2}(\partial u^s/\partial r - iku^s)=0\nonumber,
\end{eqnarray}
where $u^{\rm inc}$ denotes the incident field. If one looks for the scattered field $u^s$ in the form
$$u^s(\mathbf{z})=\int_{\mathbb{S}^1}\frac{\partial G_k^{2D}(\mathbf z-\mathbf y)}{\partial \mathbf{n}(\mathbf y)}\phi(\mathbf y)ds(\mathbf y)-i\eta\int_{\mathbb{S}^1}G_k^{2D}(\mathbf z-\mathbf y)\phi(\mathbf y)ds(\mathbf{y})$$
where $G_k^{2D}(\mathbf z)=\frac{i}{4}H_0^{(1)}(k|\mathbf{z}|)$, then $\phi$ is the solution of the following boundary integral equation
\begin{equation}\label{eq:eigDir}
\mathcal{D}_k\phi=-u^{inc}\ {\rm on}\ \Gamma,\qquad \mathcal{D}_k\phi=\frac{\phi}{2}+K_k\phi-i\eta S_k\phi
\end{equation}
where the double layer operator $K_k$ is defined as $(K_k\phi)(\mathbf x)=\int_{\mathbb{S}^1}\frac{\partial G_k^{2D}(\mathbf x-\mathbf y)}{\partial \mathbf{n}(\mathbf y)}\phi(\mathbf y)ds(\mathbf y),\ |\mathbf x|=1$, and the single layer operator $S_k$ is defined as $(S_k\phi)(\mathbf x)=\int_{\mathbb{S}^1}G_k^{2D}(\mathbf x-\mathbf y)\phi(\mathbf y)ds(\mathbf{y}),\ |\mathbf x|=1$. The operator $\mathcal{D}_k$ is diagonalizable in the orthogonal basis $\{e^{im\theta}: m\in\mathbb{Z},\ \theta\in[0,2\pi)\}$ and the eigenvalues of the operator $\mathcal{D}_k$ are given by
\begin{equation}\label{eq:dm}
\mathcal{D}_k e^{im\theta}=d_m(k,\eta)e^{im\theta},\ m\in\mathbb{Z},\quad d_m(k,\eta)=\frac{i\pi k}{2}J'_{|m|}(k)H_{|m|}^{(1)}(k)+\frac{\eta\pi}{2}J_{|m|}(k)H_{|m|}^{(1)}(k).
\end{equation}
The scattering problem from the unit circle in two dimensions with Neumann boundary conditions consists of finding scattered fields $u^s$ that are solutions of 
\begin{eqnarray}
  \Delta u^s+k^2 u^s = 0\ \mbox{in}\ \{\mathbf{x}:|\mathbf{x}|> 1\}\nonumber\\
  \frac{\partial u^s}{\partial\mathbf{n}}=-\frac{\partial u^{\rm inc}}{\partial \mathbf{n}}\ \mbox{on}\ \mathbb{S}^1,\ \mathbf{n}=\mathbf{x}\nonumber\\
  \lim_{|r|\to\infty}r^{1/2}(\partial u^s/\partial r - iku^s)=0\nonumber,
\end{eqnarray}
where $u^{\rm inc}$ denotes the incident field. If one looks for the scattered field $u^s$ in the form~\cite{turc7}
$$u^s(\mathbf{z})=-\int_{\mathbb{S}^1}G_k^{2D}(\mathbf z-\mathbf y)\psi(\mathbf y)ds(\mathbf{y})+i\xi\int_{\mathbb{S}^1}\frac{\partial G_k^{2D}(\mathbf z-\mathbf y)}{\partial \mathbf{n}(\mathbf y)}(S_{ik}\psi)(\mathbf y)ds(\mathbf y)$$
then $\psi$ is the solution of the following boundary integral equation
$$\mathcal{N}_k\psi=-\frac{\partial u^{inc}}{\partial\mathbf{n}}\ {\rm on}\ \Gamma,\qquad \mathcal{N}_k\psi=\frac{\phi}{2}-K_k^T\psi+i\xi(N_k\  S_{ik})\psi$$
where the operator $K_k^T$ is defined as $(K_k^T\psi)(\mathbf x)=\int_{\mathbb{S}^1}\frac{\partial G_k^{2D}(\mathbf x-\mathbf y)}{\partial \mathbf{n}(\mathbf x)}\psi(\mathbf y)ds(\mathbf y),\ |\mathbf x|=1$, and the operator $N_k$ is defined as $(N_k\psi)(\mathbf x)=FP\int_{\mathbb{S}^1}\frac{\partial^2 G_k^{2D}(\mathbf x-\mathbf y)}{\partial\mathbf{n}(\mathbf{x})\partial\mathbf{n}(\mathbf{y})}\psi(\mathbf y)ds(\mathbf y),\ |\mathbf x|=1$, where FP stands for Hadamard Finite Parts integrals. The operator $\mathcal{N}_k$ is diagonalizable in the orthogonal basis $\{e^{im\theta}: m\in\mathbb{Z},\ \theta\in[0,2\pi)\}$ and the eigenvalues of the operator $\mathcal{N}_k$ are given by~\cite{turc7}
\begin{eqnarray}\label{eq:eigNeu}
\mathcal{N}_k e^{im\theta}&=&p_m(k,\xi)e^{im\theta},\ m\in\mathbb{Z},\nonumber\\
p_m(k,\xi)&=&1-\frac{i\pi k}{2}J'_{|m|}(k)H_{|m|}^{(1)}(k)+\frac{\xi \pi^2k^2}{4}J'_{|m|}(k)(H_{|m|}^{(1)}(k))'[iJ_{|m|}(ik)H_{|m|}^{(1)}(ik)].
\end{eqnarray}
It was proved in~\cite{Graham,turc7} that there exists a constant $C$ and $k_0$ large enough such that for all $k_0\leq k$ the following relations hold:
\begin{equation}\label{eq:first_bound}
|d_\nu(k,\eta)|\leq C(1+|\eta|k^{-2/3}),\qquad |p_\nu(k,\xi)|\leq C(1+|\xi|)\qquad {\rm for\ all}\ 0\leq \nu.
\end{equation}
Furthermore, if $\eta=k$ in equations~\eqref{eq:eigDir}, then the following coercivity property was established in~\cite{Graham}
\begin{equation}\label{eq:coercDir}
\Re(d_\nu(k,k))\geq \frac{1}{2}\ {\rm for\ all}\ 0\leq \nu,\ k_0\leq k.
\end{equation}
Also, if $\xi=k^{1/3}$ in equations~\eqref{eq:eigNeu}, then there exists a constant $C_0\approx 0.377$ such that~\cite{turc7}
\begin{equation}\label{eq:coercNeu}
\Re(p_\nu(k,k^{1/3}))\geq C_0\ {\rm for\ all}\ 0\leq \nu,\ k_0\leq k.
\end{equation}
\begin{remark}\label{remark:coerc}
It was also established in~\cite{turc7} that by selecting $\eta=1$ in equations~\eqref{eq:eigDir}, the real part of the eigenvalues $p_\nu(k,1)$ is still positive for sufficiently large values of $k$, but the corresponding lower bounds are proportional to $k^{-1/3}$ in that case. 
\end{remark}
Having reviewed the spectral properties of combined field integral operators for 2D scattering problems in circular geometries, we return to the case of eigenvalues of the electromagnetic scattering boundary integral operators. For a spherical scatterer of radius one, the spectral properties of the electromagnetic boundary integral operators can be expressed in terms of gradients and rotationals of the spherical harmonics $(Y_n^{m})_{0\leq n,-n\leq m \leq n}$. The system $(Y_n^{m})_{0\leq n,-n\leq m \leq n}$ constitutes an orthonormal basis of the space $L^2(\mathbb{S}^2)$. The system $(\nabla_{\mathbb{S}^2} Y_n^{m}, \overrightarrow{\rm{curl}}_{\mathbb{S}^2} Y_n^{m})_{1\leq n,-n\leq m \leq n}$ forms an orthogonal basis of the space $L^2(TM(\mathbb{S}^2))$. Here $\nabla_{\mathbb{S}^2}$ denotes the surface gradient on the sphere ${\mathbb{S}^2}$, and for a scalar-valued function $b$ defined on
${\mathbb{S}^2}$, $\overrightarrow{\rm{curl}}_{\mathbb{S}^2}\, b = (\nabla_{\mathbb{S}^2}\, b)\times\mathbf{n}$. Furthermore, the system  $\left(\frac{\nabla_{\mathbb{S}^2} Y_n^{m}}{\sqrt{n(n+1)}}, \frac{\overrightarrow{\rm{curl}}_{\mathbb{S}^2} Y_n^{m}}{\sqrt{n(n+1)}}\right)_{1\leq n,-n\leq m \leq n}$ forms an orthonormal basis of the space $L^2(TM(\mathbb{S}^2))$. The spectral properties of the operators $\mathcal{K}_K$ and $\mathcal{T}_K$ for a wavenumber $K$ are recounted below~\cite{Kress1985}:
\begin{eqnarray}
\label{eq:eigK}
\mathcal{K}_K(\nabla_{\mathbb{S}^2} Y_n^{m})&=&-\lambda_n(K) \nabla_{\mathbb{S}^2} Y_n^{m},\\
\mathcal{K}_K(\overrightarrow{\rm{curl}}_{\mathbb{S}^2} Y_n^{m})&=&\lambda_n(K) \overrightarrow{\rm{curl}}_{\mathbb{S}^2} Y_n^{m},
\end{eqnarray} 
and
\begin{eqnarray}
\label{eq:eigT}
\mathcal{T}_K(\nabla_{\mathbb{S}^2} Y_n^{m})&=&\Lambda_n^{(1)}(K)\overrightarrow{\rm{curl}}_{\mathbb{S}^2} Y_n^{m},\\
\label{eq:eigT2}\mathcal{T}_K(\overrightarrow{\rm{curl}}_{\mathbb{S}^2} Y_n^{m})&=&
\Lambda_n^{(2)}(K) \nabla_{\mathbb{S}^2} Y_n^{m},
\end{eqnarray} 
where
\begin{equation}
\label{eq:lambda_n}
\lambda_n(K)=\frac{iK}{2}\{j_n(K)[Kh_n^{(1)}(K)]'+h_n^{(1)}(K)[Kj_n(K)]'\},
\end{equation}
\begin{equation}
\label{eq:Lambda_n_1}
\Lambda_n^{(1)}(K)=[Kj_n(K)]'[Kh_n^{(1)}(K)]',
\end{equation}
and
\begin{equation}
\label{eq:Lambda_n_2}
\Lambda_n^{(2)}(K)=-K^{2}j_n(K)h_n^{(1)}(K).
\end{equation}
In the formulas above $j_n$ and $h_n^{(1)}$ denote the spherical Bessel and Hankel functions of the
first kind, respectively. From the spectral properties of the operators $\mathcal{T}_K$ presented above, the spectral properties of the operator $\mathbf{n}\times \mathbf{S}_K$ can be derived immediately. Indeed, it follows from the definition (\ref{eq:kernelEFIE}) (with $k = K$) that
\begin{equation}
\label{eq:eig_Svec_k_1}
(\mathbf{n}\times \mathbf{S}_K)(\overrightarrow{{\rm curl}}_{\mathbb{S}^2} Y_n^{m})=\frac{\Lambda_n^{(2)}(K)}{iK}\nabla_{\mathbb{S}^2} Y_n^{m},
\end{equation}
and, using the definition of the Laplace-Beltrami operator
$\Delta_{\mathbb{S}^2}=\rm{div}_{\mathbb{S}^2} \nabla_{\mathbb{S}^2}$ and the fact that the spherical
harmonic $Y_n^{m}$ is an eigenfunction of both the operator
$\Delta_{\mathbb{S}^2}$, with eigenvalue $-n(n+1)$, as well as the single layer
acoustic operator corresponding to the wavenumber $K$~\cite{Nedelec}
\begin{equation}
\label{eq:s_k}
(S_K Y_n^{m})(\mathbf{x})=\int_{\mathbf{{\mathbb{S}^2}}} G_K(\mathbf{x}-\mathbf{y})Y_n^{m}(\mathbf{y})d\sigma(\mathbf{y})=iKj_n(K)h_n^{(1)}(K)Y_n^{m}(\mathbf{x}),\ |\mathbf{x}|=1
\end{equation}
we obtain 
\begin{equation}
\label{eq:eig_Svec_k_2}
(\mathbf{n}\times \mathbf{S}_K)(\nabla_{\mathbb{S}^2} Y_n^{m})=\frac{1}{iK}(\Lambda_n^{(1)}(K)+n(n+1)j_n(K)h_n^{(1)}(K))\overrightarrow{{\rm curl}}_{\mathbb{S}^2} Y_n^{m}=S_n^{(1)}(K)\overrightarrow{{\rm curl}}_{\mathbb{S}^2} Y_n^{m}.
\end{equation}
Consequently, the choice of the regularizing operators $\mathcal{R}=\eta\ \kappa\ \mathbf{n}\times \mathbf{S}_{i\kappa},\ \kappa>0$ advocated in~\cite{turc1} leads to boundary integral operators with the following spectral properties
\begin{eqnarray}
\left(\frac{I}{2}-\mathcal{K}_k+\eta \kappa\mathcal{T}_k\ (\mathbf{n}\times \mathbf{S}_{i\kappa})\right)\nabla_{\mathbb{S}^2} Y_n^{m}&=&\left(\frac{1}{2}+\lambda_n(k)+\eta\kappa \Lambda_n^{(2)}(k) S_n^{(1)}(i\kappa)\right)\nabla_{\mathbb{S}^2} Y_n^{m}\nonumber\\
&=&A_n^{(1)}(k,\kappa,\eta)\nabla_{\mathbb{S}^2} Y_n^{m}\nonumber\\
\left(\frac{I}{2}-\mathcal{K}_k+\eta \kappa \mathcal{T}_k\ (\mathbf{n}\times \mathbf{S}_{i\kappa})\right)\overrightarrow{\rm{curl}}_{\mathbb{S}^2}Y_n^m &=&\left(\frac{1}{2}-\lambda_n(k)-\eta\Lambda_n^{(1)}(k) \Lambda_n^{(2)}(i\kappa)\right) \overrightarrow{\rm{curl}}_{\mathbb{S}^2} Y_n^{m}\nonumber\\
&=&A_n^{(2)}(k,\kappa,\eta)\overrightarrow{\rm{curl}}_{\mathbb{S}^2} Y_n^{m}.
\end{eqnarray} 
On the other hand, the choice of the regularizing operators $\mathcal{R}=-\xi\mathcal{T}_{i\kappa},\ \kappa>0,\ \xi>0$ introduced in~\cite{Contopa_et_al} leads to boundary integral operators with the following spectral properties
\begin{eqnarray}
\left(\frac{I}{2}-\mathcal{K}_k-\xi \mathcal{T}_k\ \mathcal{T}_{i\kappa}\right)\nabla_{\mathbb{S}^2} Y_n^{m}&=&\left(\frac{1}{2}+\lambda_n(k)-\xi \Lambda_n^{(2)}(k) \Lambda_n^{(1)}(i\kappa)\right)\nabla_{\mathbb{S}^2} Y_n^{m}\nonumber\\
&=&B_n^{(1)}(k,\kappa,\xi)\nabla_{\mathbb{S}^2} Y_n^{m}\nonumber\\
\left(\frac{I}{2}-\mathcal{K}_k-\xi \mathcal{T}_k\ \mathcal{T}_{i\kappa}\right)\overrightarrow{\rm{curl}}_{\mathbb{S}^2}Y_n^m &=&\left(\frac{1}{2}-\lambda_n(k)-\xi \Lambda_n^{(1)}(k) \Lambda_n^{(2)}(i\kappa)\right) \overrightarrow{\rm{curl}}_{\mathbb{S}^2} Y_n^{m}\nonumber\\
&=&B_n^{(2)}(k,\kappa,\xi)\overrightarrow{\rm{curl}}_{\mathbb{S}^2} Y_n^{m}\nonumber.\\
\end{eqnarray} 
We note that the operators $\left(\frac{I}{2}-\mathcal{K}_k+\eta \kappa \mathcal{T}_k\ (\mathbf{n}\times \mathbf{S}_{i\kappa})\right)$ and $\left(\frac{I}{2}-\mathcal{K}_k-\xi \mathcal{T}_k\ \mathcal{T}_{i\kappa}\right)$ are diagonal in the orthonormal basis $\left(\frac{\nabla_{\mathbb{S}^2} Y_n^{m}}{\sqrt{n(n+1)}}, \frac{\overrightarrow{\rm{curl}}_{\mathbb{S}^2} Y_n^{m}}{\sqrt{n(n+1)}}\right)_{1\leq n,-n\leq m \leq n}$ of the space $L^2(TM(\mathbb{S}^2))$, having the same eigenvalues $(A_n^{(1)}(k,\kappa,\eta),A_n^{(2)}(k,\kappa,\eta)$ and $(B_n^{(1)}(k,\kappa,\xi),B_n^{(2)}(k,\kappa,\xi))$ in that basis. We investigate next the behavior of the eigenvalues $(A_n^{(1)}(k,\kappa,\eta),A_n^{(2)}(k,\kappa,\eta)$ and $(B_n^{(1)}(k,\kappa,\xi),B_n^{(2)}(k,\kappa,\xi))$ for suitable choices of the parameters $\kappa$, $\eta$, and $\xi$. The numerical evidence suggests that the choice $\kappa=ck$ (where $c=1/2$ or $c=1$) leads to boundary integral operators with very good spectral properties~\cite{turc1} and~\cite{Contopa_et_al}. We will consider next therefore $\kappa=k$---the general case $\kappa=ck$ with $c$ being a constant independent of $k$ can be treated analogously. We start by stating a useful result whose proof follows the same lines as the proof of Lemma 3.1 in~\cite{turc7} and it will be presented in the Appendix:
\begin{lemma}\label{ImKm} 
There exist constants $C_j>0,j=1,\ldots,4$ and a number $\tilde{k}_0>0$ such that
$$ (i)\ C_1 k^{-2}(n^2+k^2)^{1/2}\leq iJ'_{n+1/2}(ik)(H_{n+1/2}^{(1)})'(ik)\leq C_2k^{-2}(n^2+k^2)^{1/2}$$
$$ (ii)\ \frac{1}{4}(n^2+k^2)^{-1/2}\leq -S_n^{(1)}(ik)\leq C_3(n^2+k^2)^{-1/2}+C_3k^{-2} $$
$$ (iii) |J'_{n+1/2}(ik)H_{n+1/2}^{(1)}(ik)|\leq C_4k^{-1},\ |J_{n+1/2}(ik)(H_{n+1/2}^{(1)})'(ik)|\leq C_4k^{-1}$$ 
for all $k>\tilde{k}_0$ and all $n\geq 0$.
\end{lemma}

We investigate next the properties of the eigenvalues $(A_n^{(1)}(k,k,\eta),A_n^{(2)}(k,k,\eta))_{1\leq n}$. We have 
\begin{eqnarray}\label{eq:eigA}
A_n^{(1)}(k,k,\eta)&=&\frac{i\pi k}{2}J'_{n+1/2}(k)H_{n+1/2}^{(1)}(k)-\frac{\eta\pi k^2}{2}J_{n+1/2}(k)H_{n+1/2}^{(1)}(k)S_n^{(1)}(ik)+a_n^{(1)}(k)\nonumber\\
A_n^{(2)}(k,k,\eta)&=&1-\frac{i\pi k}{2}J'_{n+1/2}(k)H_{n+1/2}^{(1)}(k)+\frac{\eta\pi^2k^2}{4}J'_{n+1/2}(k)(H_{n+1/2}^{(1)})'(k)[iJ_{n+1/2}(ik)H_{n+1/2}^{(1)}(ik)]\nonumber\\
&+&a_n^{(2)}(k,\eta)=p_{n+1/2}(k,\eta)+a_n^{(2)}(k,\eta)\nonumber\\
\end{eqnarray}
where 
\begin{eqnarray}
a_n^{(1)}(k)&=&\frac{i\pi}{4}J_{n+1/2}(k)H_{n+1/2}^{(1)}(k)\nonumber\\
a_n^{(2)}(k,\eta)&=&-\frac{i\pi}{4}J_{n+1/2}(k)H_{n+1/2}^{(1)}(k)+\frac{\eta\pi^2}{4}[iJ_{n+1/2}(ik)H_{n+1/2}^{(1)}(ik)]\nonumber\\
&\times&[\frac{1}{4}J_{n+1/2}(k)H_{n+1/2}^{(1)}(k)+\frac{k}{2}J_{n+1/2}(k)(H_{n+1/2}^{(1)})'(k)+\frac{k}{2}J'_{n+1/2}(k)H_{n+1/2}^{(1)}(k)]
\end{eqnarray}
and $p_\nu(k,\xi)$ are defined in equations~\eqref{eq:eigNeu}. We use the following two estimates which were established in~\cite{BanjaiSauter} (Proposition 3.10): there exists a constant $C>0$ independent of $k$ such that $|J_\nu(k)H_\nu^{(1)}(k)|\leq Ck^{-\frac{2}{3}}$ and $|kJ_\nu'(k)H_\nu^{(1)}(k)|\leq C$ for sufficiently large $k$ and for all $\nu\geq 0$. We obtain immediately that $|a_n^{(1)}(k)|\leq C k^{-2/3}$ for all $1\leq n$ and $k$ sufficiently large. Using the Wronskian identity $J_\nu(k)Y'_\nu(k)-J'_\nu(k)Y_\nu(k)=\frac{2}{\pi k}$ we get that $|kJ_\nu(k)(H_\nu^{(1)})'(k)|\leq C$ as well. We established in~\cite{turc7}, that there exists a constant $C$ independent of $k$ such that $0 < iJ_\nu(ik)H_\nu^{(1)}(ik)\leq C (\nu^2 + k^2)^{-1/2}$ for sufficiently large $k$ and all $\nu\geq 0$. Using this result together with the estimates recounted above and estimate (ii) in Lemma~\ref{ImKm}  we get that 
\begin{theorem}\label{upper_bd_1}
There exists a constant $C_5>0$ and a wavenumber $k_4$ sufficiently large so that
$$|A_n^{(1)}(k,k,\eta)|\leq C_5|\eta| k^{1/3}\qquad |A_n^{(2)}(k,k,\eta)|\leq C_5(1+|\eta|)\ {\rm for\ all}\ k_4\leq k,\ {\rm and\ all}\ 1\leq n.$$
\end{theorem}
The same arguments that were used to derive the result in Theorem~\ref{upper_bd_1} also lead to the following estimate: there exists a constant $C_6$ independent of $k$ and a wavenumber $k_5$ sufficiently large such that
\begin{equation}\label{eq:est_an2}
|a_n^{(2)}(k,\eta)|\leq C_6 k^{-\frac{2}{3}}(1+|\eta|k^{-\frac{1}{3}})\quad {\rm for\ all}\ k_5\leq k,\ 1\leq n.
\end{equation}
The next result concerns the positivity of the real part of the eigenvalues $A_n^{(1)}(k,k,\eta)$. Given formulas~\eqref{eq:eigA} and the previously established estimates~\eqref{eq:coercNeu} and~\eqref{eq:est_an2}, we consider the choice $\eta=k^{1/3}$ in the definition of $A_n^{(j)}(k,k,\eta),j=1,2$. In this case, it turns out that the positivity of the eigenvalues $A_n^{(1)}(k,k,k^{1/3})$ is a consequence of the positivity of the Dirichlet eigenvalues $d_n(k,k)$ defined in equation~\eqref{eq:dm}, while the positivity of the eigenvalues $A_n^{(2)}(k,k,k^{1/3})$ is a consequence of the positivity of the Neumann eigenvalues $p_n(k,k^{1/3})$ defined in equation~\eqref{eq:eigNeu}. We begin with the following result:
\begin{theorem}\label{thm_coercivity_0}
There exists a wavenumber $\tilde{k}$ large enough so that for all $k\geq \tilde{k}$ and all $1\leq n$ the following estimate holds:
$$\Re(A_n^{(1)}(k,k,k^{1/3}))\geq \frac{1}{2}.$$
\end{theorem}
\begin{proof} Using the estimate (ii) in Lemma~\ref{ImKm}, the previously established estimate $|a_n^{(1)}(k)|\leq C k^{-2/3}$ for all $1\leq n$ and $k$ sufficiently large, and the fact that $|J_\nu(k)|\leq Ck^{-1/3},\ 0\leq \nu$ and $k$ sufficiently large~\cite{BanjaiSauter}, we get that there exists a number $\tilde{k}_1$ such that 
$$\Re(A_n^{(1)}(k,k,k^{1/3}))\geq \tilde{d}_{n}(k)-Ck^{-2/3},\ {\rm for\ all}\ k\geq \tilde{k}_1$$
where 
$$\tilde{d}_{n}(k)=-\frac{\pi k}{2}J'_{n+1/2}(k)Y_{n+1/2}(k)+\frac{\pi k^{7/3}}{8\sqrt{n^2+k^2}}J^2_{n+1/2}(k),\quad 1\leq n.$$ 
We will establish that $\tilde{d}_{n}(k)\geq\frac{1}{2}$ for all $\tilde{k}\leq k$, and all $1\leq n$. Using the Wronskian identity $Y'_\nu(k)J_\nu(k)-Y_\nu(k)J'_\nu(k)=\frac{2}{\pi k}$, we can re-express $\tilde{d}_n(k)$ as
$$\tilde{d}_{n}(k)=\frac{1}{2}+\frac{\pi k}{2}\left(\frac{k^{4/3}}{4\sqrt{n^2+k^2}}J^2_{n+1/2}(k)-\frac{1}{2}(J_{n+1/2}(k)Y_{n+1/2}(k))'\right).$$
It was established in~\cite{Graham} (Lemma 4.5) that there exists $\delta>0$ and $m_0>0$ such that for all $m\geq m_0$ we have $(J_m(k)Y_m(k))'<0$ for all $k\in(0,m-\delta m^{1/3})$. It follows immediately then that for all $n\geq \max\{m_0-1/2,0\}$ we have
\begin{equation}\label{eq:est_T_1}
\tilde{d}_{n}(k)\geq \frac{1}{2},\quad k\in(0,n+1/2-\delta (n+1/2)^{1/3}).
\end{equation}
It remains to consider two cases, that is {\em Case} 1: $n<m_0-1/2$ and $k$ large enough, and {\em Case} 2: $\max\{m_0-1/2,0\}\leq n$ and $k\geq (n+1/2)-\delta (n+1/2)^{1/3}$. In both cases we will compare $\tilde{d}_n(k)$ with $\Re(d_{n+1/2}(k,k))$, where $d_m(k,k)$ was defined in equation~\eqref{eq:dm}. A simple application of the Wronskian identity for Bessel functions allows us to rewrite $\Re(d_{n+1/2}(k,k))$ in the following form
$$\Re(d_{n+1/2}(k,k))=\frac{1}{2}+\frac{\pi k}{2}\left(J^2_{n+1/2}(k)-\frac{1}{2}(J_{n+1/2}(k)Y_{n+1/2}(k))'\right).$$
{\em Case} 1: $n<m_0-1/2$ and $k$ large enough: We require that in this case $\tilde{d}_{n}(k)\geq \Re(d_{n+1/2}(k,k))$, which is equivalent to $4\sqrt{n^2+k^2}\leq k^{4/3}$. For the last inequality, it suffices to require that $4\sqrt{m_0^2+k^2}\leq k^{4/3}$. It follows immediately that $4\sqrt{m_0^2+k^2}\leq k^{4/3}$ provided that $k\geq\max\{m_0,2^{15/2}\}$.\\\\
{\em Case} 2: $\max\{m_0-1/2,0\}\leq n$ and $k\geq (n+1/2)-\delta (n+1/2)^{1/3}$: Again we require that $4\sqrt{n^2+k^2}\leq k^{4/3}$ in this case. Clearly we can choose $m_0$ large enough so that if both $\max\{m_0-1/2,0\}\leq n$ and $(n+1/2)-\delta (n+1/2)^{1/3}\leq k$ hold, then it implies that $\frac{n}{2}\leq (n+1/2)-\delta (n+1/2)^{1/3}\leq k$ and 
$$16(n^2+k^2)\leq 16(4k^2 + k^2) =80k^2\leq k^{8/3}.$$
In conclusion, in both {\em Case} 1 and {\em Case} 2 we have
\begin{equation}\label{eq:est_T_2}
\tilde{d}_n(k)\geq \Re(d_{n+1/2}(k,k)).
\end{equation}
 Given that $\Re(d_\nu(k,k))\geq\frac{1}{2}$ for all large enough $k$ and all $\nu\geq 0$~\cite{Graham}, the result of the Theorem now follows if we take into account estimates~\eqref{eq:est_T_1} and~\eqref{eq:est_T_2}.
\end{proof}

Having established the positivity of $\Re(A_n^{(1)}(k,k,k^{1/3}))$, we establish next the positivity of\newline$\Re(A_n^{(2)}(k,k,k^{1/3}))$:
\begin{theorem}\label{thm_coercivity_00}
There exists a constant $C_7$ and a wavenumber $k'$ such that for all $k\geq k'$ we have that
$$\Re(A_n^{(2)}(k,k,k^{1/3}))\geq C_7\ {\rm for\ all}\ n\geq 1.$$
\end{theorem}
\begin{proof} It follows from the definition of $A_n^{(2)}(k,k,k^{1/3})$ given in equations~\eqref{eq:eigA} that\newline
 $\Re(A_n^{(2)}(k,k,k^{1/3}))\geq \Re(p_{n+1/2}(k,k^{1/3}))-|a_n^{(2)}(k,k^{1/3})|$, which implies that there exists a wavenumber $k'$  such that $\Re(A_n^{(2)}(k,k,k^{1/3}))\geq \frac{C_0}{2}=C_7$ for all $n\geq 1$ and all $k\geq k'$ if we take into account the estimates~\eqref{eq:coercNeu} and~\eqref{eq:est_an2}. 
\end{proof}

Given the fact that $\left(\frac{\nabla_{\mathbb{S}^2} Y_n^{m}}{\sqrt{n(n+1)}}, \frac{\overrightarrow{\rm{curl}}_{\mathbb{S}^2} Y_n^{m}}{\sqrt{n(n+1)}}\right)_{1\leq n,-n\leq m \leq n}$ constitute an orthonormal basis for the space $L^2(TM(\mathbb{S}^2))$ (and an orthogonal basis in the spaces $H^{m}_{\rm div}(\mathbb{S}^2)$ for all $2\leq m$), we derive from the results in Theorem~\ref{upper_bd_1}, Theorem~\ref{thm_coercivity_0}, and Theorem~\ref{thm_coercivity_00} the following
\begin{corollary} Let $\mathcal{A}=\frac{I}{2}-\mathcal{K}_k+k^{4/3} 
\mathcal{T}_k\ (\mathbf{n}\times\mathbf{S}_{ik})$. There exists a constant $C_A$ independent of $k$ and $m$ such that for all sufficiently large $k$ and all $0\leq m$
$$\|\mathcal{A}\|_{H^{m}_{\rm div}(\mathbb{S}^2)\to H^{m}_{\rm div}(\mathbb{S}^2)}\|\mathcal{A}^{-1}\|_{H^{m}_{\rm div}(\mathbb{S}^2)\to H^{m}_{\rm div}(\mathbb{S}^2)}\leq C_A k^{2/3}.$$
\end{corollary}

If we use $PS(\mathbf{n}\times\mathbf{S}_{ik})$ instead of $\mathbf{n}\times\mathbf{S}_{ik}$ as regularizing operators in the CFIER formulations~\eqref{eq:CFIE_R} we obtain boundary integral operators $PS\mathcal{A}=\frac{I}{2}-\mathcal{K}_k+k^{4/3} 
\mathcal{T}_k\ (PS(\mathbf{n}\times\mathbf{S}_{ik}))$ whose spectral properties are qualitatively similar to those of the operators $\mathcal{A}$ defined above. Indeed, we have
\begin{eqnarray}
\left(\frac{I}{2}-\mathcal{K}_k+k^{4/3} \mathcal{T}_k\ (PS(\mathbf{n}\times \mathbf{S}_{ik}))\right)\nabla_{\mathbb{S}^2} Y_n^{m}&=&\left(\frac{1}{2}+\lambda_n(k)-\frac{k^{4/3} \Lambda_n^{(2)}(k)}{2(n(n+1)+k^2)^{1/2}}\right)\nabla_{\mathbb{S}^2} Y_n^{m}\nonumber\\
&=&PSA_n^{(1)}(k,k,k^{1/3})\nabla_{\mathbb{S}^2} Y_n^{m}\nonumber\\
\left(\frac{I}{2}-\mathcal{K}_k+k^{4/3} \mathcal{T}_k\ (PS(\mathbf{n}\times \mathbf{S}_{ik}))\right)\overrightarrow{\rm{curl}}_{\mathbb{S}^2}Y_n^m &=&\left(\frac{1}{2}-\lambda_n(k)+\frac{k^{4/3}\Lambda_n^{(1)}(k)}{2(n(n+1)+k^2)^{1/2}}\right) \overrightarrow{\rm{curl}}_{\mathbb{S}^2} Y_n^{m}\nonumber\\
&=&PSA_n^{(2)}(k,k,k^{1/3})\overrightarrow{\rm{curl}}_{\mathbb{S}^2} Y_n^{m}.
\end{eqnarray} 
Similarly to equations~\eqref{eq:eigA} we get
\begin{eqnarray}\label{eq:eigPSA}
PSA_n^{(1)}(k,k,k^{1/3})&=&\frac{i\pi k}{2}J'_{n+1/2}(k)H_{n+1/2}^{(1)}(k)+\frac{\pi k^{7/3}J_{n+1/2}(k)H_{n+1/2}^{(1)}(k)}{4(n(n+1)+k^2)^{1/2}}+a_n^{(1)}(k)\nonumber\\
PSA_n^{(2)}(k,k,k^{1/3})&=&1-\frac{i\pi k}{2}J'_{n+1/2}(k)H_{n+1/2}^{(1)}(k)+\frac{\pi k^{7/3}J'_{n+1/2}(k)(H_{n+1/2}^{(1)})'(k)}{4(n(n+1)+k^2)^{1/2}}\nonumber\\
&+&PSa_n^{(2)}(k,k)\nonumber\\
\end{eqnarray}
where 
\begin{eqnarray}
PSa_n^{(2)}(k,\eta)&=&-\frac{i\pi}{4}J_{n+1/2}(k)H_{n+1/2}^{(1)}(k)+\frac{k^{1/3}\pi}{4(n(n+1)+k^2)^{1/2}}\nonumber\\
&\times&[\frac{1}{4}J_{n+1/2}(k)H_{n+1/2}^{(1)}(k)+\frac{k}{2}J_{n+1/2}(k)(H_{n+1/2}^{(1)})'(k)+\frac{k}{2}J'_{n+1/2}(k)H_{n+1/2}^{(1)}(k)].\nonumber\\
\end{eqnarray}
Given the results in Lemma~\ref{ImKm}, it follows immediately that qualitatively similar results on upper bounds and coercivity properties hold for the eigenvalues $(PSA_n^{(1)}(k,k,k^{1/3}),PSA_n^{(2)}(k,k,k^{1/3}))_{1\leq n}$ and the eigenvalues $(A_n^{(1)}(k,k,k^{1/3}),A_n^{(2)}(k,k,k^{1/3}))_{1\leq n}$; in particular, the condition number of the integral equation formulations based on the operators $PS\mathcal{A}$ grows like $k^{2/3}$ in the high-frequency regime, just like those based on the operators $\mathcal{A}$.
\begin{remark}
In the light of the result mentioned in Remark~\ref{remark:coerc} and the results established in Theorem~\ref{upper_bd_1}, Theorem~\ref{thm_coercivity_0}, and Theorem~\ref{thm_coercivity_00}, we can derive that the same asymptotic bounds would hold for the condition numbers of the CFIER formulations based on the operators $\frac{I}{2}-\mathcal{K}_k+k \mathcal{T}_k\ (\mathbf{n}\times\mathbf{S}_{ik/2})$ and $\frac{I}{2}-\mathcal{K}_k+k \mathcal{T}_k\ (PS(\mathbf{n}\times\mathbf{S}_{ik/2}))$.
\end{remark}

We investigate next the properties of the eigenvalues blue$(B_n^{(1)}(k,k,\xi),B_n^{(2)}(k,k,\xi))_{0\leq n}$. We have
\begin{eqnarray}\label{eq:eigB}
B_n^{(1)}(k,k,\xi)&=&\frac{i\pi k}{2}J'_{n+1/2}(k)H_{n+1/2}^{(1)}(k)+\frac{\xi\pi^2k^2}{4}J_{n+1/2}(k)H_{n+1/2}^{(1)}(k)[iJ'_{n+1/2}(ik)(H_{n+1/2}^{(1)})'(ik)]\nonumber\\
&+&b_n^{(1)}(k,\xi)\nonumber\\
B_n^{(2)}(k,k,\xi)&=&A_n^{(2)}(k,k,\xi)
\end{eqnarray}
where  
\begin{eqnarray}
b_n^{(1)}(k,\xi)&=&\frac{i\pi}{4}J_{n+1/2}(k)H_{n+1/2}^{(1)}(k)+\frac{\xi\pi^2}{4i}J_{n+1/2}(k)H_{n+1/2}^{(1)}(k)\nonumber\\
&\times&[\frac{1}{4}J_{n+1/2}(ik)H_{n+1/2}^{(1)}(ik)+\frac{ik}{2}J_{n+1/2}(ik)(H_{n+1/2}^{(1)})'(ik)+\frac{ik}{2}J'_{n+1/2}(ik)H_{n+1/2}^{(1)}(ik)]\nonumber\\
\end{eqnarray}
and $p_\nu(k,\xi)$ are defined in equations~\eqref{eq:eigNeu}. We make use again of the estimate established in~\cite{BanjaiSauter} (Proposition 3.10): there exists a constant $C>0$ independent of $k$ such that $|J_\nu(k)H_\nu^{(1)}(k)|\leq Ck^{-\frac{2}{3}}$ for sufficiently large $k$ and for all $\nu\geq 0$.  We established in~\cite{turc7}, that there exists a constant $C$ independent of $k$ such that $0 < iJ_\nu(ik)H_\nu^{(1)}(ik)\leq C (\nu^2 + k^2)^{-1/2}\leq C k^{-1}$ for sufficiently large $k$ and all $\nu\geq 0$. Using these two results in conjunction with estimates (iii) in Lemma~\ref{ImKm} we get that there exist $C>0$ and $k_6>0$ such that
\begin{equation}\label{eq:est_ans}
|b_n^{(1)}(k,\xi)|\leq C k^{-\frac{2}{3}}(1+|\xi|)\ {\rm for\ all}\ k\geq k_6,\ 0\leq n.
\end{equation}

We investigate next the possibility to find upper bounds for the quantities $|B_n^{(1)}(k,k,\xi)|$ for all values of $n$ and large enough values of the wavenumber $k$. We make use one more time of the estimate established in~\cite{BanjaiSauter} (Proposition 3.10): there exists a constant $C>0$ independent of $k$ such that $|kJ_\nu'(k)H_\nu^{(1)}(k)|\leq C$ for sufficiently large $k$ and for all $\nu\geq 0$. If we further take into account estimates (i) in Lemma~\ref{ImKm}, we see that upper bounds for the quantities $|B_n^{(1)}(k,k,\xi)|$ for all values of $n$ and large enough values of the wavenumber $k$ can be obtained once upper bounds for the expressions
$$b_\nu(k)=|J_{\nu}(k)H_\nu^{(1)}(k)|\sqrt{\nu^2+k^2}$$
were established for all values of $\nu\geq 0$ and large enough $k$. Given the estimate $|J_\nu(k)H_\nu^{(1)}(k)|\leq Ck^{-\frac{2}{3}}$ for sufficiently large $k$ and for all $\nu\geq 0$, it follows immediately that for large enough values of $k$ we have
$$b_\nu(k)\leq C\sqrt{2}\ k^{1/3},\ \nu\leq k.$$
We plot in Figure~\ref{fig:max_b_nu_k} the values of $k^{-1/3}\max_{k\leq\nu}b_\nu(k)$ for 5041 values of $k$ from $k=8$ to $k=512$. In view of the estimate above and the results illustrated in Figure~\ref{fig:max_b_nu_k}, we are led to the following heuristic result
$$b_\nu(k)\leq C k^{1/3},\ {\rm for\ all}\ k_7\leq k,\ 0\leq \nu$$
which would in turn imply that there exists a constant $C_8$ such that
\begin{equation}\label{eq:est_uBn1}
|B_n^{(1)}(k,k,\xi)|\leq C_8|\xi| k^{1/3}\ {\rm for\ all}\ k_7\leq k,\ {\rm and\ all}\ 0\leq n.
\end{equation}
\begin{figure}
\centering
\includegraphics[height=65mm]{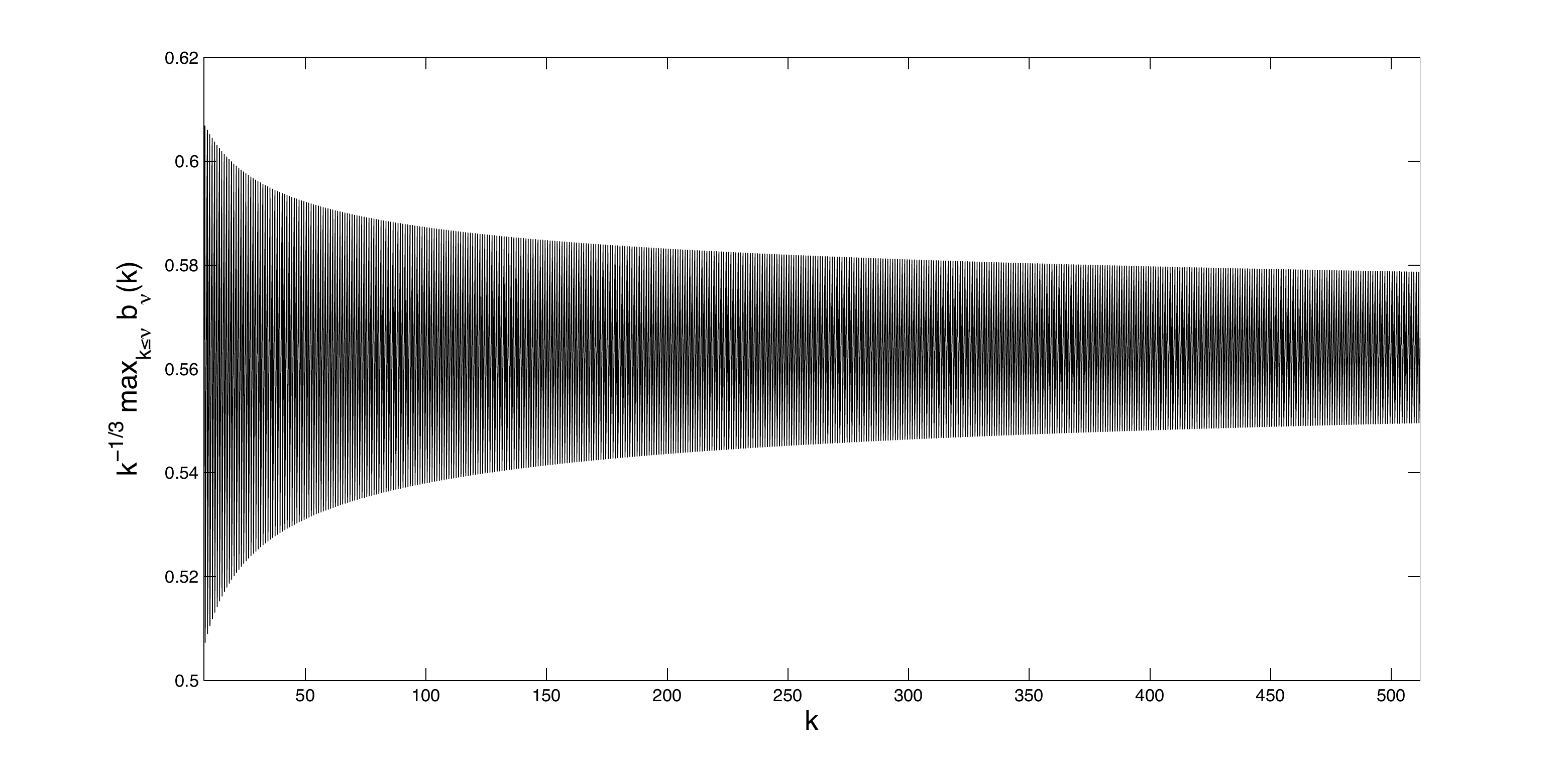}
\caption{Plot of $k^{-1/3}\max_{k\leq\nu}b_\nu(k)$ for 5041 values of $k$ from $k=8$ to $k=512$.}
\label{fig:max_b_nu_k}
\end{figure}
Although a rigorous proof of the heuristic results above is outside the scope of this paper, we give further asymptotic evidence in two important regimes ($\nu\sim k,\ \nu\to\infty$ and $k$ fixed and $\nu\to\infty$) that the heuristic bounds are true. We use the following asymptotic expansions~\cite{Abramowitz} (Formulas (9.3.31)--(9.3.32)) 
$$J_\nu(\nu)=a\nu^{-1/3}+\mathcal{O}(\nu^{-5/3})\qquad Y_\nu(\nu)=-\sqrt{3}a\nu^{-1/3}+\mathcal{O}(\nu^{-5/3})$$
which are valid as $\nu\to\infty$, and where $a=\frac{2^{1/3}}{3^{2/3}\Gamma(2/3)}$, we obtain that for large enough values of $k$ the following estimate is valid
$$b_\nu(k)\sim 2\sqrt{2}a^2k^{1/3},\ \nu\sim k, k\to\infty.$$
The plot of $k^{-1/3}\max_{k\leq\nu}b_\nu(k)$ in Figure~\ref{fig:max_b_nu_k} seems to be consistent with the estimate just derived above, given that $2\sqrt{2}a^2\approx 0.56592$. Furthermore,  we use the result in Formula 9.3 in~\cite{Abramowitz} which is valid for fixed $k$ and $\nu\to\infty$
$$J_\nu(k)H_\nu^{(1)}(k)\sim \frac{1}{2\pi\nu}\left(\frac{ek}{2\nu}\right)^{2\nu}-i\frac{1}{\pi\nu}=\mathcal{O}(\nu^{-1})$$
and we get that for a fixed and large enough $k$ there exists a constant $C(k)$ such that for all $\nu$ large enough we have
$$b_\nu(k)\leq C(k).$$

The next result is a coercivity result that establishes that given the choice $\xi=k^{1/3}$ in the definition of the regularizing operators $\mathcal{R}=-\xi\mathcal{T}_{ik}$, the real parts of the eigenvalues $B_n^{(1)}(k,k,k^{1/3})$ and $B_n^{(2)}(k,k,k^{1/3})$ are bounded from below by strictly positive constants for all values of $n$ for large enough values of the wavenumber $k$. More precisely, we have that
\begin{theorem}\label{thm_coercivity_1}
There exists a constant $C_9$ and a wavenumber $k'$ such that for all $k\geq k'$ we have that
$$\min\{\Re(B_n^{(1)}(k,k,k^{1/3})),\Re(B_n^{(2)}(k,k,k^{1/3}))\}\geq C_9\ {\rm for\ all}\ n\geq 1.$$
\end{theorem}
\begin{proof}  We take into account the estimate (i) established in Lemma~\ref{ImKm} to obtain
$$\Re(B_n^{(1)}(k,k,k^{1/3}))\geq -\frac{\pi k}{2}J'_{n+1/2}(k)Y_{n+1/2}(k)+\frac{C_1\pi^2k^{1/3}}{4}J_{n+1/2}^2(k)\sqrt{n^2+k^2}-|b_n^{(1)}(k,k^{1/3})|.$$
On the other hand, the estimate~\eqref{eq:coercDir} can be written explicitly in the form
$$\Re(d_{n+1/2}(k,k))=-\frac{\pi k}{2}J'_{n+1/2}(k)Y_{n+1/2}(k)+\frac{k\pi}{2}J_{n+1/2}^2(k)\geq\frac{1}{2}\ {\rm for\ all}\ 1\leq n$$
and sufficiently large $k$. Since $\sqrt{n^2+k^2}\geq k$ and $C_1\pi k^{1/3}\geq 2$ for sufficiently large $k$, it follows from the previous two estimates that
$$\Re(B_n^{(1)}(k,k,k^{1/3}))\geq \Re(d_{n+1/2}(k,k))-|b_n^{(1)}(k,k^{1/3})|\geq \frac{1}{2}-|b_n^{(1)}(k,k^{1/3})|\ {\rm for\ all}\ n\geq 1$$
and $k$ sufficiently large. The result of the theorem now follows if we take into account estimates~\eqref{eq:est_ans} and the result established in Theorem~\ref{thm_coercivity_0}.
\end{proof}

Given the fact that $\left(\frac{\nabla_{\mathbb{S}^2} Y_n^{m}}{\sqrt{n(n+1)}}, \frac{\overrightarrow{\rm{curl}}_{\mathbb{S}^2} Y_n^{m}}{\sqrt{n(n+1)}}\right)_{1\leq n,-n\leq m \leq n}$ constitute an orthonormal basis for the space $L^2(TM(\mathbb{S}^2))$ (and hence $H^{m}_{\rm div}(\mathbb{S}^2)$ for all $2\leq m$), we derive from the results in Theorem~\ref{upper_bd_1}, equation~\eqref{eq:est_uBn1}, and Theorem~\ref{thm_coercivity_1} the following
\begin{corollary} Let $\mathcal{B}=\frac{I}{2}-\mathcal{K}_k-k^{1/3} 
\mathcal{T}_k\  \mathcal{T}_{ik}$. Assuming that estimate~\eqref{eq:est_uBn1} holds, then there exists a constant $C_B$ independent of $k$ and $m$ such that for all sufficiently large $k$ and all $2\leq m$
$$\|\mathcal{B}\|_{H^{m}_{\rm div}(\mathbb{S}^2)\to H^{m}_{\rm div}(\mathbb{S}^2)}\|\mathcal{B}^{-1}\|_{H^{m}_{\rm div}(\mathbb{S}^2)\to H^{m}_{\rm div}(\mathbb{S}^2)}\leq C_B k^{2/3}.$$
\end{corollary}

If we use $PS(\mathcal{T}_{ik})$ instead of $\mathcal{T}_{ik}$ as regularizing operators in the CFIER formulations~\eqref{eq:CFIE_R} we obtain boundary integral operators $PS\mathcal{B}=\frac{I}{2}-\mathcal{K}_k-k^{1/3} 
\mathcal{T}_k\ (PS(\mathcal{T}_{ik}))$ whose spectral properties are qualitatively similar to those of the operators $\mathcal{B}$ defined above. Given that for a complex wavenumber $K=\kappa+i\epsilon,\ \kappa\geq 0,\ \epsilon>0$ we have that 
\begin{eqnarray}\label{eq:eigPST}
PS(\mathcal{T}_K)\nabla_{\mathbb{S}^2} Y_n^{m}&=&\frac{i}{2K}(n(n+1)-K^2)^{1/2}\ \overrightarrow{\rm{curl}}_{\mathbb{S}^2}Y_n^m\nonumber\\
PS(\mathcal{T}_K)\overrightarrow{\rm{curl}}_{\mathbb{S}^2}Y_n^m&=&\frac{iK}{2}(n(n+1)-K^2)^{-1/2}\ \nabla_{\mathbb{S}^2} Y_n^{m}
\end{eqnarray}
we obtain 
\begin{eqnarray}
\left(\frac{I}{2}-\mathcal{K}_k-k^{1/3} \mathcal{T}_k\ (PS(\mathcal{T}_{ik}))\right)\nabla_{\mathbb{S}^2} Y_n^{m}&=&\left(\frac{1}{2}+\lambda_n(k)-\frac{\Lambda_n^{(2)}(k)(n(n+1)+k^2)^{1/2}}{2k^{2/3}}\right)\nabla_{\mathbb{S}^2} Y_n^{m}\nonumber\\
&=&PSB_n^{(1)}(k,k,k^{1/3})\nabla_{\mathbb{S}^2} Y_n^{m}\nonumber\\
\left(\frac{I}{2}-\mathcal{K}_k-k^{1/3}\mathcal{T}_k\ (PS(\mathcal{T}_{ik}))\right)\overrightarrow{\rm{curl}}_{\mathbb{S}^2}Y_n^m &=&PSB_n^{(2)}(k,k,k^{1/3})\overrightarrow{\rm{curl}}_{\mathbb{S}^2} Y_n^{m}\nonumber\\
&=&PSA_n^{(2)}(k,k,k^{1/3})\overrightarrow{\rm{curl}}_{\mathbb{S}^2} Y_n^{m}.
\end{eqnarray} 
Taking into account the results in Lemma~\ref{ImKm}, it follows immediately that qualitatively similar results on upper bounds and coercivity properties hold for the eigenvalues\\ $(PSB_n^{(1)}(k,k,k^{1/3}),PSB_n^{(2)}(k,k,k^{1/3}))_{1\leq n}$ and the eigenvalues $(B_n^{(1)}(k,k,k^{1/3}),B_n^{(2)}(k,k,k^{1/3}))_{1\leq n}$; in particular, the condition number of the integral equation formulations based on the operators $PS\mathcal{B}$ grow like $k^{2/3}$ in the high-frequency regime, just like those based on the operators $\mathcal{B}$.

\section{Dirichlet to Neumann maps and nearly optimal choices of regularizing operators for spherical scatterers}\label{DtN}
We present in this subsection the remarkable spectral properties that the Calder\'on-Complex CFIER integral operators $\mathcal{B}_{k,\gamma,\kappa_1,\kappa_2}$ and their principal value analogues $PS\mathcal{B}_{k,\gamma,\kappa_1,\kappa_2}$ defined in equations~\eqref{eq:opB} and~\eqref{eq:param_PS3} respectively possess in the case when $\gamma=2$, $\kappa_1=k$, and $\kappa_2=0.4k^{1/3}$. This choice of parameters $\gamma$, $\kappa_1$, and $\kappa_2$ is motivated by considerations on the Dirichlet-to-Neumann map of the electromagnetic scattering problem~\cite{Darbas}. Specifically, the Dirichlet-to-Neumann (DtN) map is defined as $Y(\mathbf{n}\times\mathbf{E}^s)=\mathbf{n}\times\mathbf{H}^s$. Using the Stratton-Chu representation formula~\cite{co-kr}
\begin{eqnarray}
\mathbf{E}^s(\mathbf{z})&=&{\rm curl}\ \int_\Gamma G_k(\mathbf{z}-\mathbf{y})(\mathbf{n}(\mathbf{y})\times\mathbf{E}^s(\mathbf{y}))d\sigma(\mathbf{y})+\frac{i}{k}{\rm curl}\ {\rm curl}\ \int_\Gamma G_k(\mathbf{z}-\mathbf{y})(\mathbf{n}(\mathbf{y})\times\mathbf{H}^s(\mathbf{y}))d\sigma(\mathbf{y})\nonumber\\
&=&{\rm curl}\ \int_\Gamma G_k(\mathbf{z}-\mathbf{y})(\mathbf{n}\times\mathbf{E}^s)(\mathbf{y})d\sigma(\mathbf{y})+\frac{i}{k}{\rm curl}\ {\rm curl}\ \int_\Gamma G_k(\mathbf{z}-\mathbf{y})Y(\mathbf{n}\times\mathbf{E}^s)(\mathbf{y})d\sigma(\mathbf{y})\nonumber
\end{eqnarray}
we get if we apply the $\mathbf{n}\times\cdot$ trace that $\mathcal{T}_kY=\frac{I}{2}+\mathcal{K}_k$. Clearly, if the regularizing operator $\mathcal{R}$ in equations~\eqref{eq:Representation} were chosen such that $\mathcal{R}=Y$, then the ensuing CFIER boundary integral operators would be equal to the identity operator. From this point of view, the better the regularizing operator $\mathcal{R}$ approximates the DtN operator, the closer the regularized combined field integral operators would be to the identity operator. The DtN operators $Y$ cannot be computed analytically but for simple scatterers (e.g. spherical, planes); for general surfaces $\Gamma$, the computation of the DtN operators, if at all possible, is typically as expensive as the solution of the scattering problem since $Y=\mathcal{T}_k^{-1}\left(\frac{I}{2}+\mathcal{K}_k\right)$. Using Calder\'on's identities and the previous formula we get that $Y\sim-2\mathcal{T}_k$. However, as previously argued, the choice $\mathcal{R}=-2\mathcal{T}_k$ does not lead to uniquely solvable regularized formulations for all wavenumbers $k$. Nevertheless, we can seek an approximation to the DtN operator $Y$ in the form $\mathcal{R}=-2\mathcal{T}_{k+i\kappa_2}$, which does lead to invertible operators $\mathcal{B}_{k,2,k,\kappa_2}=\frac{I}{2}-\mathcal{K}_k-2\mathcal{T}_k\  \mathcal{T}_{k+i\kappa_2}$---see Theorem~\ref{tm11}. On the other hand, using Fourier transforms, the operators DtN can be computed as Fourier multipliers in the case when $\Gamma$ is a plane in $\mathbb{R}^3$. Standard techniques of tangent plane approximations lead to approximations of the DtN operators $Y$ for general smooth surfaces $\Gamma$ in the form $\mathcal{R}=-2PS(\mathcal{T}_{k+i\kappa_2})$~\cite{Darbas}; the complexification in the definition of the latter operator is needed in order to ensure the injectivity of the operators $PS\mathcal{B}_{k,2,k,\kappa_2}=\frac{I}{2}-\mathcal{K}_k-2\mathcal{T}_k\  PS(\mathcal{T}_{k+i\kappa_2})$---see Theorem~\ref{thm2}. The selection of the parameters $\kappa_2$ is guided by considerations on the spectral properties of the DtN operators for spherical scatterers. The spectral properties of the DtN operators in the case $\Gamma=\mathbb{S}^2$~\cite{Nedelec} are given below
\begin{eqnarray}
\label{eq:eigY}
Y(\nabla_{\mathbb{S}^2} Y_n^{m})&=&\frac{i(z_n(k)+1)}{k}\overrightarrow{\rm{curl}}_{\mathbb{S}^2} Y_n^{m}=Z_n^{(1)}(k)\overrightarrow{\rm{curl}}_{\mathbb{S}^2} Y_n^{m},\\
\label{eq:eigY2}Y(\overrightarrow{\rm{curl}}_{\mathbb{S}^2} Y_n^{m})&=&
\frac{ik}{z_n(k)+1} \nabla_{\mathbb{S}^2} Y_n^{m}=Z_n^{(2)}(k) \nabla_{\mathbb{S}^2} Y_n^{m},
\end{eqnarray} 
where $z_n(k)=\frac{k(h_n^{(1)}(k))'}{h_n^{(1)}(k)}$. The value of the parameter $\kappa_2$ can be selected by minimizing the expression $\max_{1\leq n}\{|Z_n^{(1)}(k)+2\Lambda_n^{(1)}(k+i\kappa_2)|,|Z_n^{(2)}(k)+2\Lambda_n^{(2)}(k+i\kappa_2)|\}$, where the values $\{Z_n^{(1)}(k),Z_n^{(2)}(k)\}_{1\leq n}$ were defined in equations~\eqref{eq:eigY}-\eqref{eq:eigY2} and the values $\{\Lambda_n^{(1)}(k+i\kappa_2),\Lambda_n^{(2)}(k+i\kappa_2)\}_{1\leq n}$ defined in equations~\eqref{eq:Lambda_n_1}-\eqref{eq:Lambda_n_2} respectively. For large values of the wavenumber $k$, the maximum sought after occurs for values of the index $n$ such that $n\approx k$. Just as in~\cite{Darbas}, using standard asymptotic formulas of Hankel functions~\cite{Abramowitz} (Formulas (9.3.31)--(9.3.34)), the optimal value $\kappa_2=0.4k^{1/3}$ can be derived in the case when $\Gamma=\mathbb{S}^2$. We present in Figure~\ref{fig:diffBmarion} the quotients $\frac{|Z_n^{(1)}(k)+2\Lambda_n^{(1)}(k+0.4\ ik^{1/3})|}{|Z_n^{(1)}(k)|}$ (black), $\frac{|Z_n^{(2)}(k)+2\Lambda_n^{(2)}(k+0.4\ ik^{1/3})|}{|Z_n^{(2)}(k)|}$ (red) for $k=32$ (top) and $k=160$ (bottom) and and $n=1,\ldots,320$ (bottom). Similar results are obtained if we replace in the formulas above the values $\{\Lambda_n^{(1)}(k+i\kappa_2),\Lambda_n^{(2)}(k+i\kappa_2)\}_{1\leq n}$ by $\{PS\Lambda_n^{(1)}(k+i\kappa_2),PS\Lambda_n^{(2)}(k+i\kappa_2)\}_{1\leq n}$ defined in equation~\eqref{eq:eigPST}. Furthermore, in the case when $\Gamma=R\ \mathbb{S}^2$ (that is spheres of radius $R$) similar calculations lead to the almost optimal choices of the regularizing operators in the form $\mathcal{R}=-2\mathcal{T}_{k+0.4 i R^{-2/3}k^{1/3}}$ and $\mathcal{R}=-2PS(\mathcal{T}_{k+0.4 i R^{-2/3}k^{1/3}})$~\cite{Darbas}. 
\begin{figure}
\centering
\includegraphics[height=65mm]{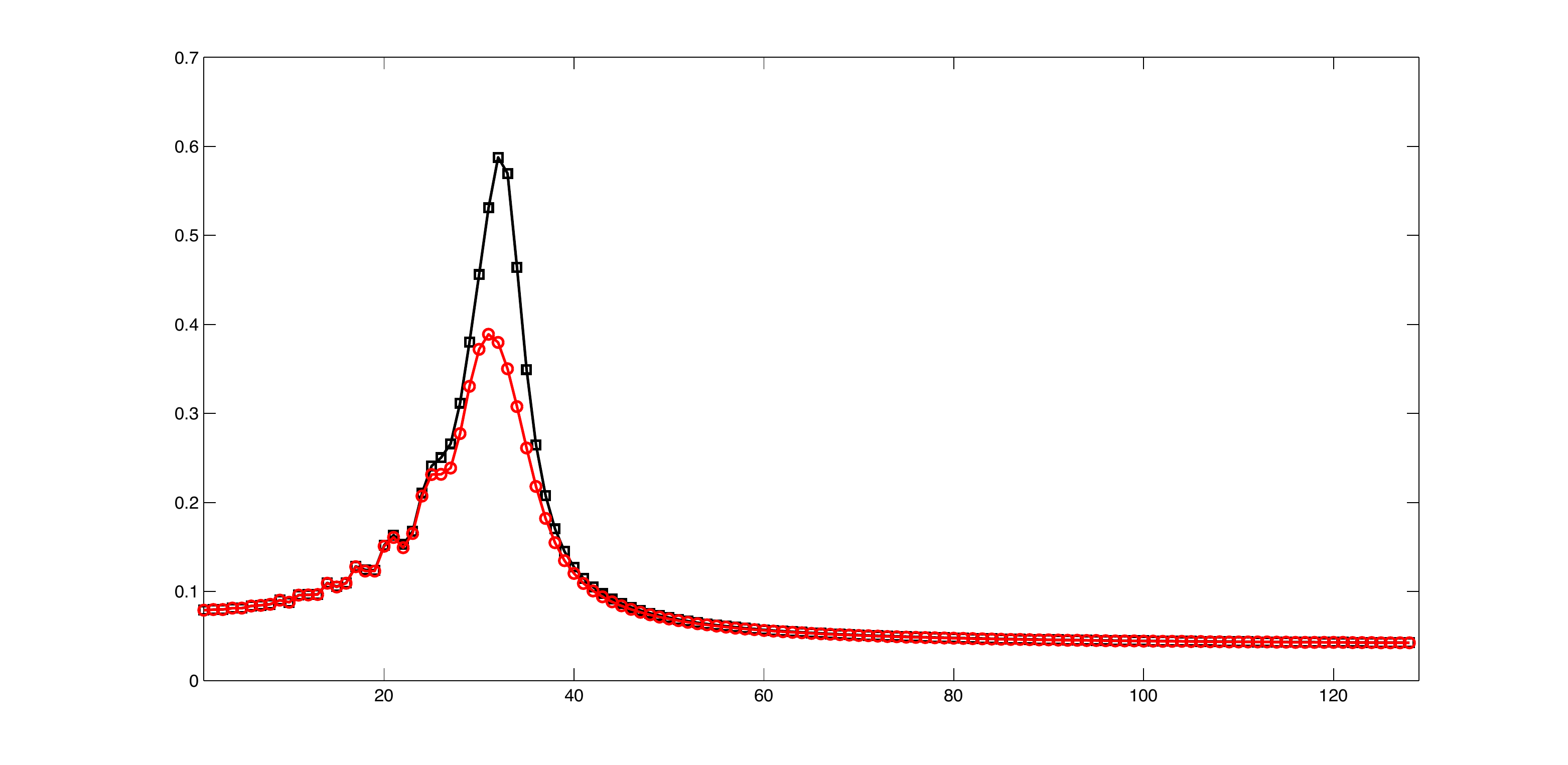}\\
\includegraphics[height=65mm]{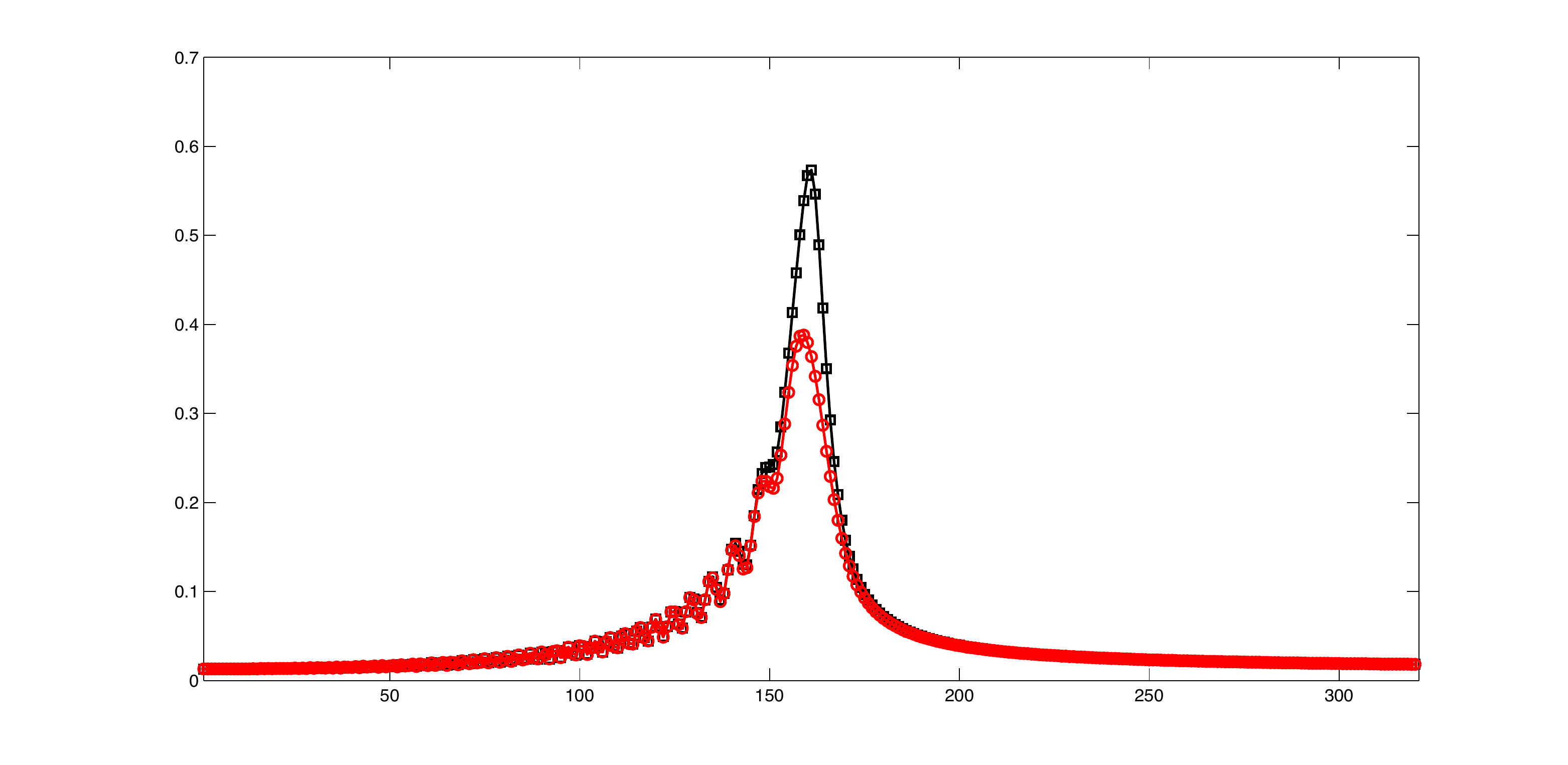}
\caption{Plots of $\frac{|Z_n^{(1)}(k)+2\Lambda_n^{(1)}(k+0.4\ ik^{1/3})|}{|Z_n^{(1)}(k)|}$ (black), $\frac{|Z_n^{(2)}(k)+2\Lambda_n^{(2)}(k+0.4\ ik^{1/3})|}{|Z_n^{(2)}(k)|}$ (red) for $k=32$ (top) and $k=160$ (bottom) for $n=1,\ldots,128$ (top) and $n=1,\ldots,320$ (bottom).}
\label{fig:diffBmarion}
\end{figure}

In addition, the eigenvalues of the operators $\mathcal{B}_{k,2,k,0.4k^{1/3}}$ and $PS\mathcal{B}_{k,2,k,0.4k^{1/3}}$ can be computed easily from equations~\eqref{eq:lambda_n},~\eqref{eq:Lambda_n_1},~\eqref{eq:Lambda_n_2}, and~\eqref{eq:eigPST}. We display in Figure~\ref{fig:Bmarion} and Figure~\ref{fig:Bmarionps} the coercivity constants (that is the minimum value of the real parts of the eigenvalues of these operators) and the condition numbers of the operators $\mathcal{B}_{k,2,k,0.4k^{1/3}}$ and $PS\mathcal{B}_{k,2,k,0.4k^{1/3}}$ respectively for 5041 values of $k$ from $k=8$ to $k=512$. The numerical results depicted in Figure~\ref{fig:Bmarion} and Figure~\ref{fig:Bmarionps} suggest that both operators $\mathcal{B}_{k,2,k,0.4k^{1/3}}$ and $PS\mathcal{B}_{k,2,k,0.4k^{1/3}}$ are coercive for large enough values of $k$ in the case of spherical scatterers, and their condition numbers appear to be bounded independently of the wavenumber $k$ for large enough values of $k$. A rigorous proof of these results is outside the scope of the present effort.
\begin{figure}
\centering
\includegraphics[height=65mm]{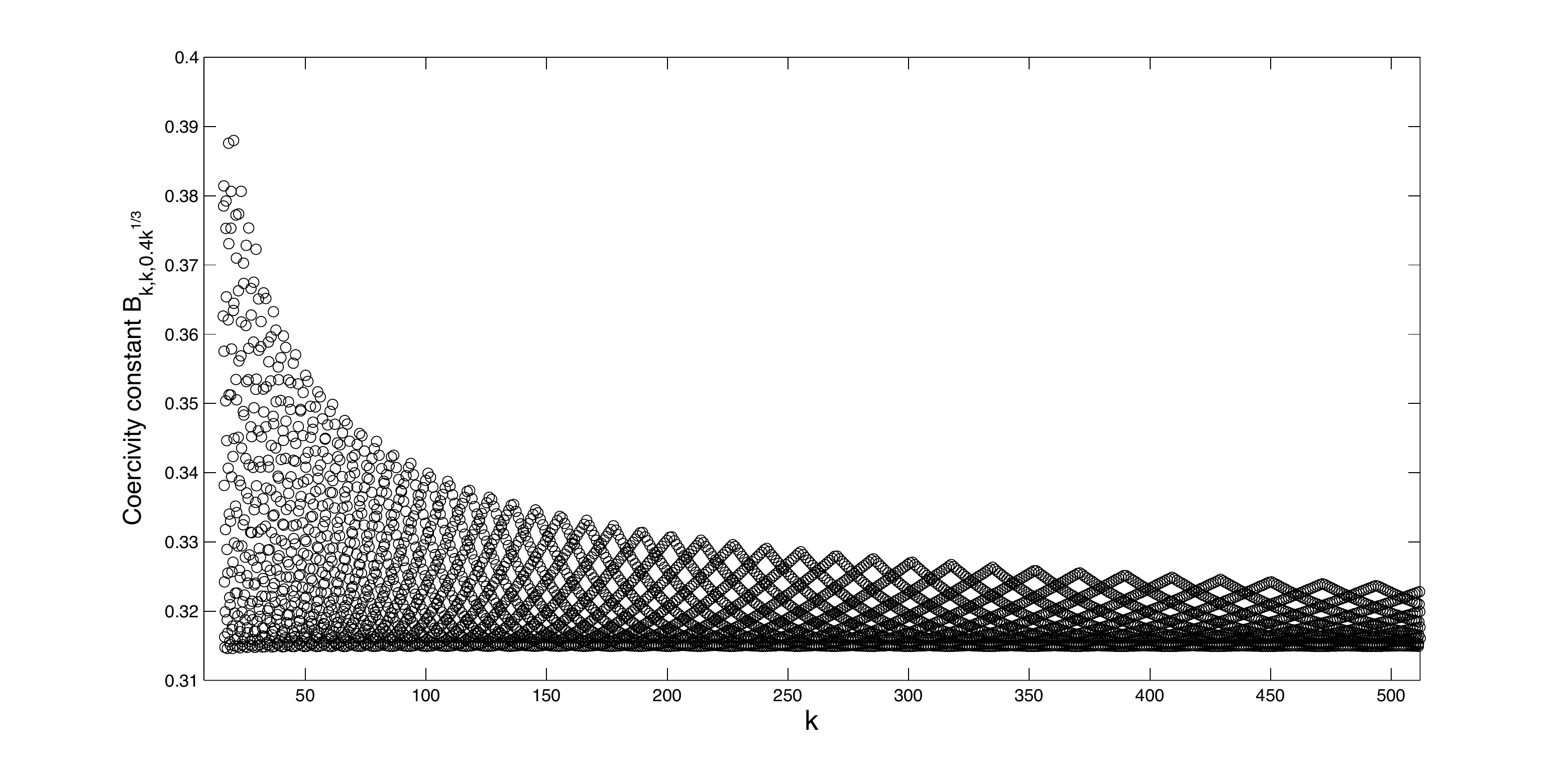}\\
\includegraphics[height=65mm]{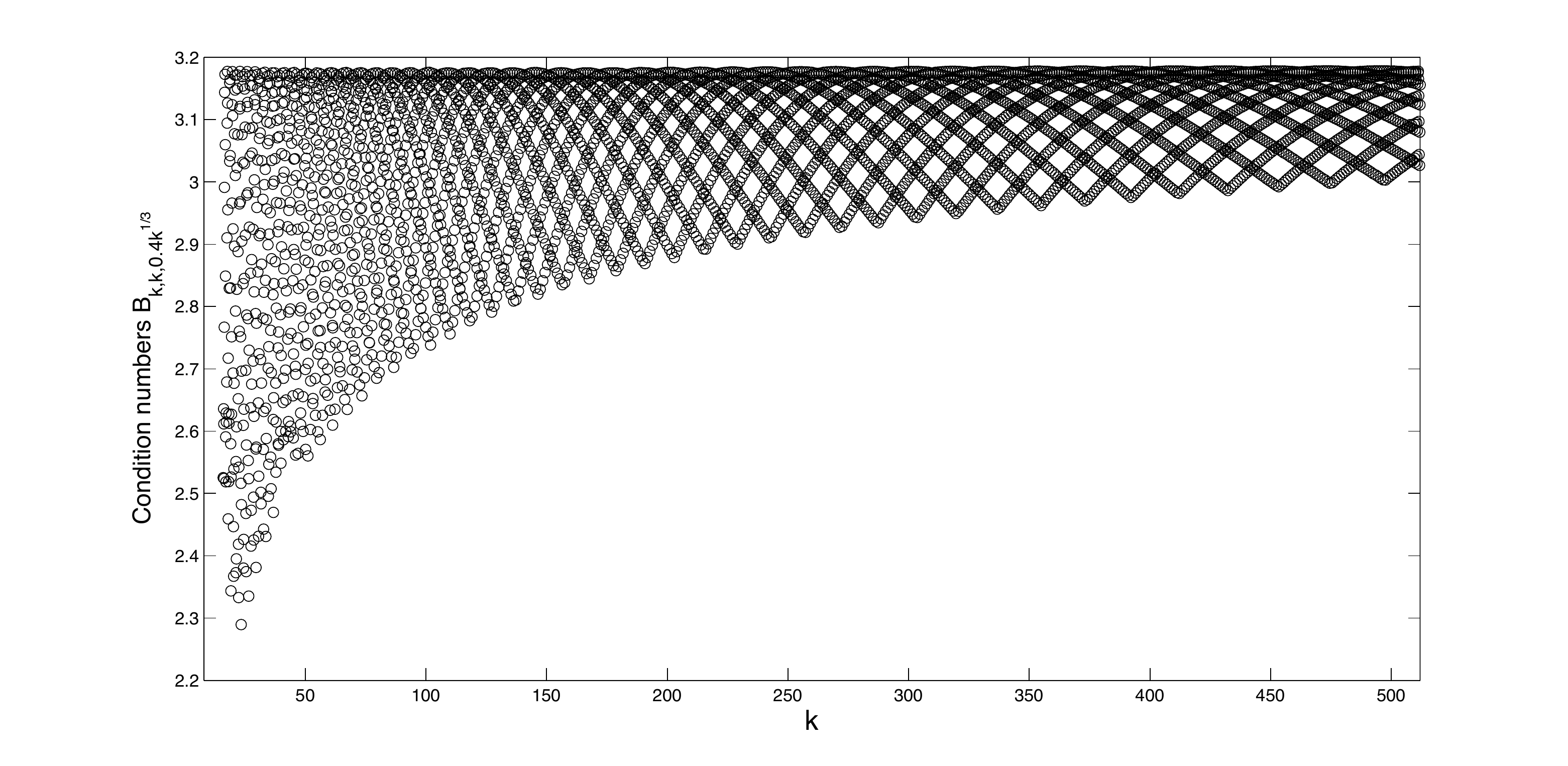}
\caption{Plots of the coercivity constants (top) and condition numbers (bottom) of the operators $\mathcal{B}_{k,2,k,0.4k^{1/3}}$ for 5041 values of the wavenumber $k$ from $k=8$ to $k=512$.}
\label{fig:Bmarion}
\end{figure}

\begin{figure}
\centering
\includegraphics[height=65mm]{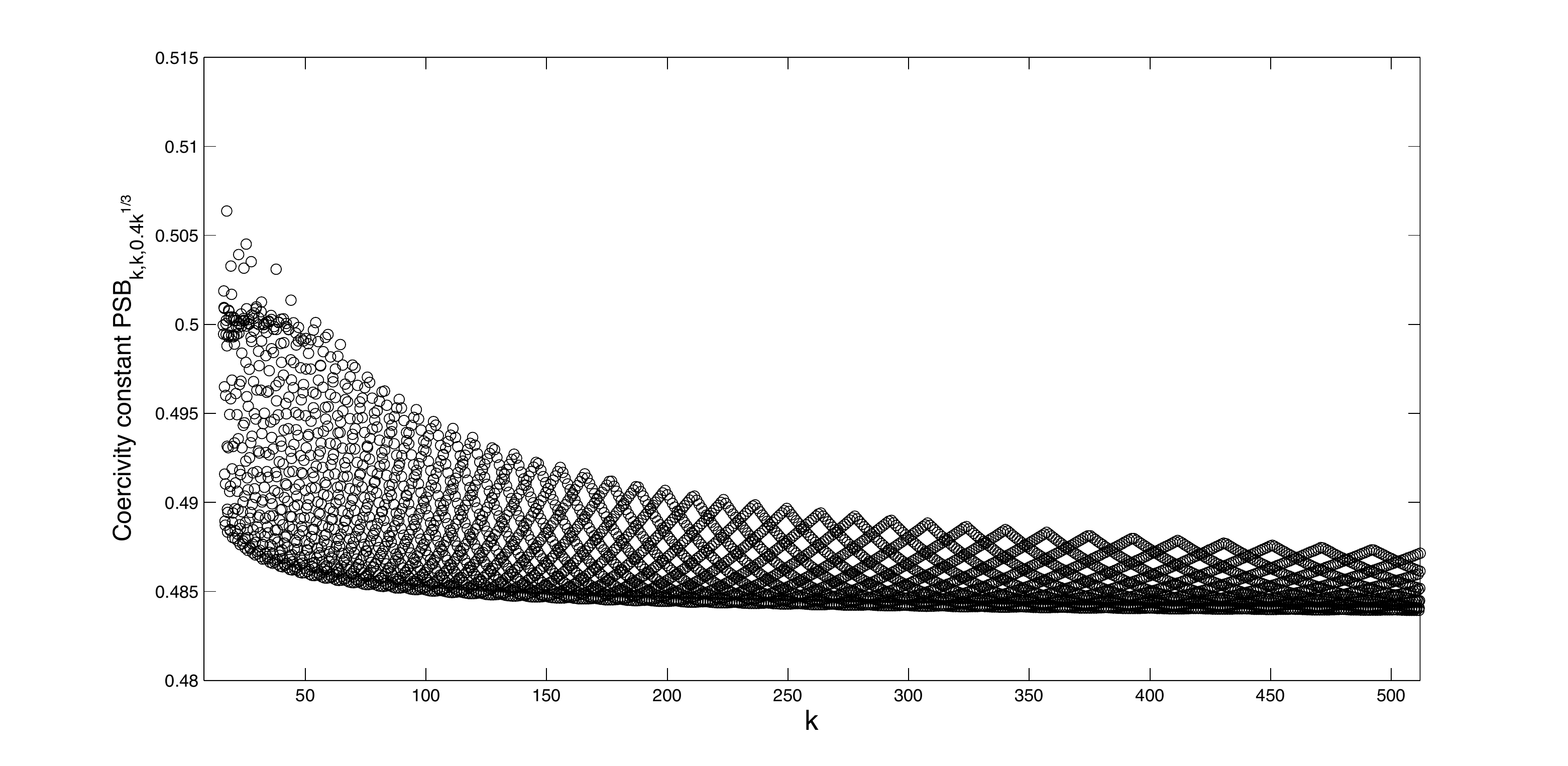}\\
\includegraphics[height=65mm]{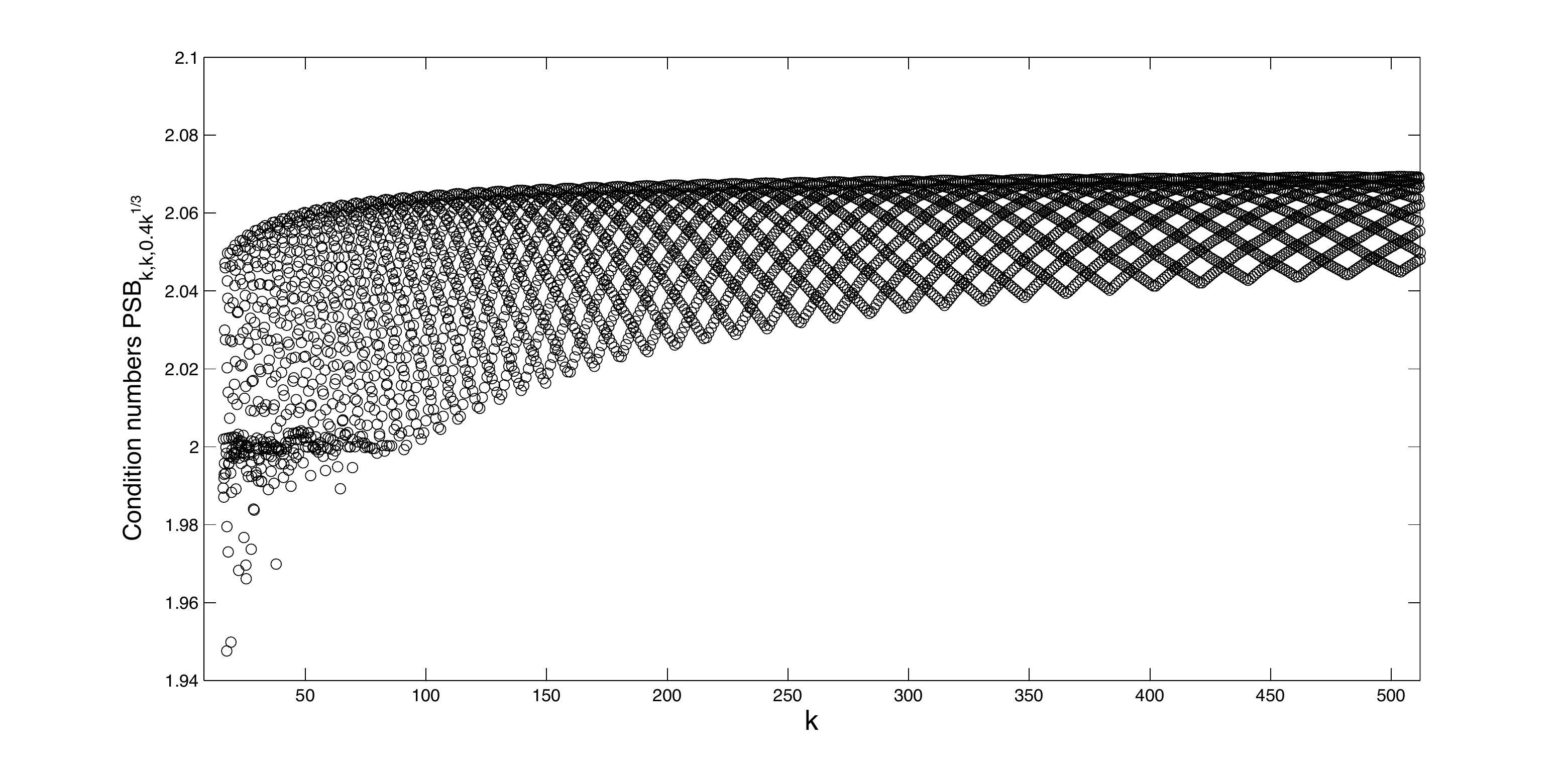}
\caption{Plots of the coercivity constants (top) and condition numbers (bottom) of the operators $PS\mathcal{B}_{k,2,k,0.4k^{1/3}}$ for 5041 values of the wavenumber $k$ from $k=8$ to $k=512$.}
\label{fig:Bmarionps}
\end{figure}

\section{High-frequency numerical experiments}\label{num_exp}

We present in this section a variety of numerical results that
demonstrate the important computational savings that can be garnered from use of regularized combined field integral equations in the high-frequency regime. Most importantly, we showcase the superior performance of solvers based on the novel Calder\'on-Complex CFIER formulations that involve the boundary integral operators
\begin{equation}\label{eq:CFIERc}
\mathcal{B}_{k,2,k,0.4\mathcal{H}^{2/3}k^{1/3}}=\frac{I}{2}-\mathcal{K}_k-2\mathcal{T}_k\  \mathcal{T}_{k+i\ 0.4\mathcal{H}^{2/3}k^{1/3}},
\end{equation}
where $\mathcal{H}$ is the maximum of the absolute values of mean curvatures on surface $\Gamma$. The latter choice is motivated by the fact that $c=R^{-2/3}$ leads to operators $\mathcal{B}_{k,2,k,0.4ck^{1/3}}$ with nearly optimal spectral properties in the high-frequency regime in the case when $\Gamma$ is a sphere of radius $R$ (see Section~\ref{DtN}). For general smooth surfaces $\Gamma$, we followed the heuristical practice of replacing $R^{-1}$ by $\mathcal{H}$. 

In addition, we illustrate the performance of solvers based on Calder\'on-Ikawa CFIER formulations at high frequencies. This type of formulations has been shown to work effectively in the low and medium frequency regime~\cite{Contopa_et_al,turc1,Andriulli1,Andriulli}. We focus on the arguably simplest version of Calder\'on CFIER formulations that involve the following boundary integral operators~\cite{turc1}:
\begin{equation}\label{eq:fIE}
\mathcal{A}_{k,S_{ik/2}}=I/2-\mathcal{K}_k+k \mathcal{T}_k\ (\mathbf{n}\times \mathbf{S}_{ik/2}).
\end{equation} 
The motivation for the consideration of formulations based on the operators $\mathcal{A}_{k,S_{ik/2}}$ is the fact that amongst Calder\'on-Ikawa type CFIER formulations, those based on the operators $\mathcal{A}_{k,S_{ik/2}}$ are almost optimal. Indeed, the properties of the solvers based on the integral operators $\mathcal{A}_{k,S_{ik/2}}$ defined in equations~\eqref{eq:fIE} were investigated in~\cite{turc1} in the low and medium-frequency range. The same reference~\cite{turc1} provides ample numerical comparisons with solvers based on the Calder\'on-Ikawa CFIER operators $\mathcal{B}_{k,1,0,k}$ that were proposed in~\cite{Contopa_et_al}: solvers based on the operators $\mathcal{B}_{k,1,0,k}$ lead to somewhat larger iteration counts than those based on the operators $\mathcal{A}_{k,S_{ik/2}}$, and the cost of evaluating a matrix-vector product associated with the former operators is on average about $1.6$ times more expensive than the that associated with the latter operators. In order to strengthen our claim, we present further numerical evidence in Figure~\ref{fig:Bmarioncomp} on the wavenumber dependence of the condition numbers of the operators $\mathcal{A}_{k,S_{ik/2}}$, $\mathcal{B}_{k,1,0,k}$, and $\mathcal{B}_{k,2,k,0.4k^{1/3}}$ for spherical geometries of radius one. As it can be seen from the results in Figure~\ref{fig:Bmarioncomp}, the condition numbers of the operators $\mathcal{A}_{k,S_{ik/2}}$ and $\mathcal{B}_{k,1,0,k}$ behave asymptotically as $\mathcal{O}(k^{2/3})$ for spherical scatterers, with smaller proportionality constants for the former operators. In contrast, the condition numbers of the operators $\mathcal{B}_{k,2,k,0.4k^{1/3}}$ are bounded independently of frequency for spherical scatterers. 
\begin{figure}
\centering
\includegraphics[height=65mm]{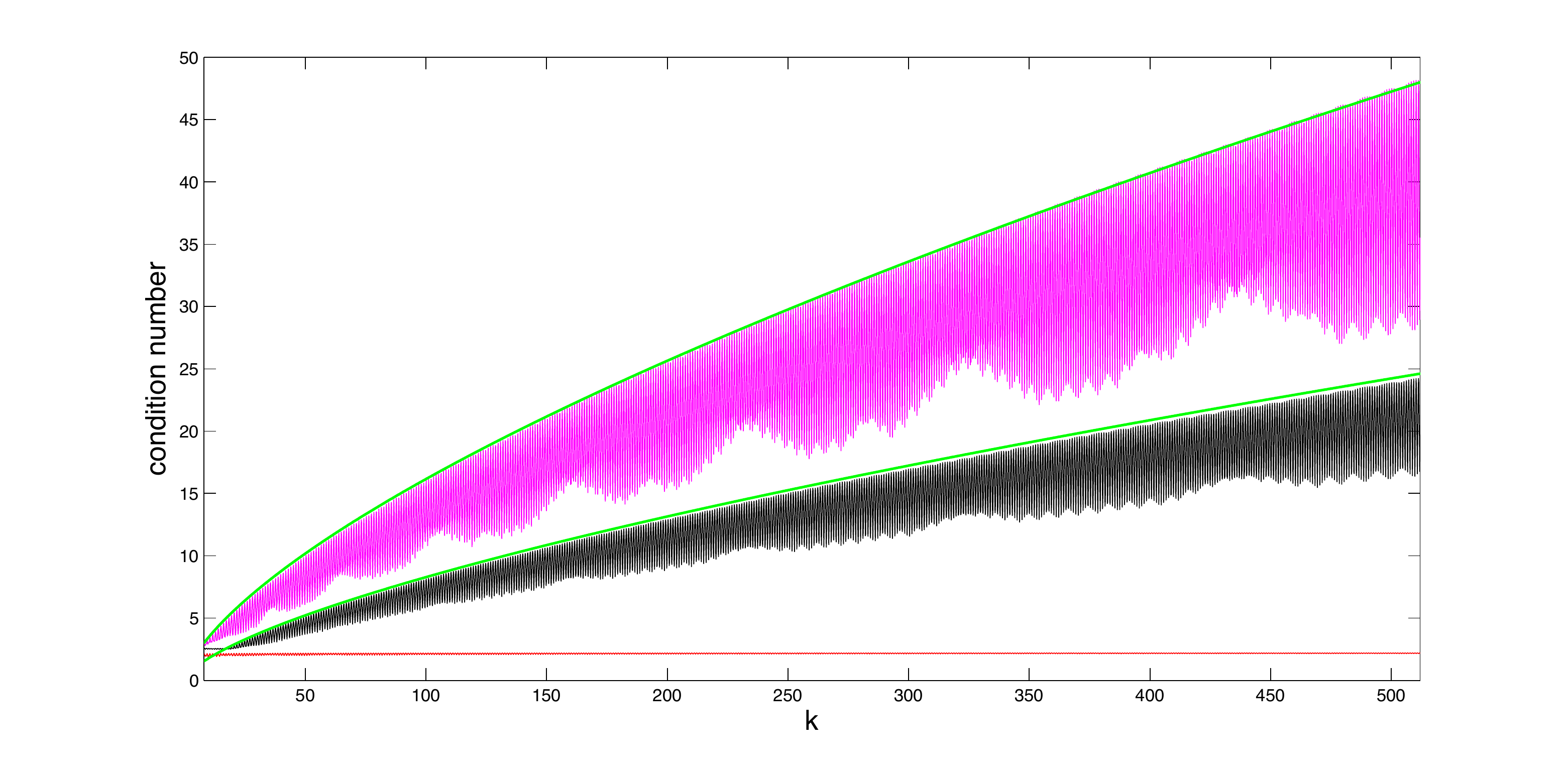}
\caption{Plot of the condition numbers of the three regularized combined field integral operators $\mathcal{A}_{k,S_{ik/2}}$ (black), $\mathcal{B}_{k,1,0,k}$ (magenta), and $\mathcal{B}_{k,2,k,0.4k^{1/3}}$ (red)  as a function of the wavenumber $k$ for 5041 values of $k$ ranging from $k=8$ to $k=512$. A fairly sharp upper bound on the condition number of the operators $\mathcal{A}_{k,S_{ik/2}}$ is given by $0.38\ k^{2/3}$ (plotted in green) and a fairly sharp upper bound on the condition number of the operators $\mathcal{B}_{k,1,0,k}$ is given by $0.75\ k^{2/3}$ (plotted in green). Both upper bounds are in agreement with the theoretical predictions presented in Section~\ref{spec_prop_cfie}.}
\label{fig:Bmarioncomp}
\end{figure}
We compare the performance of solvers based on regularized combined field integral equations involving the integral operators defined in equations~\eqref{eq:CFIERc} and~\eqref{eq:fIE} to those based on the classical combined field integral equation that features the boundary integral operator~\cite{Kress1985}
\begin{equation}\label{eq:cIE}
\mathcal{C}_k=I/2-\mathcal{K}_k+ \mathcal{T}_k\ (\mathbf{n}\times I).
\end{equation}
We re-emphasize that unlike the CFIER formulations considered in this text, the combined field integral equations based on the operator $\mathcal{C}_k$ are not integral equations of the second kind. Nevertheless, formulations based on $\mathcal{C}_k$ are robust for all wavenumbers $k$, and are by far the most widely used formulations in the computational electromagnetic community. We illustrate in this section that solvers based on CFIER formulations give rise to important computational savings over solvers based on CFIE formulations in the high-frequency regime.

Our strategy for evaluation of the relevant discrete
integro-differential operators that enter the integral equation formulations based on the operators $\mathcal{B}_{k,2,k,0.4\mathcal{H}^{2/3}k^{1/3}}$, $\mathcal{A}_{k,S_{ik/2}}$, and $\mathcal{C}_k$ described above relies on a Nystr\"om (collocation) method that uses {\em local coordinate
charts} together with {\em fixed} and {\em floating} partitions of
unity (POU), as proposed in~\cite{BrunoKun,turc1,turc3}. First, following the prescriptions in~\cite{turc1}, we express all the electromagentic boundary integral operators as integral operators with weakly singular kernels, that is
\begin{equation}
\label{eq:MFIEweakly}
(\mathcal K_K \mathbf {a})(\mathbf{x}) = \int_{\Gamma}\left((\mathbf{n}(\mathbf{x})-\mathbf{n}(\mathbf{y}))\cdot \mathbf{a}(\mathbf{y})\nabla_{\mathbf {y}}G_K(\mathbf x-\mathbf {y})+\frac{\partial G_K(\mathbf{x}-\mathbf{y})}{\partial \mathbf{n}(\mathbf{x)}}\mathbf{a}(\mathbf{y})\right)d\sigma(\mathbf{y}),
\end{equation}
and
\begin{eqnarray}
\label{eq:EFIEweakly}
(\mathcal T_K \mathbf{a})(\mathbf{x})&=& iK\ \mathbf{n}(\mathbf{x})\times \int_{\Gamma}G_K(\mathbf{x}-\mathbf{y})\mathbf{a}(\mathbf{y})d\sigma(\mathbf {y})\nonumber\\
&-&\frac{i}{K} \int_{\Gamma}(\mathbf n(\mathbf{y})-\mathbf{n}(\mathbf{x}))\times\nabla_{\mathbf{x}}G_K(\mathbf{x}-\mathbf{y})\rm{div}_\Gamma \mathbf a(\mathbf{y})d\sigma(\mathbf{y})\nonumber\\
&-&\frac{i}{K}\int_{\Gamma}G_K(\mathbf x-\mathbf{y})\ \overrightarrow{\rm{curl}}_\Gamma\ \rm{div}_\Gamma \mathbf{a}(\mathbf{y})d\sigma(\mathbf{y}),
\end{eqnarray}
where $\mathbf{a}$ is a tangential vector field, and $K$ is a wavenumber, possibly complex. We note that in order to perform accurate evaluations of the operators $\mathcal{B}_{k,2,k,0.4 \mathcal{H}^{2/3}k^{1/3}}$ defined in equation~\eqref{eq:CFIERc}, we make use of the explicit cancellation of the composition of the hypersingular terms that enter the definition of the operators $ \mathcal{T}_k$ and $\mathcal{T}_{k+i\ 0.4\ \mathcal{H}^{2/3}k^{1/3}}$ respectively. Accordingly, the operators $\mathcal{B}_{k,2,k,0.4 \mathcal{H}^{2/3}k^{1/3}}$ are evaluated using their equivalent definition
\begin{eqnarray}\label{eq:sIE1}
\mathcal{B}_{k,2,k,0.4\mathcal{H}^{2/3}k^{1/3}}&=&I/2-\mathcal{K}_k- 2i(k+i\ 0.4\mathcal{H}^{2/3}k^{1/3})\mathcal{T}_k\ (\mathbf{n}\times \mathbf{S}_{k+i\ 0.4\mathcal{H}^{2/3}k^{1/3}})\nonumber\\
&+&\frac{2k}{k+i\ 0.4\mathcal{H}^{2/3}k^{1/3}}(\mathbf{n}\times\mathbf{S}_k)\ \mathcal{T}^{1}_{k+i\ 0.4\mathcal{H}^{2/3}k^{1/3}}{\rm div}_\Gamma.
\end{eqnarray}

The various integral operators that
enter equations~\eqref{eq:MFIEweakly} and~\eqref{eq:EFIEweakly} are computed in two stages which consist of (a) the
evaluation of the adjacent/singular interactions of sources (i.e. the
values of the density $\mathbf{a}(\mathbf y)$ or its derivatives for $\mathbf y$ on the
surface $\Gamma$) via the Green's functions (i.e. the terms that
involve $G_K(\mathbf{x}-\mathbf{y})$) when the target points $\mathbf{x}$ are close to the integration points $\mathbf{y}$
and (b) the accelerated evaluation of far-away interactions of sources
that are well-separated. The separation of these contributions is effected by floating POUs which are pairs of functions of the form
$(\eta_\mathbf{x}(\mathbf y),1-\eta_\mathbf{x}(\mathbf y))$ (where
$\eta_\mathbf{x}$ is a function with a ``small'' support which equals
1 in a neighborhood of the point $\mathbf{x}$). The approach for the solution of the integration problems (a) and (b) recounted above relies on use of smooth surface parametrizations of the
surface $\Gamma$ via a family of overlapping two-dimensional patches
$\mathcal{P}^{\ell},\ell=1,\ldots P$ together with smooth mappings to
coordinate sets $\mathcal{H}^\ell$ in two-dimensional space (where
actual integrations are performed) and subordinated partitions of
unity i.e smooth functions $w_\ell$ supported on $\mathcal{P}^\ell$
such that $\sum_\ell w_\ell=1$ throughout $\Gamma$. This framework
allows us to (i) reduce the integration of the tangential densities
$\mathbf{a}$ over the surface $\Gamma$ to calculations of integrals of
smooth vector fields $\mathbf{a}^\ell$ compactly supported in the
planar sets $\mathcal{H}^\ell$ and (ii) compute the derivatives of the
density $\mathbf{a}$ via derivatives of smooth and periodic functions
only~\cite{turc1,turc3}. The first part (a) requires the analytic resolution
of weakly singular Green's functions (i.e. the order of the singularity is $\mathcal
O(|\mathbf x -\mathbf y|^{-1})$) which is performed via polar changes
of variables whose Jacobian cancels the singularity and interpolation
procedures that allow for evaluations of the densities at radial
integration points. On the other hand, part (b) corresponds to the
acceleration procedure performed based on {\em two-face planar arrays}
of ``equivalent sources'' that allow for 3D FFT based fast evaluations
of cartesian grid convolutions of the equivalent source intensities
with Green's functions~\cite{BrunoKun,turc3}. The various boundary integral operators that are components of the operators described in equations~\eqref{eq:MFIEweakly} and~\eqref{eq:EFIEweakly} are vector versions of the scalar operators presented in~\cite{turc3}. A lengthy discussion was given in~\cite{turc3} on the accelerated evaluation of those boundary integral operators and on the specific choices of the various parameters that are needed in the accelerators. The accelerated evaluation of electromagentic integral operators of the type~\eqref{eq:MFIEweakly} and~\eqref{eq:EFIEweakly} use the following four parameters: (1) the number $L^3$ of small boxes that make up the partition of a box enclosing the scatterer; (2) the number $M^{eq}$ of equivalent sources placed on faces of each of the small boxes so that the far-field contributions of the equivalent sources matches in the least square sense the one of the real surface sources; (3) the number $n^{coll}$ of collocation points needed in the solution of the least-square problem related to equivalent sources expansions; and (4) the number $n^w$ of plane waves needed for the evaluation of the fields on the scatterer from the FFTs based computations of the fields arising from the Cartesian grids placed equivalent sources. As a result of the acceleration procedure, and for all of the formulations considered in this text the cost of one matrix-vector product is $\mathcal{O}(N^{4/3}\log{N})$, where $N$ is the number of discretization points---see Table~\ref{results33}.
\begin{remark} We note that Galerkin discretizations of the Calder\'on-Ikawa CFIER operators $\mathcal{B}_{k,1,0,k}$ were presented in~\cite{Andriulli}; the algorithms presented therein ought to be easily extendable to the Galerkin discretization of the operators $\mathcal{B}_{k,2,k,0.4\mathcal{H}^{2/3}k^{1/3}}$. Very recently, Galerkin discretization of the operators
 \begin{equation*}
PS\mathcal{B}_{k,2,k,0.4\mathcal{H}^{2/3}k^{1/3}}=\frac{I}{2}-\mathcal{K}_k-2\mathcal{T}_k\  PS(\mathcal{T}_{k+i\ 0.4\mathcal{H}^{2/3}k^{1/3}}),
\end{equation*}
were presented in~\cite{Boujaji}, where the square root operators featured in the definition of the principal symbol operators were replaced by their Pad\'e approximations. However, the results based on this formulations appear to be less accurate than those based on the classical CFIE formulation~\cite{Boujaji}. The implementation of the same operators in the framework of our Nystr\"om discretization is currently underway. The main difficulty is handling with care the composition $\mathcal{T}_k\  PS(\mathcal{T}_{k+i\ 0.4\mathcal{H}^{2/3}k^{1/3}})$ along the lines of equation~\eqref{eq:sIE1}---note that per equation~\eqref{eq:defPSR} the operator $PS(\mathcal{T}_{k+i\ 0.4\mathcal{H}^{2/3}k^{1/3}})$ features second order surface derivatives. Thus, in order to prevent losses of accuracy, the evaluation of the operator $PS(\mathcal{T}_{k+i\ 0.4\mathcal{H}^{2/3}k^{1/3}})$, which is a Fourier multiplier composed to a second order differential operator, should be pursued using discrete Helmholtz decompositions and FFTs.
\end{remark}

In what follows we present the major steps of the accelerated algorithms for the evaluation of the matrix vector products associated with each of the three operators $\mathcal{C}_k$, $\mathcal{A}_{k,S_{ik/2}}$, and $\mathcal{B}_{k,2,k,0.4\mathcal{H}^{2/3}k^{1/3}}$.

\begin{center}
{\em Algorithm~I: Accelerated CFIE algorithm for the evaluation of the operator $\mathcal{C}_k$~\eqref{eq:cIE}}
\end{center}
\begin{enumerate}
\item For a given tangential density $\mathbf{a}$, evaluate $\mathbf{n}\times \mathbf{a}$ and compute the surface derivatives ${\rm div}_\Gamma(\mathbf{n}\times \mathbf{a})$, $\overrightarrow{\rm{curl}}_\Gamma\ \rm{div}_\Gamma (\mathbf{n}\times \mathbf{a})$ and combine the densities $\mathbf{a}$, $\mathbf{n}\times \mathbf{a}$, ${\rm div}_\Gamma(\mathbf{n}\times \mathbf{a})$, and $\overrightarrow{\rm{curl}}_\Gamma\ \rm{div}_\Gamma (\mathbf{n}\times \mathbf{a})$ into an extended vector density $\tilde{\mathbf{a}}$;
\item Apply the operators $\mathcal{K}_k$ and $\mathcal{T}_k$ to the
corresponding components of the vector density
   $\tilde{\mathbf{a}}$---using accelerated algorithms with wavenumber $k$ and lumping common kernel components that enter in the
   definitions of $\mathcal{K}_k$ and $\mathcal{T}_k$.
\end{enumerate}

\begin{center}
{\em Algorithm~II: Accelerated CFIER algorithm for the evaluation of the operator $\mathcal{A}_{k,S_{ik/2}}$~\eqref{eq:fIE}}
\end{center}
\begin{enumerate}
\item For a given tangential density $\mathbf{a}$, evaluate $\mathbf{n}\times \mathbf{S}_{ik/2}\mathbf{a}$ using accelerated algorithms with imaginary wavenumbers $ik/2$;
\item Compute the surface derivatives ${\rm div}_\Gamma(\mathbf{n}\times \mathbf{S}_{ik/2}\mathbf{a})$, $\overrightarrow{\rm{curl}}_\Gamma\ {\rm div}_\Gamma (\mathbf{n}\times \mathbf{S}_{ik/2}\mathbf{a})$ and combine the densities $\mathbf{a}$, $\mathbf{n}\times \mathbf{S}_{ik/2}\mathbf{a}$, ${\rm div}_\Gamma(\mathbf{n}\times \mathbf{S}_{ik/2}\mathbf{a})$, and $\overrightarrow{{\rm curl}}_\Gamma\ {\rm div}_\Gamma (\mathbf{n}\times \mathbf{S}_{ik/2}\mathbf{a})$ into an extended vector density $\tilde{\mathbf{a}}$;
\item Apply the operators $\mathcal{K}_k$ and $\mathcal{T}_k$ to the
corresponding components of the vector density
   $\tilde{\mathbf{a}}$---using accelerated algorithms with wavenumber $k$ and lumping common kernel components that enter in the
   definitions of $\mathcal{K}_k$ and $\mathcal{T}_k$.
\end{enumerate}

\begin{center}
{\em Algorithm~III: Accelerated CFIER algorithm for the evaluation of the operator $\mathcal{B}_{k,2,k,0.4\mathcal{H}^{2/3}k^{1/3}}$~\eqref{eq:sIE1}}
\end{center}
\begin{enumerate}
\item For a given tangential density $\mathbf{a}$, compute the surface derivatives ${\rm div}_\Gamma\mathbf{a}$, $\overrightarrow{\rm{curl}}_\Gamma\ \rm{div}_\Gamma \mathbf{a}$ and combine the densities $\mathbf{a}$, ${\rm div}_\Gamma \mathbf{a}$, and $\overrightarrow{\rm{curl}}_\Gamma\ \rm{div}_\Gamma \mathbf{a}$ into an extended vector density $\tilde{\mathbf{a}}_1$;
\item Apply the operators $\mathbf{n}\times \mathbf{S}_{k+i0.4\mathcal{H}^{2/3}k^{1/3}}$ and $\mathcal{T}_{k+i0.4\mathcal{H}^{2/3}k^{1/3}}^{1}{\rm div}_\Gamma$ to the corresponding components of the extended vector density $\tilde{\mathbf{a}}_1$ using accelerated algorithms with complex wavenumbers $k+i0.4\mathcal{H}^{2/3}k^{1/3}$ and lumping common kernel components that enter the definitions of the operators $\mathbf{n}\times \mathbf{S}_{k+i0.4\mathcal{H}^{2/3}k^{1/3}}$ and $\mathcal{T}_{k+i0.4\mathcal{H}^{2/3}k^{1/3}}^{1}{\rm div}_\Gamma$;
\item Compute the surface derivatives ${\rm div}_\Gamma(\mathbf{n}\times \mathbf{S}_{k+i0.4\mathcal{H}^{2/3}k^{1/3}}\mathbf{a})$, $\overrightarrow{{\rm curl}}_\Gamma\ {\rm div}_\Gamma (\mathbf{n}\times \mathbf{S}_{k+i0.4\mathcal{H}^{2/3}k^{1/3}}\mathbf{a})$ and combine the densities $\mathbf{a}$, $\mathbf{n}\times \mathbf{S}_{k+i0.4\mathcal{H}^{2/3}k^{1/3}}\mathbf{a}$, ${\rm div}_\Gamma(\mathbf{n}\times \mathbf{S}_{k+i0.4\mathcal{H}^{2/3}k^{1/3}}\mathbf{a})$, and $\overrightarrow{{\rm curl}}_\Gamma\ {\rm div}_\Gamma (\mathbf{n}\times \mathbf{S}_{k+i0.4\mathcal{H}^{2/3}k^{1/3}}\mathbf{a})$ into an extended vector density $\tilde{\mathbf{a}}_2$;
\item Apply the operators $\mathcal{K}_k$ and $\mathcal{T}_k$ to the
corresponding components of the vector density
   $\tilde{\mathbf{a}}_2$---using accelerated algorithms with wavenumber $k$ and lumping common kernel components that enter in the
   definitions of $\mathcal{K}_k$ and $\mathcal{T}_k$; at the same time apply the operator $\mathbf{n}\times\mathbf{S}_k$ to the density $\mathcal{T}^{1}_{k+i0.4\mathcal{H}^{2/3}k^{1/3}}{\rm div}_\Gamma\mathbf{a}$ using accelerated algorithms with wavenumber $k$;
\item Perform a linear combination of the results obtained in the previous step according to the formula~\eqref{eq:sIE1}.
\end{enumerate}

Solutions of the linear systems arising from the
discretization of the boundary integral equations
under consideration, namely those based on the boundary integral operators $\mathcal{B}_{k,2,k,0.4\mathcal{H}^{2/3}k^{1/3}}$, $\mathcal{A}_{k,S_{ik/2}}$, and $\mathcal{C}_k$ defined in equations~\eqref{eq:CFIERc},~\eqref{eq:fIE}, and~\eqref{eq:cIE} are
obtained by means of the fully complex version of the iterative solver GMRES~\cite{SaadSchultz} without restart.  All of the results contained in the tables presented in this section were
obtained by prescribing a GMRES residual tolerance equal to $10^{-4}$. We present results for four scattering surfaces: a sphere of radius
one, an elongated ellipsoid of principal axes $2$, $0.5$ and $0.5$, an ellipsoid of principal axes $2$, $0.5$, and $2$, and a non-convex bean shaped geometry given by the equation~\cite{BrunoKun}
$$\frac{x^{2}}{a^{2}(1-\alpha_3\cos{\frac{\pi z}{R}})}+\frac{(\alpha_1
  R\cos{\frac{\pi z}{R}})^{2}}{b^{2}(1-\alpha_2\cos{\frac{\pi
      z}{R}})}+\frac{z^{2}}{c^{2}}=R^{2},$$ with $a=0.8$, $b=0.8$,
$c=1$, $\alpha_1=0.3$, $\alpha_2=0.4$, $\alpha_3=0.1$, and
$R=1$. We consider scattering problems involving wavenumbers that correspond to scatterers whose diameters $D$ correspond to $10.2\lambda$, $20.4\lambda$, $30.6\lambda$, $40.8\lambda$, and $51.0\lambda$ 
respectively.  All of the results contained in the tables presented in this section were
obtained by  using discretizations corresponding to either $6$ or $9$ points/wavelength, leading to discretizations of size $N$. Given the coercivity results established in Section~\ref{spec_prop_cfie} for spherical scatterers, and which are probably true for the four scatterers considered (see reference~\cite{turc7} for numerical evidence of similar claims in the scalar case), discretizations corresponding to $6$ points/wavelength lead to equally accurate results throughout the very wide frequency range considered. For every scattering experiment we present the maximum relative error amongst
all directions $\hat{\mathbf{x}}=\frac{\mathbf{x}}{|\mathbf{x}|}$ of
the far field $\mathbf{E}_\infty(\hat{\mathbf{x}})$:
\begin{equation}
\label{eq:far_field}
\mathbf{E}(\mathbf{x})=\frac{e^{ik|\mathbf{x}|}}{|\mathbf{x}|}\left(\mathbf{E}_\infty(\hat{\mathbf{x}})+\mathcal{O}\left(\frac{1}{|\mathbf{x}|}\right)\right)\quad{\mbox as}\quad
|\mathbf{x}|\rightarrow\infty.\\
\end{equation}
The maximum relative far-field error, which we denote by $\varepsilon_\infty$,
\begin{equation}
\label{eq:farField_error}
\varepsilon_\infty=\frac{\max_{\hat{\mathbf{x}}}|\mathbf{E}_\infty^{\rm
  calc}(\hat{\mathbf{x}})-\mathbf{E}_\infty^{\rm ref}(\hat{\mathbf{x}})|}{\max_{\hat{\mathbf{x}}}|\mathbf{E}_\infty^{\rm ref}(\hat{\mathbf{x}})|},
\end{equation}
was evaluated in our numerical examples as the maximum difference evaluated at sufficiently many points between far fields
$\mathbf{E}_\infty^{\rm calc}$ obtained from our numerical solutions
and corresponding far fields $\mathbf{E}_\infty^{\rm ref}$ associated
with reference solutions. The reference solutions $\mathbf{E}_\infty^{\rm ref}$ were computed by Mie series in the case of spherical scatterers and by use of Combined Field Integral Equations based on the boundary integral operator $\mathcal{C}_k$ defined in equation~\eqref{eq:cIE} and the same levels of discretization. We used accelerator parameters that lead to small computational times and memory usage, while delivering results with about three digits of accuracy in the far-field metrics $\varepsilon_\infty$. Specifically, in all our numerical experiments we have used the following three accelerator parameter values: $M^{eq}=6$, $n^{coll}=10$, and $n^w=6$. The values of the remaining accelerator parameters were chosen as follows: $L=6$ for scattering problems with $D=10.2\lambda$; $L=12$ for scattering problems with $D=20.4\lambda$; $L=16$ for scattering problems with $D=30.6\lambda$; $L=24$ for scattering problems with $D=40.8\lambda$; and $L=30$ for scattering problems with $D=51.0\lambda$. The computational times and the memory usage reported resulted from a  C++ numerical implementation of our various accelerated algorithms on a workstation with 16 cores, 24 GB RAM and each processor is $2.27$ GHz Intel (R) Xeon. All of the times reported correspond to runs performed on a single processor running GNU/Linux, and
using the GNU/gcc compiler, the PETSC 3.0 library for the fully
complex implementation of GMRES, and the FFTW3 library for evaluation
of FFTs.  

We begin the presentation of our numerical results with an illustration in Table~\ref{results30} of the accuracy that can be achieved by our accelerated solvers for two scatterers: a sphere of diameter $20.4\lambda$ and an ellipsoid of size $20.4\lambda\times 5.1\lambda\times 5.1\lambda$. We used two discretizations corresponding to $6$ points and respectively $9$ points per wavelength. In the case of the sphere, the far-field errors were computed using as reference solutions given by Mie series, whereas in the case of the ellipsoid the reference solution was computed using the CFIE formulation based on the operator $\mathcal{C}_k$ and a finer discretization corresponding to $12$ points per wavelength. As it can be seen, the numbers of iterations required to reach the desired GMRES residual are stable in the case of solvers based on the second kind Fredholm operators $\mathcal{A}_{k,S_{ik/2}}$ and $\mathcal{B}_{k,2,k,0.4\mathcal{H}^{2/3}k^{1/3}}$, while in the case of solvers based on the CFIE operators $\mathcal{C}_k$ the numbers of iterations grow with the size of the discretization.
\begin{table}
\begin{center}
\resizebox{!}{1.6cm}{
\begin{tabular}{|c|c|c|c|c|c|c|c|}
\hline
$D$  & $N$ & \multicolumn{2}{c|}{$\mathcal{C}_k$~\eqref{eq:cIE}} & \multicolumn{2}{c|}{$\mathcal{A}_{k,S_{ik/2}}$~\eqref{eq:fIE}} & \multicolumn{2}{c|}{$\mathcal{B}_{k,2,k,0.4\mathcal{H}^{2/3}k^{1/3}}$~\eqref{eq:CFIERc}}\\
\cline{3-8}
& & It & $\epsilon_\infty$ & It & $\epsilon_\infty$ & It & $\epsilon_\infty$\\
\hline
Sphere & 193,548 & 43 & 2.5 $\times$ $10^{-3}$ & 21  & 3.0 $\times$ $10^{-3}$ & 8 & 1.1 $\times$ $10^{-3}$\\
\hline
Sphere & 437,772 & 50 & 1.5 $\times$ $10^{-4}$ & 21 & 4.8 $\times$ $10^{-4}$ & 8 & 2.3 $\times$ $10^{-4}$\\
\hline
Ellipsoid & 193,548 & 75 & 1.1 $\times$ $10^{-3}$ & 15 & 9.4 $\times$ $10^{-4}$ & 11 & 9.3 $\times$ $10^{-4}$\\
\hline
Ellipsoid & 437,772 & 93 & 9.4 $\times$ $10^{-5}$ & 15 & 1.8 $\times$ $10^{-4}$& 10  & 1.0 $\times$ $10^{-4}$\\
\hline
\end{tabular}}
\caption{\label{results30}{ Convergence of our solvers based on formulations that involve the boundary integral operators $\mathcal{C}_k$, $\mathcal{A}_{k,S_{ik/2}}$, and  $\mathcal{B}_{k,2,k,0.4\mathcal{H}^{2/3}k^{1/3}}$. Accelerated computations for a sphere of diameter $20.4\lambda$ and an ellipsoid of size $20.4\lambda\times 5.1\lambda\times 5.1\lambda$ for two discretizations corresponding to $6$ points and respectively $9$ points per wavelength, under $x$-polarized plane wave normal incidence.}}
\end{center}
\end{table}

We illustrate in Tables~\ref{results31}-\ref{results33b} the performance of our accelerated solvers based on formulations that involve the boundary integral operators $\mathcal{C}_k$, $\mathcal{A}_{k,S_{ik/2}}$, and  $\mathcal{B}_{k,2,k,0.4\mathcal{H}^{2/3}k^{1/3}}$. In all the numerical experiments presented in Tables~\ref{results31}-\ref{results33}--which correspond to strictly convex scatterers, we assumed $x$-polarized plane wave normal incidence. The numerical experiments related to the non-convex bean-shaped scatterer in Table~\ref{results33b} assumed the polarization
$(0,-\frac{1}{\sqrt{2}},\frac{1}{\sqrt{2}})$ and, in order for the
configuration to give rise to multiple reflections, a direction of
incidence
$(-\frac{\sqrt{3}}{3},-\frac{\sqrt{3}}{3},-\frac{\sqrt{3}}{3})$. As it can be seen from the results presented in Tables~\ref{results31}-\ref{results33b}, the numbers of GMRES iterations associated with solvers based on the operators $\mathcal{C}_k$ and $\mathcal{A}_{k,S_{ik/2}}$ grow with the frequency, which is in an accord with the theoretical predictions for spherical scatterers presented in Section~\ref{spec_prop_cfie}. In contrast, solvers based on the operators $\mathcal{B}_{k,2,k,0.4\mathcal{H}^{2/3}k^{1/3}}$ appear to require numbers of GMRES iterations that do not depend on the frequency for both convex and non-convex scatterers whose curvatures vary slowly. We note that in the high-frequency range presented in Tables~\ref{results31}-\ref{results33b}, the solvers based on the operators $\mathcal{B}_{k,2,k,0.4\mathcal{H}^{2/3}k^{1/3}}$ outperform the solvers based on the classical CFIE operators $\mathcal{C}_k$ in each numerical experiment performed, and the gains provided by the use of the former formulations over those using the latter formulations become more significant with increased problem size. Notably, these gains can be up to factors of $3.3$ for problems of $51$ wavelengths in electromagnetic size. Also, solvers based on the operators $\mathcal{A}_{k,S_{ik/2}}$ consistently outperform those based on the classical CFIE operators $\mathcal{C}_k$ in terms of computational times, and the gains can be up to factors of $2.6$ for high-frequencies. Furthermore, solvers based on the operators $\mathcal{B}_{k,2,k,0.4\mathcal{H}^{2/3}k^{1/3}}$ outperform those based on the operators $\mathcal{A}_{k,S_{ik/2}}$ in terms of computational times, in some cases for all frequencies considered (e.g. Table~\ref{results31}, Table~\ref{results33}, and Table~\ref{results33b}), and in the other case for higher frequencies (e.g. Table~\ref{results32}); these gains can be up to a factor of $2$ for the bean-shaped scatterer of size $51$ wavelengths. Given that the levels of accuracy reached by our solvers seems to be commensurate for all the formulations considered, formulations based on the operators $\mathcal{B}_{k,2,k,0.4\mathcal{H}^{2/3}k^{1/3}}$ appear to be very suitable for high-frequency simulations.

\begin{table}
\begin{center}
\resizebox{!}{1.6cm}{
\begin{tabular}{|c|c|c|c|c|c|c|c|}
\hline
$D$  & $N$ & \multicolumn{2}{c|}{$\mathcal{C}_k$~\eqref{eq:cIE}} & \multicolumn{2}{c|}{$\mathcal{A}_{k,S_{ik/2}}$~\eqref{eq:fIE}} & \multicolumn{2}{c|}{$\mathcal{B}_{k,2,k,0.4\mathcal{H}^{2/3}k^{1/3}}$~\eqref{eq:CFIERc}}\\
\cline{3-8}
& & It/ Total time & $\epsilon_\infty$ & It/ Total time & $\epsilon_\infty$ & It/ Total time & $\epsilon_\infty$\\
\hline
10.2 $\lambda$  & 47,628 & 37/0.38 h & 8.0 $\times$ $10^{-4}$& 16/0.25 h & 1.0 $\times$ $10^{-3}$ & 8/0.18 h & 1.0 $\times$ $10^{-3}$\\
\hline
20.4 $\lambda$ & 193,548 & 43/2.48 h & 2.5 $\times$ $10^{-3}$ & 21/1.89 h & 3.0 $\times$ $10^{-3}$ & 8/1.09 h & 1.1 $\times$ $10^{-3}$\\
\hline
30.6 $\lambda$& 437,772 & 48/7.71 h & 1.9 $\times$ $10^{-3}$ & 25/6.60 h & 2.7 $\times$ $10^{-3}$& 8/3.17 h & 2.4 $\times$ $10^{-3}$\\
\hline
40.8 $\lambda$& 780,300 & 54/17.6 h & 3.0 $\times$ $10^{-3}$ & 29/15.0 h & 3.8 $\times$ $10^{-3}$ & 8/6.34 h & 2.6 $\times$ $10^{-3}$\\
\hline
51.0 $\lambda$& 1,175,628 & 58/35.8 h & 4.0 $\times$ $10^{-3}$ & 32/33.4 h & 4.5 $\times$ $10^{-3}$ & 8/12.1 h & 3.1 $\times$ $10^{-3}$\\
\hline
\end{tabular}}
\caption{\label{results31}{ Performance of our solvers based on formulations that involve the boundary integral operators $\mathcal{C}_k$, $\mathcal{A}_{k,S_{ik/2}}$, and  $\mathcal{B}_{k,2,k,0.4\mathcal{H}^{2/3}k^{1/3}}$. Accelerated computations for spheres of diameters $D$, $x$-polarized plane wave normal incidence.}}
\end{center}
\end{table}

\begin{table}
\begin{center}
\resizebox{!}{1.6cm}{
\begin{tabular}{|c|c|c|c|c|c|c|}
\hline
$D$  & $N$ & \multicolumn{1}{c|}{$\mathcal{C}_k$~\eqref{eq:cIE}} & \multicolumn{2}{c|}{$\mathcal{A}_{k,S_{ik/2}}$~\eqref{eq:fIE}} & \multicolumn{2}{c|}{$\mathcal{B}_{k,2,k,0.4\mathcal{H}^{2/3}k^{1/3}}$~\eqref{eq:CFIERc}}\\
\cline{3-7}
& & It/ Total time & It/ Total time & $\epsilon_\infty$ & It/ Total time & $\epsilon_\infty$\\
\hline
10.2 $\lambda$  & 47,628 & 69/0.70 h & 14/0.21 h & 6.2 $\times$ $10^{-4}$ & 11/0.25 h & 7.6 $\times$ $10^{-4}$\\
\hline
20.4 $\lambda$ & 193,548 & 75/4.32 h & 15/1.35 h & 9.3 $\times$ $10^{-4}$ & 11/1.48 h & 1.0 $\times$ $10^{-3}$\\
\hline
30.6 $\lambda$& 437,772 & 80/12.8 h & 16/4.22 h & 6.7 $\times$ $10^{-4}$& 11/4.33 h & 7.7 $\times$ $10^{-4}$\\
\hline
40.8 $\lambda$& 780,300 & 86/28.0 h & 18/9.30 h & 8.6 $\times$ $10^{-4}$ & 11/8.72 h & 9.2 $\times$ $10^{-4}$\\
\hline
51.0 $\lambda$& 1,175,628 & 89/54.9 h & 20/20.8 h & 1.3 $\times$ $10^{-3}$ & 11/16.6 h & 1.1 $\times$ $10^{-3}$\\
\hline
\end{tabular}}
\caption{\label{results32}{Performance of our solvers based on formulations that involve the boundary integral operators $\mathcal{C}_k$, $\mathcal{A}_{k,S_{ik/2}}$, and  $\mathcal{B}_{k,2,k,0.4\mathcal{H}^{2/3}k^{1/3}}$. Accelerated computations for ellipsoids of size $D\times D/4\times D/4$, $x$-polarized plane wave normal incidence.}}
\end{center}
\end{table}

\begin{table}
\begin{center}
\resizebox{!}{1.6cm}{
\begin{tabular}{|c|c|c|c|c|c|c|}
\hline
$D$  & $N$ & \multicolumn{1}{c|}{$\mathcal{C}_k$~\eqref{eq:cIE}} & \multicolumn{2}{c|}{$\mathcal{A}_{k,S_{ik/2}}$~\eqref{eq:fIE}} & \multicolumn{2}{c|}{$\mathcal{B}_{k,2,k,0.4\mathcal{H}^{2/3}k^{1/3}}$~\eqref{eq:CFIERc}}\\
\cline{3-7}
& & It/ Total time & It/ Total time & $\epsilon_\infty$ & It/ Total time & $\epsilon_\infty$\\
\hline
10.2 $\lambda$  & 47,628 & 68/0.69 h & 19/0.29 h & 6.8 $\times$ $10^{-4}$ & 11/0.25 h & 1.1 $\times$ $10^{-3}$\\
\hline
20.4 $\lambda$ & 193,548 & 73/4.21 h & 20/1.80 h & 6.9 $\times$ $10^{-4}$ & 11/1.48 h & 9.5 $\times$ $10^{-4}$\\
\hline
30.6 $\lambda$& 437,772 & 76/12.16 h & 21/5.54 h & 6.6 $\times$ $10^{-4}$& 11/4.33 h & 7.0 $\times$ $10^{-4}$\\
\hline
40.8 $\lambda$& 780,300 & 78/25.39 h & 22/11.3 h & 7.5 $\times$ $10^{-4}$ & 10/7.93 h & 9.7 $\times$ $10^{-4}$\\
\hline
51.0 $\lambda$& 1,175,628 & 79/48.7 h & 23/24.0 h & 7.9 $\times$ $10^{-4}$ & 10/15.1 h & 8.8 $\times$ $10^{-4}$\\
\hline
\end{tabular}}
\caption{\label{results33}{Performance of our solvers based on formulations that involve the boundary integral operators $\mathcal{C}_k$, $\mathcal{A}_{k,S_{ik/2}}$, and  $\mathcal{B}_{k,2,k,0.4\mathcal{H}^{2/3}k^{1/3}}$. Accelerated computations for ellipsoids of size $D\times D/4\times D$, $x$-polarized plane wave normal incidence.}}
\end{center}
\end{table}

\begin{table}
\begin{center}
\resizebox{!}{1.6cm}{
\begin{tabular}{|c|c|c|c|c|c|c|}
\hline
$D$  & $N$ & \multicolumn{1}{c|}{$\mathcal{C}_k$~\eqref{eq:cIE}} & \multicolumn{2}{c|}{$\mathcal{A}_{k,S_{ik/2}}$~\eqref{eq:fIE}} & \multicolumn{2}{c|}{$\mathcal{B}_{k,2,k,0.4\mathcal{H}^{2/3}k^{1/3}}$~\eqref{eq:CFIERc}}\\
\cline{3-7}
& & It/ Total time & It/ Total time & $\epsilon_\infty$ & It/ Total time & $\epsilon_\infty$\\
\hline
10.2 $\lambda$  & 47,628 & 56/0.71 h & 17/0.32 h & 7.1 $\times$ $10^{-4}$ & 10/0.29 h & 5.0 $\times$ $10^{-4}$\\
\hline
20.4 $\lambda$ & 193,548 & 68/5.01 h & 23/2.50 h & 6.5 $\times$ $10^{-4}$ & 11/1.83 h & 3.7 $\times$ $10^{-4}$\\
\hline
30.6 $\lambda$& 437,772 & 74/15.7 h & 27/8.49 h & 5.6 $\times$ $10^{-4}$& 11/5.34 h & 3.0 $\times$ $10^{-4}$\\
\hline
40.8 $\lambda$& 780,300 & 82/35.2 h & 31/20.4 h & 4.9 $\times$ $10^{-4}$ & 11/10.8 h & 3.3 $\times$ $10^{-4}$\\
\hline
51.0 $\lambda$& 1,175,628 & 88/69.7 h & 35/43.4 h & 4.4 $\times$ $10^{-4}$ & 11/20.7 h & 3.1 $\times$ $10^{-4}$\\
\hline
\end{tabular}}
\caption{\label{results33b}{Performance of our solvers based on formulations that involve the boundary integral operators $\mathcal{C}_k$, $\mathcal{A}_{k,S_{ik/2}}$, and  $\mathcal{B}_{k,2,k,0.4\mathcal{H}^{2/3}k^{1/3}}$. Accelerated computations for bean-shaped scatterers of diameters $D$, under plane-wave incidence of direction $(-\frac{\sqrt{3}}{3},-\frac{\sqrt{3}}{3},-\frac{\sqrt{3}}{3})$ and polarization
$(0,-\frac{1}{\sqrt{2}},\frac{1}{\sqrt{2}})$.
}}
\end{center}
\end{table}

We illustrate in Table~\ref{results34} statistics of the memory requirements and costs of one matrix-vector product associated to solvers based on formulations that involve the boundary integral operators $\mathcal{C}_k$, $\mathcal{A}_{k,S_{ik/2}}$, and  $\mathcal{B}_{k,2,k,0.4\mathcal{H}^{2/3}k^{1/3}}$. As it can be seen, the cost of one matrix-vector product of all of these formulations is $\mathcal{O}(N^{4/3}\log{N})$. The computational times required by one matrix-vector product for the cases of the sphere and the ellipsoid are virtually the same given that the parametrizations of these scatterers was constructed identically. Two more conclusions can be drawn from the results in Table~\ref{results34}. The computational times required by the bean-shaped cases are slightly larger on account of our use of larger overlaps between the patches in the chart atlas. This is necessary in this case to achieve comparable levels of accuracy as for the sphere and ellipsoid configurations. A matrix-vector product resulted from
discretization of the formulations based on the operators $\mathcal{A}_{k,S_{ik/2}}$ is on
average at most $1.6$ more computationally expensive than the
matrix-vector product for the CFIE formulation based on the operators $\mathcal{C}_k$ at
the same level of discretization. A matrix-vector product resulted from
discretization of the formulations based on the operators $\mathcal{B}_{k,2,k,0.4\mathcal{H}^{2/3}k^{1/3}}$ is on
average at most $2.5$ more computationally expensive than the
matrix-vector product for the CFIE formulation based on the operators $\mathcal{C}_k$ at
the same level of discretization.

\begin{table}
\begin{center}
\resizebox{!}{1.8cm}{
\begin{tabular}{|c|c|c|c|c|c|c|c|}
\hline
$D$  & $N$ & \multicolumn{2}{c|}{$\mathcal{C}_k$~\eqref{eq:cIE}} & \multicolumn{2}{c|}{$\mathcal{A}_{k,S_{ik/2}}$~\eqref{eq:fIE}} & \multicolumn{2}{c|}{$\mathcal{B}_{k,2,k,0.4\mathcal{H}^{2/3}k^{1/3}}$~\eqref{eq:CFIERc}}\\
\cline{3-8}
& & Mem & Time/It & Mem & Time/It &  Mem & Time/It \\
\hline
10.2 $\lambda$  & 47,628 & 0.1 Gb & 0.61 m  & 0.16 Gb & 0.93 m & 0.22 Gb & 1.38 m \\
\hline
20.4 $\lambda$ & 193,548 & 0.4 Gb & 3.46 m & 0.64 Gb & 5.40 m & 0.88 Gb & 8.10 m \\
\hline
30.6 $\lambda$& 437,772 & 0.9 Gb & 9.63 m & 1.44 Gb & 15.84 m & 1.98 Gb & 23.64 m \\
\hline
40.8 $\lambda$& 780,300 & 1.6 Gb & 19.55 m & 2.5 Gb & 31.03 m & 3.5 Gb & 47.58 m \\
\hline
51.0 $\lambda$& 1,175,628 & 2.5 Gb & 37.03 m & 3.9 Gb & 62.62 m & 5.5 Gb & 90.60 m \\
\hline
\end{tabular}}
\caption{\label{results34}{ Statistics of the memory requirements and costs of one matrix-vector product associated to solvers based on formulations that involve the boundary integral operators $\mathcal{C}_k$, $\mathcal{A}_{k,S_{ik/2}}$, and  $\mathcal{B}_{k,2,k,0.4\mathcal{H}^{2/3}k^{1/3}}$. The cost of one matrix-vector product of all of these formulations is $\mathcal{O}(N^{4/3}\log{N})$.}}
\end{center}
\end{table}

We conclude the numerical experiments with an illustration in Figure~\ref{fig:iteration_counts} of the iteration counts required by our solvers based on the integral operators $\mathcal{A}_{k,S_{ik/2}}$ and $\mathcal{B}_{k,2,k,0.4\mathcal{H}^{2/3}k^{1/3}}$ respectively to reach GMRES residuals of $10^{-4}$ for two scatterers, namely a unit sphere and an ellipsoid with principal axes $2$, $0.5$, and $2$, and $121$ wavenumbers $k=8,9,\ldots,127,128$. The corresponding electromagnetic sizes of the scattering problems range from $2.5$ to $40.8$ wavelengths. As it can be seen, the number of GMRES iterations required by our solvers based on the operators $\mathcal{B}_{k,2,k,0.4\mathcal{H}^{2/3}k^{1/3}}$ are independent of frequency. In the case of the spherical scatterers and bean-shaped scatterers, the runtimes of our solvers based on the operators $\mathcal{B}_{k,2,k,0.4\mathcal{H}^{2/3}k^{1/3}}$ are always smaller than those based on the operators $\mathcal{A}_{k,S_{ik/2}}$. For ellipsoid scatterers, the runtimes of our solvers based on the operators $\mathcal{B}_{k,2,k,0.4\mathcal{H}^{2/3}k^{1/3}}$ are smaller than those based on the operators $\mathcal{A}_{k,S_{ik/2}}$ for scattering problems corresponding to wavenumbers $k$ that are greater than equal to the value $k_c=98$.

\begin{figure}
\centering
\includegraphics[height=85mm]{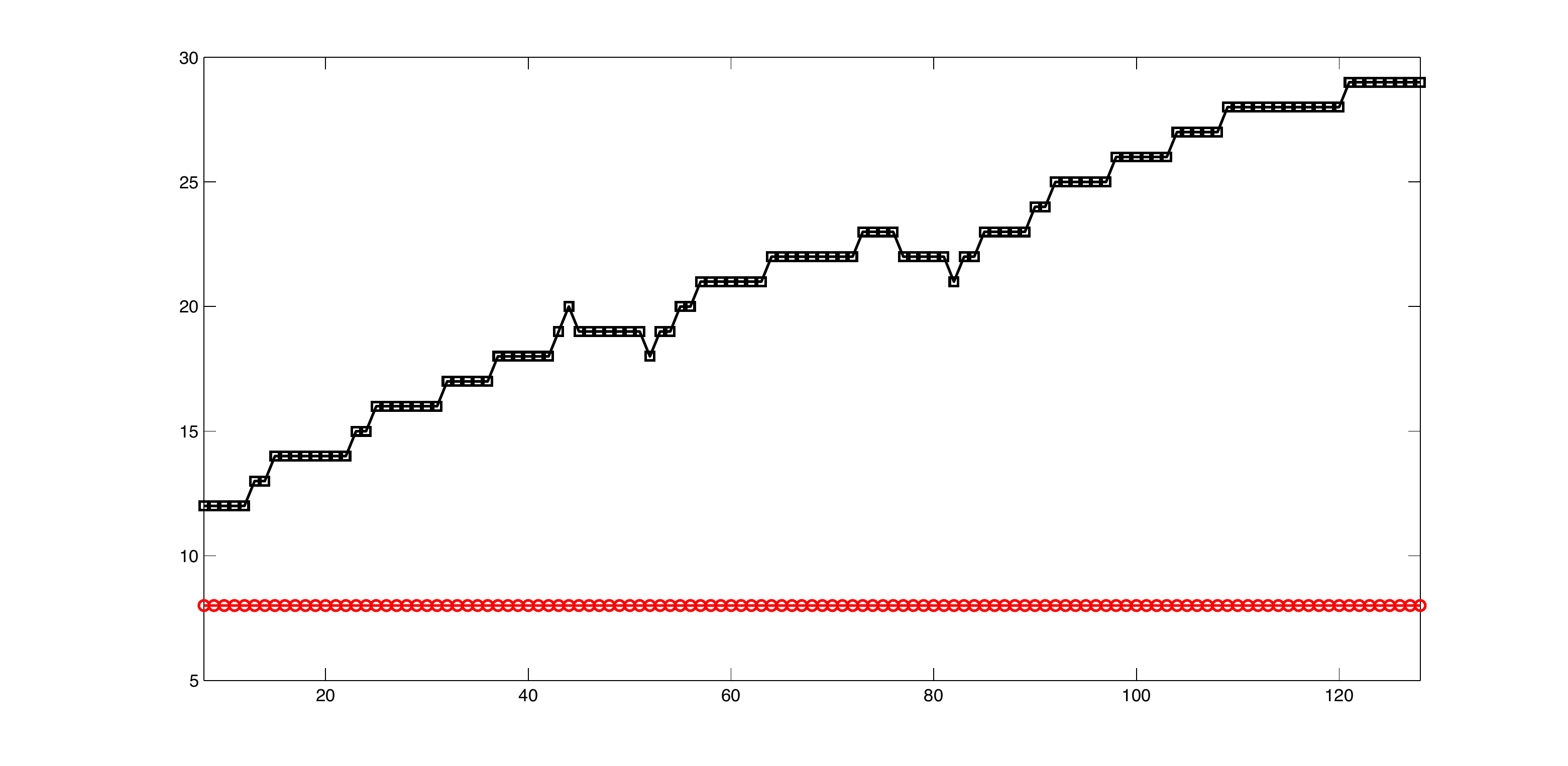}\\
\includegraphics[height=85mm]{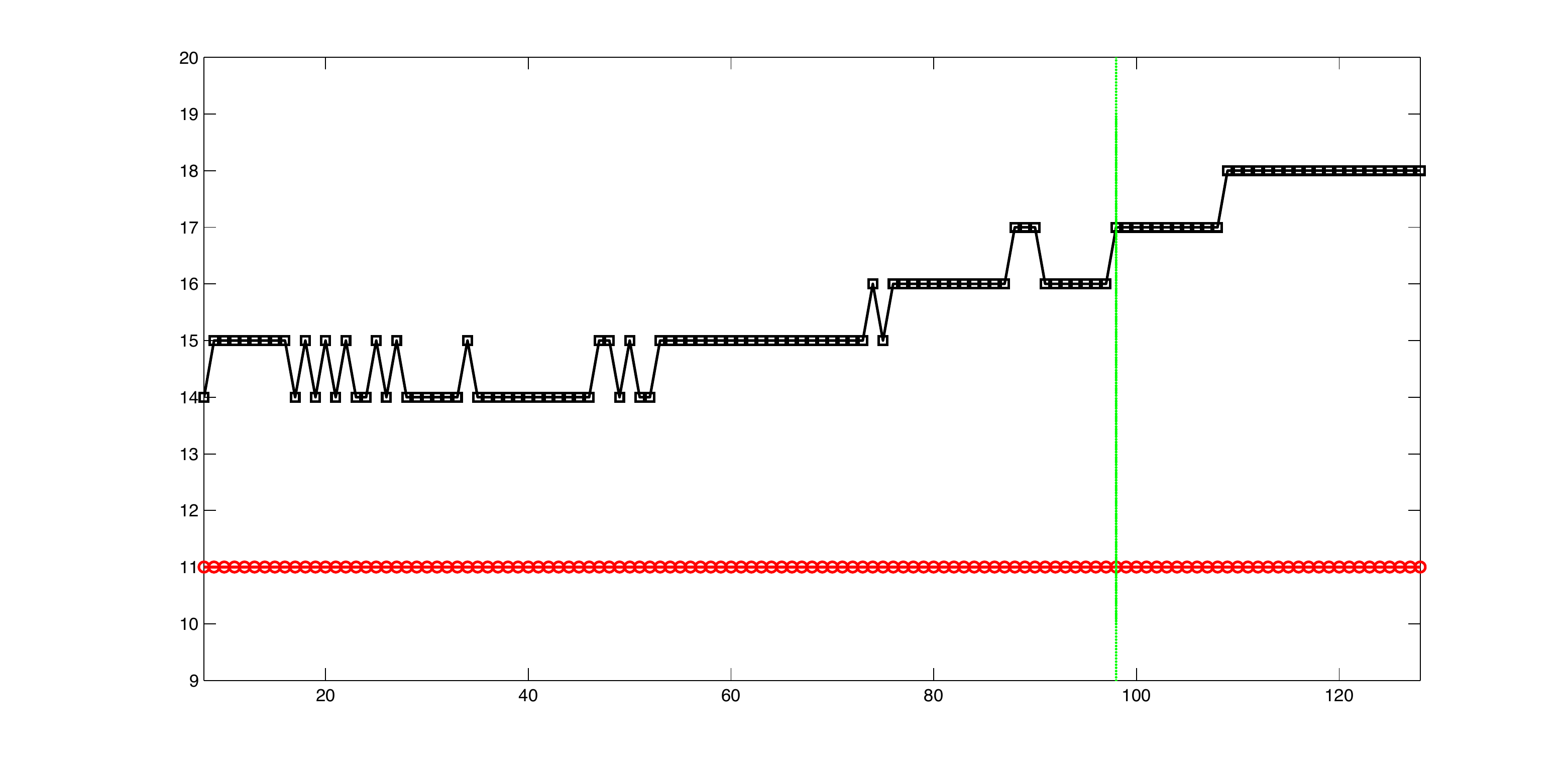}
\caption{Plot of the numbers of iterations required by our solvers based on the operators $\mathcal{A}_{k,S_{ik/2}}$ (black)  and $\mathcal{B}_{k,2,k,0.4\mathcal{H}^{2/3}k^{1/3}}$ (red)  as a function of the wavenumber $k$ for 121 values of $k$ ranging from $k=8$ to $k=128$ for a unit sphere scatterer (top) and an ellipsoid of of principal axes $2$, $0.5$, and $2$ (bottom). The vertical green line in the figure on the bottom represent the cut-off frequency $k_c=98$: for values of the wavenumber larger than $k_c$ our solvers based on the operators $\mathcal{B}_{k,2,k,0.4\mathcal{H}^{2/3}k^{1/3}}$ require less computational times than our solvers based on the operators $\mathcal{A}_{k,S_{ik/2}}$.}
\label{fig:iteration_counts}
\end{figure}

We summarize the results presented in this section as follows:
\begin{enumerate}
\item the numbers of iterations required by solvers based on the CFIE operators $\mathcal{C}_k$ defined in equation~\eqref{eq:cIE} and the Calder\'on-Ikawa CFIER operators $\mathcal{A}_{k,S_{ik/2}}$ defined in equation~\eqref{eq:fIE} depend on the wavenumber $k$;
\item the numbers of iterations required by solvers based on the Calder\'on-Complex CFIER operators $\mathcal{B}_{k,2,k,0.4\mathcal{H}^{2/3}k^{1/3}}$ defined in equation~\eqref{eq:CFIERc} are small and independent of the wavenumber $k$;
\item although due to the more complex nature of their underlying operators, matrix-vector products associated with the  Calder\'on-Complex CFIER operators $\mathcal{B}_{k,2,k,0.4\mathcal{H}^{2/3}k^{1/3}}$ are more expensive than those associated to the CFIE operators $\mathcal{C}_k$ and  Calder\'on-Ikawa CFIER operators $\mathcal{A}_{k,S_{ik/2}}$, solvers based on the former formulations give rise to important computational gains over solvers based on the latter formulations;
\item for a given wavenumber $k$, the numbers of iterations required by solvers based on both CFIER operators $\mathcal{A}_{k,S_{ik/2}}$ and $\mathcal{B}_{k,2,k,0.4\mathcal{H}^{2/3}k^{1/3}}$ do not depend on the discretization size; this is in contrast to the case of solvers based on the classical CFIE operators $\mathcal{C}_k$.
\end{enumerate}

We conclude this section with a brief comparison between the Calder\'on-Complex CFIER formulations showcased in this work and other regularized formulations based on approximations of the DtN operators proposed in the literature
\begin{enumerate}
\item the Calder\'on-Complex CFIER operators $\mathcal{B}_{k,2,k,0.4\mathcal{H}^{2/3}k^{1/3}}$ were shown to have excellent properties in the high-frequency regime for scatterers with slowly varying curvatures; a Galerkin discretization of these seems possible along the lines in~\cite{Andriulli,Andriulli1}. Although the evaluation of the regularizers $\mathcal{T}_{k+i\kappa_2}$ is somewhat expensive, it has the advantage of being not that different from that of the EFIE operators $\mathcal{T}_k$. The operator composition $\mathcal{T}_k\ \mathcal{T}_{k+i\kappa_2}$ can be evaluated without a fuss in the context of Nystr\"om discretization and it also can be evaluated in the context of Galerkin discretizations, albeit there are some subtleties associated with the latter case~\cite{Andriulli,Andriulli1};
\item another choice of regularizing operators was given in~\cite{BorelLevadouxAlouges,AlougesLevadoux} by $\mathcal{R}=-2\sum_j \chi_j\mathcal{T}_k\chi_j$ where $\chi_j$ are smooth and compactly supported functions such as $\sum_j \chi_j^2=1$ on $\Gamma$; the ensuing formulations do not appear to lead to iteration counts that are independent of frequency~\cite{BorelLevadouxAlouges,AlougesLevadoux}, and the evaluation of that regularizer is more expensive than the evaluation of the regularizer $\mathcal{R}=-2\mathcal{T}_{k+i\kappa_2}$ that we used in this text;
\item in the very recent contribution~\cite{Boujaji}, Galerkin discretizations of the operators $PS\mathcal{B}_{k,2,k,0.4k^{1/3}}$ using loop and star basis functions were introduced for the solution of PEC scattering problems. The results presented in that contribution show that the iteration counts of these formulations are independent of frequency in the low to medium frequency range, even for Lipschitz domains. Another recent contribution~\cite{LevadouxN} uses the same type of principal value regularizing operators for the solution of time-harmonic Maxwell equations with impedance boundary conditions via boundary integral equations. The use of Fast Multipoles in conjunction with Galerkin discretization of the regularized formulations enables the solution of larger size problems (up to $10$ wavelength) in~\cite{LevadouxN}. The regularized formulations are shown to lead to numbers of iterations that are virtually independent of frequency for a larger range of frequencies. One potential drawback of these formulations is their need to invert surface vector Laplace-Beltrami operators which prevents the use of non-conforming meshes~\cite{Bendali1}.
\end{enumerate}

\section{Conclusions}

We presented several versions of Regularized Combined Field Integral
Equations formulations for the solution of electromagnetic scattering equations with PEC boundary conditions. These formulations consist of combined field representations where the electric field integral operators act on certain regularizing operators. The construction of the regularizing operators, in turn, is based on Calder\'on's calculus and complexification techniques. These
integral equation formulations are well conditioned on account of the
choice of the regularizing operators. Some of the resulting integral operators possess excellent
eigenvalues clustering properties, which translate in the case of redsmooth scatterers with slowly varying  curvatures into 
very small numbers of iterations necessary to obtain the solution of the ensuing
linear systems, regardless of the discretization size and the frequency of the scattering problems. In addition, the
far field errors incurred by our implementation are of the same order
as those resulting from the classical Combined Field Integral Equation
counterparts. Thus, the important gains in computational complexity make the Regularized Combined Field
Integral Equations discussed in this text a viable method of solution to PEC 
scattering problems.

\section*{Acknowledgments}
Yassine Boubendir gratefully acknowledge support from NSF through contract DMS-1016405 and DMS-1319720. Catalin Turc gratefully acknowledge support from NSF through contract DMS-1312169. 

\section{Appendix}

We present in this section a proof of the technical Lemma~\ref{ImKm} 
\begin{lemma}
There exist constants $C_j>0,j=1,\ldots,4$ and a number $\tilde{k}_0>0$ such that
$$ (i)\ C_1 k^{-2}(n^2+k^2)^{1/2}\leq iJ'_{n+1/2}(ik)(H_{n+1/2}^{(1)})'(ik)\leq C_2k^{-2}(n^2+k^2)^{1/2}$$
$$ (ii)\ \frac{1}{4}(n^2+k^2)^{-1/2}\leq -S_n^{(1)}(ik)\leq C_3(n^2+k^2)^{-1/2}+C_4k^{-2},\ $$
$$ (iii) |J'_{n+1/2}(ik)H_{n+1/2}^{(1)}(ik)|\leq C_4k^{-1},\ |J_{n+1/2}(ik)(H_{n+1/2}^{(1)})'(ik)|\leq C_4k^{-1}$$ 
for all $k>\tilde{k}_0$ and all $n\geq 0$.
\end{lemma}
\begin{proof}
(i) We begin by using the representation of the functions $J_{n+1/2}(ik)$ and $H_{n+1/2}^{(1)}(ik)$ in terms of the Bessel and Hankel functions of the third kind
$$iJ'_{n+1/2}(ik)(H_{n+1/2}^{(1)})'(ik)=-\frac{2}{\pi}I'_{n+1/2}(k)K'_{n+1/2}(k)$$
 where $I'_\nu(k)>0$ and $K'_\nu(k)<0$ for $\nu\geq 0$. Using the uniform asymptotic expansions as $\nu\rightarrow\infty$ (Formulas 9.7.9 and 9.7.10 in~\cite{Abramowitz})
\begin{eqnarray}\label{eq:estimateK}
I'_\nu(\nu z)&\sim&\frac{1}{\sqrt{2\pi\nu}}\frac{(1+z^2)^\frac{1}{4}}{z}e^{\nu\mu}(1+\mathcal{O}(\nu^{-1}))\nonumber\\
K'_\nu(\nu z)&\sim&-\frac{\sqrt{\pi}}{\sqrt{2\nu}}\frac{(1+z^2)^\frac{1}{4}}{z}e^{-\nu\mu}(1+\mathcal{O}(\nu^{-1}))
\end{eqnarray}
where $\mu=\sqrt{1+z^2}+\ln\frac{z}{1+\sqrt{1+z^2}}$,  we get that there exists a constant $N_0>0$ such that
\begin{equation}
\frac{1}{4} k^{-2}((n+1/2)^2+k^2)^{1/2}\leq -I'_{n+1/2}(k)K'_{n+1/2}(k)\leq k^{-2}((n+1/2)^2+k^2)^{1/2},\nonumber
\end{equation}
for all $n>N_0,\ k>0$. Given that we can choose the constant $N_0$ above to further satify $(n+1/2)^2+k^2<2(n^2+k^2)$ for all $n>N_0$ and all $k>0$, we obtain the following estimate 
\begin{equation}\label{eq:bound_eig_ik_00}
\frac{1}{4} k^{-2}(n^2+k^2)^{1/2}\leq -I'_{n+1/2}(k)K'_{n+1/2}(k)\leq \sqrt{2}k^{-2}(n^2+k^2)^{1/2},\ {\rm for\ all}\ n>N_0,\ k>0.
\end{equation}
 Using the asymptotic expansion (Formula 9.7.6 in~\cite{Abramowitz}) which is valid for $\nu$ fixed and $z\to\infty$
\begin{equation}
-I'_\nu(z)K'_\nu(z)\sim \frac{1}{2z}\left(1+\frac{1}{2}\frac{\mu-3}{(2z)^2}+\mathcal{O}\left(\frac{\mu^2}{z^4}\right)\right),\ \mu=4\nu^2
\end{equation}
we get that for a fixed $n$ there exists $k_n'$ such that
\begin{equation}\label{eq:bound_eig_ik_1}
\frac{1}{4k}\leq -I'_{n+1/2}(k)K'_{n+1/2}(k)\leq \frac{1}{k},\ {\rm for\ all}\ k\geq k'_n.
\end{equation}
It follows from the estimate~\eqref{eq:bound_eig_ik_1} that there exist $k_1=\max_{n\leq N_0}k'_n$ and  $c=(1+N_0^2/k_1^2)^{-\frac{1}{2}}$ such that
\begin{equation}\label{eq:bound_eig_ik_2}
\frac{1}{4} ck^{-2}(n^2+k^2)^{1/2}\leq \frac{1}{4k}\leq -I'_{n+1/2}(k)K'_{n+1/2}(k)\leq \frac{1}{k}\leq \sqrt{2}k^{-2}(n^2+k^2)^{1/2},\ 0\leq n\leq N_0,\ k\geq k_1.
\end{equation}
We obtain by combining estimates~\eqref{eq:bound_eig_ik_00} and~\eqref{eq:bound_eig_ik_2} that there exist constants $C_1=\frac{1}{2\pi}\max(1,(1+N_0^2/k_1^2)^{-\frac{1}{2}})$ and $C_2=\frac{2\sqrt{2}}{\pi}$ such that
\begin{equation}\label{eq:bound_eig_ik}
C_1 k^{-2}(n^2+k^2)^{1/2}\leq iJ'_{n+1/2}(ik)(H_{n+1/2}^{(1)})'(ik)=-\frac{2}{\pi}I'_{n+1/2}(k)K'_{n+1/2}(k)\leq C_2 k^{-2}(n^2+k^2)^{1/2}
\end{equation}
for all $n\geq 0$ and all $k\geq k_1$.\\

(ii) In order to prove the estimate concerning $S_n^{(1)}(ik)$, we use the identities (formulas 9.1.27 in~\cite{Abramowitz}) 
$$J'_{n+1/2}(K)=\frac{m+1/2}{K}J_{n+1/2}(K)-J_{n+3/2}(K),\ (H^{(1)}_{n+1/2})'(K)=\frac{n+1/2}{K}H_{n+1/2}^{(1)}(K)-H_{n+3/2}^{(1)}(K)$$to derive
\begin{eqnarray}
-S_n^{(1)}(ik)&=&\frac{\pi}{2ik^2}[(2n^2+3n+1)J_{n+1/2}(ik)H^{(1)}_{n+1/2}(ik)-k^2J_{n+3/2}(ik)H^{(1)}_{n+3/2}(ik)\nonumber\\
&-&ik(n+1)(J_{n+1/2}(ik)H^{(1)}_{n+3/2}(ik)+J_{n+3/2}(ik)H^{(1)}_{n+1/2}(ik))].\nonumber
\end{eqnarray}
We use Bessel and Hankel functions of the third kind $K_{n+1/2}$ and $I_{n+1/2}$ together with the Wronskian identity $I_\nu(k)K_{\nu+1}(k)+I_{\nu+1}(k)K_\nu(k)=\frac{1}{k}$ (formula 9.6.15 in~\cite{Abramowitz}) to reexpress the previous identity in the form
\begin{eqnarray}
-S_n^{(1)}(ik)&=&-\frac{1}{k^2}[(2n^2+3n+1)I_{n+1/2}(k)K_{n+1/2}(k)-k^2I_{n+3/2}(k)K_{n+3/2}(k)\nonumber\\
&+&k(n+1)(2I_{n+3/2}(k)K_{n+1/2}(k)-1/k)].\nonumber
\end{eqnarray}
Using the fact that $I'_{n+1/2}(k)=\frac{n+1/2}{k}I_{n+1/2}(k)+I_{n+3/2}(k)$ ( (formulas 9.6.26 in~\cite{Abramowitz})) we get
\begin{equation}\label{eq:final_form_eq_S1}
-S_n^{(1)}(ik)=\frac{n+1}{k^2}(2kI'_{n+1/2}(k)K_{n+1/2}(k)-1)+I_{n+3/2}(k)K_{n+3/2}(k).
\end{equation}
Using the uniform asymptotic expansions as $\nu\rightarrow\infty$ (Formulas 9.7.8 and 9.7.9 in~\cite{Abramowitz})
\begin{eqnarray}\label{eq:estimateK1}
K_\nu(\nu z)&\sim&\frac{\sqrt{\pi}}{\sqrt{2\nu}}\frac{e^{-\nu\mu}}{(1+z^2)^\frac{1}{4}}(1+\mathcal{O}(\nu^{-1}))\nonumber\\
I'_\nu(\nu z)&\sim&\frac{1}{\sqrt{2\pi\nu}}\frac{(1+z^2)^\frac{1}{4}}{z}e^{\nu\mu}(1+\mathcal{O}(\nu^{-1}))
\end{eqnarray}
where $\mu=\sqrt{1+z^2}+\ln\frac{z}{1+\sqrt{1+z^2}}$,  we get that there exist constants $N_1>0$ and $c_1>0$ such that
\begin{equation}\label{eq:bound_eig_ik_0}
\frac{1}{2} c_1k^{-2}\leq \frac{n+1}{k^2}(2kI'_{n+1/2}(k)K_{n+1/2}(k)-1)\leq c_1k^{-2},\ {\rm for\ all}\ n>N_1,\ k>0.
\end{equation}
 Using the asymptotic expansions (Formula 9.7.2 and 9.7.3 in~\cite{Abramowitz}) which are valid for $\nu$ fixed and $z\to\infty$
\begin{eqnarray}\label{eq:estimateK1fixedm}
K_\nu(z)&\sim&\frac{\sqrt{\pi}}{\sqrt{2z}}e^{-z}\left(1+\frac{\mu-1}{8z}+\mathcal{O}\left(\frac{\mu^2}{z^2}\right)\right)\nonumber\\
I'_\nu(z)&\sim&\frac{1}{\sqrt{2\pi z}}e^{z}\left(1-\frac{\mu+3}{8z}+\mathcal{O}\left(\frac{\mu^2}{z^2}\right)\right),\ \mu=4\nu^2
\end{eqnarray}
we get that for a fixed $n\geq 0$ there exists $\tilde{k}_n$ such that
\begin{equation}\label{eq:bound_eig_ik_1n}
\frac{n+1}{4k^3}\leq \frac{m+1}{k^2}(2kI'_{n+1/2}(k)K_{n+1/2}(k)-1)\leq \frac{n+1}{k^3},\ {\rm for\ all}\ k\geq \tilde{k}_n.
\end{equation}
We conclude that there exists a constant $\tilde{c}=\max(c_1,\frac{N_1+1}{k_2})$ where $\tilde{k}_2=\max_{0\leq n \leq N_1}\tilde{k}_n$, such that
\begin{equation}\label{eq:bound_eig_ik_1n}
0\leq \frac{n+1}{k^2}(2kI'_{n+1/2}(k)K_{n+1/2}(k)-1)\leq \frac{\tilde{c}}{k^2},\ {\rm for\ all}\ k\geq \tilde{k}_2,\ {\rm for\ all}\ n\geq 0.
\end{equation}
Furthermore, if we use the following result established in Lemma 3.1 in~\cite{turc7}, namely there exists a constant $\tilde{C}_1$ and a number $\hat{k}_2>0$ such that
$$\frac{1}{4}(n^2+k^2)^{-1/2}\leq I_{n+3/2}(k)K_{n+3/2}(k) \leq \tilde{C}_1(n^2+k^2)^{-1/2} \ {\rm for\ all}\ k\geq \hat{k}_2,\ {\rm for\ all}\ n\geq 0$$
together with the estimate established in equation~\eqref{eq:bound_eig_ik_1n}, we obtain if we take into account equation~\eqref{eq:final_form_eq_S1} that there exists a constant $C_3$ and a number $k_2=\max\{\tilde{k}_2,\hat{k}_2\}$ such that
$$\frac{1}{4}(n^2+k^2)^{-1/2}\leq -S_n^{(1)}(ik)\leq C_3(n^2+k^2)^{-1/2}+C_4k^{-2} \ {\rm for\ all}\ k\geq k_2,\ {\rm for\ all}\ n\geq 0.$$

(iii) We get from the asymptotic expansions~\eqref{eq:estimateK1} that there exist constants $N_1>0$ and $c_1>0$ such that
\begin{equation}\label{eq:bound_eig_ik_0}
\frac{1}{4} k^{-1}\leq I'_{n+1/2}(k)K_{n+1/2}(k)\leq k^{-1},\ {\rm for\ all}\ n>N_1,\ k>0.
\end{equation}
On the other hand, using the asymptotic expansions~\eqref{eq:estimateK1fixedm} we get that for a fixed $n\geq 0$ there exists $\tilde{k}_n$ such that
\begin{equation}\label{eq:bound_eig_ik_1n}
\frac{1}{4}k^{-1}\leq I'_{n+1/2}(k)K_{n+1/2}(k)\leq k^{-1},\ {\rm for\ all}\ k\geq \tilde{k}_n.
\end{equation}
If we let $k_3=\max_{0\leq n\leq N_1}\tilde{k}_n$ we obtain that
\begin{equation}
|J'_{n+1/2}(ik)H_{n+1/2}^{(1)}(ik)|=\frac{2}{\pi}I'_{n+1/2}(k)K_{n+1/2}(k)\leq \frac{2}{\pi} k^{-1},\ {\rm for\ all}\ k\geq k_3,\ 0\leq n.
\end{equation}
The estimate concerning $|J_{n+1/2}(ik)(H_{n+1/2}^{(1)})'(ik)|$ follows form the previous estimate and the Wronskian identity (formula 9.6.15 in~\cite{Abramowitz}). Finally, if we take $\tilde{k}_0=\max\{k_1,k_2,k_3\}$, the result of the Lemma follows.
\end{proof}

\end{document}